\documentclass[oneside,english,reqno]{amsart}
\usepackage[T1]{fontenc}
\usepackage[latin9]{inputenc}
\usepackage{geometry}
\usepackage{verbatim}
\usepackage{float}
\usepackage{amstext}
\usepackage{amsthm}
\usepackage{amssymb}
\usepackage{mathdots}
\usepackage{stmaryrd}
\usepackage{mathtools}
\usepackage{url}
\usepackage{wrapfig,amsmath}
\usepackage{hyperref}
\usepackage{graphicx}
\usepackage{caption,subcaption}

\usepackage{thmtools}
\usepackage{thm-restate}

\makeatletter

\floatstyle{ruled}
\newfloat{algorithm}{tbp}{loa}
\providecommand{\algorithmname}{Algorithm}
\floatname{algorithm}{\protect\algorithmname}

\numberwithin{equation}{section}
\numberwithin{figure}{section}
\theoremstyle{plain}
\newtheorem{thm}{Theorem}
  \theoremstyle{plain}
  \newtheorem{lem}{Lemma}
  \theoremstyle{definition}
  
\newtheorem{remark}{Remark}

\newtheorem{proposition}{Proposition}
\allowdisplaybreaks
\usepackage{bbm}
\usepackage{algorithmic}
\usepackage{stmaryrd}

\makeatother

\usepackage{babel}

\def \a {\mathbf a}

\def \x {\mathbf x}

\def \y {\mathbf y}

\def \m {\mathbf m}
\def \n {\mathbf n}

\newcommand{\C}{\ensuremath{\mathbbm{C}}}

\begin{document}

\title[Phase Retrieval from Spectrogram Measurements]{Inverting Spectrogram Measurements via Aliased Wigner Distribution Deconvolution and Angular Synchronization}

\author{Michael Perlmutter, Sami Merhi, Aditya Viswanathan, Mark Iwen}

\begin{abstract}We propose a two-step approach for reconstructing a signal $\x\in\mathbb{C}^d$ from 
subsampled short-time Fourier transform
magnitude (spectogram) measurements: 
First, we use an aliased  Wigner distribution deconvolution approach to solve for a portion of the rank-one matrix ${\bf \widehat{{\bf x}}}{\bf \widehat{{\bf x}}}^{*}.$ Second, we use angular syncrhonization to solve for ${\bf \widehat{{\bf x}}}$ (and then for ${\bf x}$ by Fourier inversion).
Using this  method,  we produce two new  efficient phase retrieval algorithms that perform well numerically in comparison to standard approaches and also prove two theorems, one which  guarantees the recovery of discrete, bandlimited  signals ${\bf x}\in\mathbb{C}^{d}$ 
 from fewer than $d$ STFT magnitude measurements and another which establishes a new class of deterministic coded diffraction pattern measurements which are guaranteed to allow efficient and noise robust recovery.
\end{abstract}

\maketitle

\section{Introduction}\label{sec:intro}

The {\it phase retrieval problem}, i.e., reconstructing a signal from phaseless measurements, is at the core of many scientific breakthroughs related to the imaging of cells \cite{van2015imaging}, viruses \cite{seibert2011single}, and nanocrystals \cite{clark2013ultrafast}, 
%
and also advances in crystallographic imaging \cite{Harrison1993}, 
optics \cite{walther1963question}, astronomy \cite{Fienup1987}, quantum
mechanics \cite{Corbett2006}, and speech signal processing
\cite{balan2006signal,griffin1984signal}. 
As a result, many sophisticated algorithms, which achieve great empircal success,  have been developed for solving this problem in applications throughout science and engineering  (see \cite{Fienup_1978_AltProj,GSaxtonAltProj,griffin1984signal} for widely used examples).  Motivated by the  success of these methods, the mathematical community has recently began to study  the challenging problem of designing measurement masks and corresponding reconstruction algorithms with rigorous convergence guarantees and noise robustness properties (see, e.g., the work of Balan, Cand{\`e}s, Strohmer, and others \cite{alexeev2014phase, balan2006signal,candes2013phaselift,  gross2015improved}).  In this paper, we aim to extend the mathematical analysis  of phaseless measurement maps and noise-robust reconstruction algorithms to include a broad class of phaseless Gabor measurements such as those that are utilized in, e.g., ptychographic imaging \cite{Chapman1996,daSilva:15,Rodenburg2008,Rodenburg1992}.

Specifically,  we will develop and analyze several algorithms for recovering (up to a global phase) a signal $\x \in \C^d$  from the magnitudes of its inner products with shifts of  masks that are {\em locally supported} in either physical space or Fourier space.  The local support of these masks in physical space corresponds to the use of concentrated beams  in ptychographic imaging  to measure small portions of a large sample, whereas the local support of these masks in Fourier space simulates the recovery of samples belonging to a special class of deterministic coded diffraction patterns (CDP).  

Following \cite{iwen2018phase,iwen2016fast}, 
 we will  assume that we have a family of measurement masks, or windows, $\mathbf{m_0},\mathbf{m_1},\ldots,$$\mathbf{m_{K-1}} \in \mathbbm{C}^d$ such that for all $k,$ the nonzero entries of either $\mathbf{m_k}$ or $\mathbf{\widehat{m}_k}$ are contained in the set $[\delta]_0$ for some fixed $\delta< \frac{d}{2},$ where  
for any integer $n\geq 0,$ we let 
\begin{equation*}
[n]_0=\{0,1,\ldots,n-1\}
\end{equation*}
denote the set of the first $n$ nonnegative integers. Let $L$ be an integer which divides $d,$ and let  $Y': \mathbbm{C}^d \rightarrow [0,\infty)^{K\times L}$ be the matrix-valued measurement map defined by its coordinate functions
\begin{equation}
Y'_{k,\ell}\coloneqq Y'_{k,\ell}(\mathbf{x}) \coloneqq |\langle S_{\ell a}\mathbf{m_k}, \mathbf{x}\rangle|^2 + N'_{k,\ell},
\label{equ:Zmap}
\end{equation}   
for $k\in [K]_0$ and $\ell\in[L]_0,$  where  $a\coloneqq\frac{d}{L},$  $S_\ell$ is the circular shift operator on $\mathbbm{C}^d$ defined  for $\y \in \C^d$ and $\ell\in\mathbbm{Z}$ by
\begin{equation}\label{eq:defshift}
(S_\ell\mathbf{y})_j \coloneqq y_{\left((j+\ell) \!\!\!\!\mod d\right)},
\end{equation} 
and $N'=\left(N'_{k,\ell}\right)_{k\in[K]_0,\ell\in[L]_0} \in \mathbbm{R}^{K\times L}$ represents an arbitrary perturbation due to, e.g., measurement noise or imperfect knowledge of the masks $\mathbf{m_k}.$


Our goal is to reconstruct $\mathbf{x}$ from these measurements. It is clear that $Y'(\mathbf{x})=Y'(\mathbbm{e}^{\mathbbm{i}\phi}\mathbf{x})$ for all $\phi\in\mathbb{R},$ so at best we can hope to reconstruct $\mathbf{x}$ up to a global phase, i.e., up to the equivalence relation 
\begin{equation*}
\mathbf{x}\sim\mathbf{x'}\text{ if } \mathbf{x}=\mathbbm{e}^{\mathbbm{i}\phi}\mathbf{x'} \text{ for some }\phi\in\mathbb{R}.  
\end{equation*}
Algorithms 1 and 2, presented in Section \ref{sec:recGuar}, will accomplish this goal in the special case where the masks $\mathbf{m_k}$ are obtained by modulating a single mask $\mathbf{m}.$ Specifically, we let $\m\in\mathbb{C}^d,$ and for $k\in[d]_0$ we let  \begin{equation}\label{eq:fouriermask}
\mathbf{m_k}=W_{k}\mathbf{m}
\end{equation}
where  $W_{k}$ is the modulation operator given by 
\begin{equation}\label{eq:defmod}
(W_{k} \mathbf{m})_j \coloneqq   \mathbbm{e}^{\frac{2 \pi \mathbbm{i} j k}{d}}m_j.
\end{equation} 
As we will see, assuming that our masks have this form will allow us to recover $\mathbf{x},$ even when the shift size $a$ is strictly greater than one. Towards this end, we let  $Y: \mathbbm{C}^d \rightarrow [0,\infty)^{d\times d}$ be the matrix-valued measurement map defined by its coordinate functions\begin{equation}
Y_{k,\ell}\coloneqq Y_{k,\ell}(\mathbf{x}) \coloneqq |\langle S_{\ell}W_k\m, \mathbf{x}\rangle|^2 + N_{k,\ell}
\label{equ:Ymap},
\end{equation}   
where analogously to \eqref{equ:Zmap}, $N=(N_{k,\ell})_{0\leq k,\ell\leq d-1}$  represents an arbirtrary perturbation. We note that $Y$ is the special case of $Y'$ where $K=L=d$ and the masks $\mathbf{m_k}$ have the form \eqref{eq:fouriermask}. For positve integers, $K$ and $L$ which divide $d,$ we  let \begin{equation*}
\frac{d}{K}\left[K\right]_{0}\coloneqq\left\{ 0,\frac{d}{K},\frac{2d}{K},\dots,d-\frac{d}{K}\right\},
\quad\text{and }\quad
\frac{d}{L}\left[L\right]_{0}\coloneqq\left\{ 0,\frac{d}{L},\frac{2d}{L},\dots,d-\frac{d}{L}\right\},
\end{equation*}
and we let $Y_{K,L}$ be the $K\times L$ partial measurement matrix obtained by restricting $Y$ to rows in $\frac{d}{K}\left[K\right]_{0}$ and columns in $\frac{d}{L}\left[L\right]_{0}$ so that the $(k,\ell)$-th entry of $Y_{K,L}$ is given by
\begin{equation}\label{eqn:YKLdef}
(Y_{K,L})_{k,\ell}=Y_{\frac{kd}{K},\frac{\ell d}{L}}.
\end{equation}
Similarly, we let $N_{K,L}$ be the $K\times L$ matrix obtained by restricting $N$ to rows and columns in $\frac{d}{K}\left[K\right]_{0}$ and $\frac{d}{L}\left[L\right]_{0}.$

Letting $\omega_k=k\frac{d}{K}$ so that $(Y_{K,L})_{k,\ell}=Y_{\omega_k,\ell a},$ we note that 
\begin{equation}
(Y_{K,L})_{k,\ell}= \left| \left\langle \mathbf{x}, S_{\ell a} W_{\omega_k} \mathbf{m} \right\rangle \right|^2 + N_{\omega_k,\ell a} = \left| \left\langle \mathbf{x}, \mathbbm{e}^{\frac{2 \pi \mathbbm{i} \ell a \omega_k}{d}} W_{\omega_k} S_{\ell a} \mathbf{m} \right\rangle \right|^2 + N_{\omega_k,\ell a} = \left| \left\langle \mathbf{x}, W_{\omega_k} S_{\ell a} \mathbf{m} \right\rangle \right|^2 + N_{\omega_k,\ell a}.
\label{equ:GaborCase}
\end{equation}
Therefore,  $Y_{K,L}$ forms a matrix of STFT magnitude measurements.   Furthermore, when $K=d$ and $\omega_k=k,$ \eqref{equ:Ymap}  also encompasses a large class of masked Fourier magnitude measurements (i.e., CDP measurements) of the form 
\begin{equation}
\left(Y_{K,L}\right)_{k,\ell} =  \left| \left\langle \mathbf{x}, W_{\omega_k} S_{\ell a} \mathbf{m} \right\rangle \right|^2 + N_{k,\ell}  = |\left(F_d ~{\rm Diag}(\mathbf{m'}_\ell) ~\mathbf{x}\right)_k |^2 + \left(N_{K,L}\right)_{k,\ell},
\label{equ:MaskedFouerierMeas}
\end{equation}
where     $\mathbf{m'}_\ell := S_{\ell a} \overline{\mathbf{m}}$  and $F_d$ is the $d \times d$ discrete Fourier transform matrix whose entries are defined by 
\begin{equation}\label{eq:defF}
(F_d)_{j,k} \coloneqq  
\mathbbm{e}^{\frac{-2\pi\mathbbm{i} j k}{d}}.
\end{equation} 
 Measurements similar to \eqref{equ:MaskedFouerierMeas} are considered in, e.g., the recent works by Cand{\`e}s and others \cite{bandeira2014phase,candes2015phase,Candes2014WF,gross2015improved}. However,  their masks  are usually generated randomly, whereas we will consider deterministically designed mask constructed as shifts of a  single  base mask $\mathbf{m}.$

Our method for recovering $\mathbf{x}$ is based on a two-step approach. Following the example of, e.g., \cite{balan2006signal,candes2013phaselift}, we can lift the nonlinear, phaseless measurements \eqref{equ:Zmap} to  linear measurements of the Hermitian rank-one matrix $\mathbf{x}\mathbf{x}^*$. Specifically, it can be shown that
\begin{align*}
	Y'_{k,\ell}(\mathbf{x}) &= 
 \langle \x \x^*, S_{\ell a} \mathbf{m_k} \mathbf{m_k}^* S^*_{\ell a}\rangle + N'_{k,\ell},
\end{align*}
where the inner product above is the Hilbert-Schmidt inner product. Restricting, for the moment, to the case $a = 1$ (i.e. $L = d$)  and assuming that the nonzero entries of $\mathbf{m_k}$ are contained in the set $[\delta]_0$ for all $k$, one can see that every matrix $G \in {\rm span}\left(\{S_\ell \mathbf{m_k} \mathbf{m_k}^* S^*_\ell : \ell\in[d]_0,k\in[K]_0\}\right)$ will have all of its nonzero entries concentrated near  the main diagonal. Specifically, we have $G_{ij} = 0$ unless either $|i-j|<\delta$ or $|i-j|> d-\delta$
  Therefore, letting  $T_\delta : \C^{d \times d} \to \C^{d \times d}$ be the restriction  operator given by \begin{equation*} T_\delta(G)_{ij}
=\begin{cases}
G_{ij} &\text{if } |i-j| <\delta\text{ or } |i-j|>d-\delta\\ 
0 &\text{otherwise}
\end{cases},
  \end{equation*} 
we see  \begin{equation*} Y'_{k,\ell}(\mathbf{x}) = \langle \x \x^*, S_\ell \mathbf{m_k} \mathbf{m_k}^* S^*_\ell \rangle + N'_{k,\ell} = \langle T_{\delta}(\x \x^*), S_\ell \mathbf{m_k} \mathbf{m_k}^* S^*_\ell \rangle + N'_{k,\ell}, \quad (k, \ell) \in [K]_0 \times [d]_0. 
  \end{equation*}
Our lifted, linearized measurements are therefore given by $Y'_{k,\ell}(\mathbf{x})  = \mathcal{A}\big(T_\delta(\x\x^*)\big)_{(k,\ell)} + N'_{k,\ell}$ where  $\mathcal{A} : T_\delta(\C^{d \times d}) \to \C^{K \times d}$ is defined by 
\begin{equation}
(\mathcal{A}(X))_{(k, \ell)} = \langle X, S_\ell \mathbf{m_k} \mathbf{m_k}^* S^*_\ell \rangle\label{eq:linear} \quad \text{for} \quad (k, \ell)\in [K]_0 \times [d]_0 \quad \text{and} \quad X \in T_\delta(\C^{d \times d}).
\end{equation}  
As a result, one can approximately solve for $\x$ up to a global phase factor by $(i)$ evaluating $\mathcal{A}^{-1}$ on $Y'_{k,\ell}(\mathbf{x})$ in order to recover a Hermitian approximation $X_e$ to $T_\delta(\x\x^*)$, and then $(ii)$ applying a noise robust angular synchronization method (e.g., see \cite{singer2011angular,viswanathan2015fast})  to obtain an estimate of $\x$ from $X_e$.  See \cite{iwen2018phase,iwen2016fast} for further details.

The following theorem summarizes previous work using this two-stage  approach for the case where the nonzero entries of the masks $\mathbf{m_k}$ are contained in the set $[\delta]_0$ for all $k \in [K]_0$.

\begin{thm}[See \cite{iwen2018phase,iwen2016fast}]For $\x\in\mathbbm{C}^d$, let
 $\min |\mathbf{x}| \coloneqq  \min_{0\leq j\leq d} |x_j|,$ and set $K = 2\delta - 1$ and $L = d$ so that $a = 1$ in \eqref{equ:Zmap}.  There exists a practical nonlinear reconstruction algorithm that takes in measurements $Y'_{k,\ell}(\mathbf{x})$ for all $(k, \ell)\in [K] \times [d]$ and outputs an estimate $\x_e \in \mathbbm{C}^d$ that always satisfies 
 \begin{equation}
 \min_{\phi \in [0, 2 \pi]} \left\Vert  \x - \mathbbm{e}^{\mathbbm{i} \phi} \x_e \right\Vert_2 \leq C \left( \frac{\Vert \x 
        \Vert_{\infty}}{\min |\x|^2} \right) \left( \frac{d}{\delta} \right)^2 \kappa \| N' \|_F + C d^{\frac{1}{4}} \sqrt{\kappa \|N'  \|_F }.
        \label{equ:ErrorThm2}
\end{equation}
Here $\kappa > 0$ is the condition number of the linear map $\mathcal{A}$ in \eqref{eq:linear} and $C \in \mathbb{R}^+$ is an absolute universal constant.

Furthermore, it is possible to choose  masks $\mathbf{m_0},\mathbf{m_1},\ldots,\mathbf{m_{2\delta - 2}}$ such that $\kappa < 4\delta$ (see \cite{iwen2018phase}),  and it is also possible to construct a single mask $\mathbf{m} \in \mathbbm{C}^d$ such that if $\mathbf{m_k}=W_{k\frac{d}{K}}\mathbf{m},$ then $\kappa = \mathcal{O}(\delta^2)$ for the measurements that appear in \eqref{equ:Ymap} (see \cite{iwen2016fast}).
Also, if $\|N'\|_F$ is sufficiently smalll, the algorithm mentioned above is guaranteed to require just $\mathcal{O} \left( \delta^2 d \log d + \delta^3 d\right)$ total flops to achieve \eqref{equ:ErrorThm2} up to machine precision.
\label{Prop:RecovRes}
\end{thm}

Note that \eqref{equ:ErrorThm2} guarantees that the algorithm in \cite{iwen2018phase} referred to by Theorem~\ref{Prop:RecovRes} {\em exactly inverts} (up to a global phase) the measurement map in \eqref{equ:Zmap} for all nonvanishing $\x$ in the noiseless setting (i.e., when  $\| N' \|_F = 0$).  Furthermore, the error between the recovered and original signal degrades gracefully with small amounts of arbitrary additive noise, and when $\delta \ll d$ the algorithm runs in essentially FFT-time.  Indeed, a thorough numerical evaluation of this method has demonstrated it to be significantly more computationally efficient than competing techniques when, e.g., $\delta = \mathcal{O}(\log d)$ (see \cite{iwen2018phase}).  While the the first term of the error bound obtained in Theorem~\ref{Prop:RecovRes} exhibits quadratic dependence in $d,$ we note that the main results of \cite{iwen2018lower} imply, at least heuristically, that  polynomial dependencies on $d$ are actually unavoidable in any upper bound like \eqref{equ:ErrorThm2} 
when 
the masks are locally supported.  As a result, both \eqref{equ:ErrorThm2} as well as the new error bounds developed below generally must exhibit such polynomial dependences on $d$.

\subsection{Main Results}

One of the main drawbacks of Theorem~\ref{Prop:RecovRes} is that it only holds for shifts of size $a = 1,$ i.e., when $L = d$.  In real ptychographic imaging applications, however, the equivalent of our parameter $a$ will in fact often at least $0.4 \delta,$ with $\delta$ being moderately large.  Therefore, we will consider recovery scenarios where both $K$ and $L$ are strictly less than $d$ in \eqref{eqn:YKLdef}, and consider classes of $\x$ and $\m$ for which we can still guarantee noise-robust recovery results.  This motivates our first new result, which allows us to recover bandlimited signals.

\begin{restatable}[Convergence Gaurantees for Algorithm 2]{thm}{Algtwo}
\label{thm:Algorithm_2_guarantees}
 Let ${\bf x},\mathbf{m}\in\mathbb{C}^{d}$ with 
 $\mbox{supp}\left(\widehat{\mathbf{x}}\right)\subseteq\left[\gamma\right]_{0}$ and $\mbox{supp}\left({\bf \mathbf{m}}\right)\subseteq\left[\delta\right]_{0}$, and let 
\begin{equation}\label{eq:mu2}
\mu_2 = \min_{_{\substack{\left|p\right|\leq\gamma-1,\\
\left|q\right|\leq\delta-1
}
}}\left|F_d\left(\widehat{{\bf m}}\circ S_{p}\overline{\widehat{{\bf m}}}\right)_{q} \right|~>~0.
\end{equation} 
Assume that $\gamma\leq 2\delta-1< d,$ 
$L=2\gamma-1,$ $K=2\delta-1$, and also that $K$ and $L$ divide $d$.  Furthermore, suppose that the phaseless  measurements \eqref{eqn:YKLdef} have noise dominated by the norm of $\x$ so that
\begin{equation}
\label{eq:beta}
\|N_{K,L}\|_F \leq \beta \left\Vert \x \right\Vert _{2}^{2}
\end{equation}
for some $\beta\geq0$.
Then Algorithm 2 in Section \ref{sec:recGuar}  outputs an estimate $\x_e$ to $\x$ with relative error
\begin{align}
\min_{\phi\in\left[0,2\pi\right]} \frac{\left\Vert \x -\mathbbm{e}^{\mathbbm{i}\phi} \x_e \right\Vert _{2}}{\left\Vert \x \right\Vert _{2}} \leq\frac{\left(1+2\sqrt{2}\right)\beta}{\sigma_{\gamma}\left(W\right)}\frac{d^{2}}{\sqrt{KL}\mu_2},\label{eq:guarantee_2intro}
\end{align}
where $W \in \C^{2\delta-1 \times \gamma}$ is the partial Fourier matrix with entries $W_{j,k} =\mathbbm{e}^{-\frac{2\pi \mathbbm{i}\left(j-\delta+1\right)k}{d}}$ and $\gamma^{\rm th}$ singular value $\sigma_{\gamma}\left(W\right)$. Furthermore, if $\|N_{K,L}\|_F$ is sufficently small, then Algorithm 2 is always guaranteed to require at most $\mathcal{O}\left(KL\log(KL)+\delta^3+\log(\|\widehat{\x}\|_\infty)\gamma^2\right)$ total flops to achieve \eqref{eq:guarantee_2intro} up to machine precision.
\end{restatable}
%

As mentioned earlier, Theorem~\ref{thm:Algorithm_2_guarantees} allows us to recover $\mathbf{x}$ even when $K$ and $L$ are both strictly less than $d$.  Indeed, the total number of measurements, $KL = 
\mathcal{O}(\gamma\delta)$ is independent of the sample size $d$ (though it does exhibit dependence on the parameters, $\gamma$ and $\delta,$ and it also requires that $\delta > \gamma / 2$).  Nonetheless, the fact that Algorithm 2 exhibits robustness to arbitrary noise indicates that one can use it to quickly obtain a low-pass approximation to a sufficiently smooth $\x \in \C^d$ using fewer than $d$ STFT magnitude measurements.  We also note that Proposition~\ref{prop:mu_condition2}, stated in Section \ref{sec:recGuar}, shows that locally supported masks $\m$ with $\mu_2 > 0$ are relatively simple to construct.

Our second result utilizes the connection between CPD measurements and STFT magnitude measurements (see \eqref{equ:GaborCase} and \eqref{equ:MaskedFouerierMeas}) to provide a new class of deterministic CPD measurement constructions along with an associated noise-robust recovery algorithm.  Unlike previously existing deterministic constructions (see, e.g., Theorem 3.1 in \cite{candes2015phase}) the following result presents a general means of constructing deterministic CDP masks using shifts of a single bandlimited mask $\m$.

\begin{restatable}[Convergence Gaurantees for Algorithm 1]{thm}{Algone}
\label{thm:Algorithm_1_guarantees} 
Let ${\bf x},\mathbf{m}\in\mathbb{C}^{d}$ with  $\mbox{supp}\left({\bf \widehat{{\bf m}}}\right)\subseteq\left[\rho\right]_{0},$ for some $\rho<d/2.$ Let $\min\left|\widehat{{\bf x}}\right|\coloneqq \min_{0\leq n\leq d-1}|\widehat{x}_n|>0,$ and let 
\begin{equation}\label{eq:mu}
\mu_1 \coloneqq \min_{_{\substack{\left|p\right|\leq\gamma-1\\
\left|q\right|\leq\rho-1
}
}}\left|F_d\left(\widehat{{\bf m}}\circ S_{p}\overline{\widehat{{\bf m}}}\right)_{q} \right|~>~0.
\end{equation} 
Fix an integer $\kappa \in [2, \rho]$ and assume that $L=\rho+\kappa-1$ divides $d$.  Then, when $K=d,$ Algorithm 1 in Section \ref{sec:recGuar} will  output $\x_e,$ an estimate of $\x,$ such that
\begin{equation}
\min_{\phi\in\left[0,2\pi\right]}\left\Vert {\bf x}-\mathbbm{e}^{\mathbbm{i}\phi} \x_e \right\Vert _{2}\leq C\frac{d^{7/2}\left\Vert \widehat{{\bf x}}\right\Vert _{\infty}\left\Vert N_{d,L} \right\Vert _{F}}{L^{\frac{1}{2}}\mu_1\kappa^{\frac{5}{2}}\cdot\min\left|\widehat{{\bf x}}\right|^{2}}+C'\frac{d^{\frac{3}{2}}}{L^{\frac{1}{4}}}\sqrt{\frac{\left\Vert N_{d,L} \right\Vert _{F}}{\mu_1}}
\label{equ:NewErrorCDP}
\end{equation}
 for some absolute constants $C,C'\in\mathbb{R}^+$.  
 Furthermore, if $\|N_{d,L}\|_F$ is sufficently small, then  Algorithm 1 is always guaranteed to require just $\mathcal{O} \left( d (\rho + \kappa^2 ) \log d \right)$ total flops to achieve \eqref{equ:NewErrorCDP} up to machine precision.
\end{restatable}

When $\kappa = \rho = 2,$ Theorem~\ref{thm:Algorithm_1_guarantees}  guarantees that $3d$ CDP measurements suffice in order to recover any signal $\x$ with a nonvanishing discrete Fourier transform in the noiseless setting as long as $\mu_1 > 0$. Analogously to Proposition \ref{prop:mu_condition2}, Proposition~\ref{prop:mu_condition}, also stated in Section \ref{sec:recGuar}, provides a straightforward way to construct masks with $\mu_1 > 0.$ For general $\kappa$ and $\rho$, Theorem~\ref{thm:Algorithm_1_guarantees} shows that one can reconstruct signals $\x$  using $\mathcal{O}(d \rho)$ CPD measurements based on windows with Fourier support $\rho$ in just $\mathcal{O} \left( \rho^2 d \log d \right)$-time when $\left \Vert N \right\Vert _{F}$ is sufficiently small. We also note that in addition to the theoretical guarantees provided by Theorems \ref{thm:Algorithm_2_guarantees} and \ref{thm:Algorithm_1_guarantees}, Section~\ref{sec:NumEval} demonstrates that both Algorithm 1 and 2 are fast, accurate, and robust to noise in practice as well.

\subsection{Related Work}

The connections between theoretical time-frequency analysis and phaseless  imaging  (e.g., ptychography) have been touched on in the physics community many times over the past several decades.  As noted well over two decades ago in \cite{Rodenburg1992} and later in \cite{Chapman1996},  continuous spectrogram measurements can be written as the convolution of the Wigner distribution functions of the specimen $\mathbf{x}$  and the probe $\mathbf{m}.$  Furthermore, \cite{daSilva:15} has pointed out that this allows one to recover the specimen of interest if enough samples are drawn so that the Heisenberg boxes sufficiently cover the time-frequency plane.  In this work, we use similar ideas formulated in the discrete setting to efficiently invert the types of structured lifted linear maps $\mathcal{A}$ as per \eqref{eq:linear} that appear in \cite{iwen2018phase,iwen2016fast}, and use  angular synchronization approaches to recover the  signal $\x$ up to a global phase. Specifically, we produce two new, efficient algorithms for inverting discrete spectrogram measurements that are provably accurate and robust to arbitrary additive measurement errors.

In \cite{bendory2017non}, Bendory and Eldar prove results similar to some of those 
summarized in Theorem~\ref{Prop:RecovRes} in the case there $a=1$ and  $\|N'\|_F=0,$ 
and they also demonstrate numerically that their algorithms are robust to noise.  In this paper, we prove  noise-robust recovery results, where we allow $a>1.$ However, we make additional assumptions about either the support of $\widehat{\m}$ or  the supports of $\m$ and $\widehat{\x}.$   We also note the very recent and excellent work of Rayan Saab and Brian Preskitt \cite{preskitt2018phase} as well as that of Melnyk, Filbir, and Krahmer \cite{Melnyk:19} which both prove results similar to Theorem \ref{thm:Algorithm_2_guarantees}. As in Theorem \ref{thm:Algorithm_2_guarantees}, the results of \cite{Melnyk:19,preskitt2018phase} can guarantee recovery with shift sizes $a>1$. 
Their results primarily differ from Theorem \ref{thm:Algorithm_2_guarantees} in that they don't used Wigner Distribution Deconvolution (WDD) based methods. As a result, they  consider different classes of masks and signals than we do. 

 Other related work includes that of Salanevich and Pfander \cite{pfander2019robust,salanevich2015polarization} which builds upon the work of Alexeev et al. \cite{alexeev2014phase} to establish noise robust recovery results for Gabor frame-based measurements.  Their noise robust approach  has similar characteristics to the approach taken here with the primary differences being that they require additional measurements beyond those provided by shifts and modulations of a single mask (see, e.g., equation (8) in \cite{pfander2019robust}), and in some sense utilize the reverse of the approach taken here:  Instead of first solving a linear system to obtain an approximation of (a portion of) $\x \x^*$, and then using angular synchronization to obtain an approximation to $\x$, the methods of \cite{pfander2019robust,salanevich2015polarization} instead first use angular synchronization methods to obtain frame coefficients of $\x$, and then reconstruct $\x$ using the recovered frame coefficients.



The rest of the paper is organized as follows.  In Section \ref{Sec:BasicProps}, we establish necessary notation and state a number of preliminary lemmas.  Then, in Section~\ref{sec:WDD}, we establish  several discrete and aliased variants of WDD, some of which can be used 
when the mask $\m$ is locally supported in physical space, and others 
for when  $\m$ is locally supported in Fourier space.  In Section~\ref{sec:recGuar}, we prove Theorems \ref{thm:Algorithm_2_guarantees} and \ref{thm:Algorithm_1_guarantees} which provide recovery guarantees for our proposed methods 
and also state propositions which describe ways to design masks so that the assumptions  of these theorems are valid.  Finally, in Section~\ref{sec:NumEval}, we evaluate our  algorithms numerically and show that they are fast and  robust to additive measurement noise.

\section{Notation and Preliminary Results}\label{Sec:BasicProps}

For ${\bf x}\coloneqq(x_0,\ldots,x_{d-1})^T \in\mathbb{C}^{d},$
 we let
\[
\mbox{supp}\left({\bf x}\right)\coloneqq\left\{ n\in[d]_0: x_{n}\neq0\right\} 
\]
denote the support of $\mathbf{x},$ where, as in Section \ref{sec:intro}, $[d]_0=\{0,1,\ldots,d-1\}.$ 
 We let $R\mathbf{x}\coloneqq\widetilde{{\bf x}}$
denote the reversal of ${\bf x}$ about its first entry, i.e.,
\[
(R\mathbf{x})_n=\widetilde{x}_{n}\coloneqq x_{-n\!\!\!\!\mod\,d}\quad\text{for }0\leq n\leq d-1, 
\]
and we recall from \eqref{eq:defshift} and \eqref{eq:defmod}  the circular shift and modulation operators given by
$\left(S_{\ell}{\bf x}\right)_{n}= x_{\left(\ell+n\right)\!\!\!\!\mod\,d}$
and $\left(W_{k}{\bf x}\right)_{n}= x_{n}\mathbbm{e}^{\frac{2\pi \mathbbm{i}kn}{d}}.$
In order to avoid cumbersome notation, if $n$ is  not an element of 
$[d]_0,$
we will  write $x_n$ in place of $x_{n \!\!\!\mod d}.$ 
 For $\mathbf{x}\in\mathbb{C}^d$, we define the Fourier transform of $\mathbf{x}$ by  
\[
\widehat{x}_{k}\coloneqq \left(F_{d}{\bf x}\right)_{k}=\sum_{n=0}^{d-1}x_{n}\mathbbm{e}^{-\frac{2\pi \mathbbm{i}nk}{d}},
\]
where as in \eqref{eq:defF}, $F_{d}\in\mathbb{C}^{d\times d}$ denotes the $d\times d$ discrete
Fourier transform  matrix with entries 
$
\left(F_{d}\right)_{j,k}=\mathbbm{e}^{-\frac{2\pi \mathbbm{i}j k}{d}}$ for $0\leq j,k\leq d-1.
$
For  ${\bf x},{\bf y}\in\mathbb{C}^{d}$ and
$\ell\in\left[d\right]_{0}$, we define circular convolution and Hadamard (pointwise) multiplication by
\begin{align*}
\left({\bf x}\ast_{d}{\bf y}\right)_{\ell} & \coloneqq \sum_{n=0}^{d-1}x_{n}y_{\ell-n},\:\:\: \text{ and} \quad
\left({\bf x}\circ{\bf y}\right)_{\ell} \coloneqq x_{\ell}y_{\ell},
\end{align*}and we define their componentwise quotient $\frac{\mathbf{x}}{\mathbf{y}}$ and componentwise absolute value $|\mathbf{x}|$ by 
\begin{equation*}
\left(\frac{\mathbf{x}}{\mathbf{y}}\right)_n=\frac{x_n}{y_n}\quad\text{and}\quad
|\mathbf{x}|_n=|x_n|.
\end{equation*}

For a matrix $M,$ we let $M_k$ denote its $k$-th column, and 
let 
$\|M\|_F$ 
 denote its Frobenius norm. 
 When proving the convergence of our algorithms, we will use the fact that, up to a reorganization of the terms, a banded $d\times d$ matrix, whose nonzero entries are contained within $\kappa$ entries of the main diagonal is equivalent to a $(2\kappa-1)\times d$ matrix whose columns are the diagonal bands of the square, banded matrix. Towards this end, if $2\kappa-1\leq d$ and  $M=(M_{1-\kappa},\ldots,M_0,\ldots M_{\kappa-1})$ is a $(2\kappa-1)\times d$ matrix with columns indexed from $1-\kappa$ to $\kappa-1$ so that column zero is the middle column, we let  $C_{2\kappa-1}(M)$ be the banded $d\times d$ matrix with entries given  by 
\begin{equation}
\left(C_{2\kappa-1}\left(M\right)\right)_{j,k}=\begin{cases}M_{j,k-j}\ &\text{if } |j-k|< \kappa\text{ or }|j-k|>d-\kappa\\
0 	&\text{otherwise}
\end{cases}
\label{eq:circulant_operator}
\end{equation}
for $j,k\in[d]_0.$ 
By construction, the columns of $M$ are the diagonal bands of $C_{2\kappa-1}(M)$  with the middle column $M_0$ lying on the main diagonal. For example, in the case where $\kappa=2,$ \[C_3\left(\left[\begin{array}{ccc}
a_{0,-1} & a_{0,0} & a_{0,1}\\
a_{1,-1} & a_{1,0} & a_{1,1}\\
\vdots & \vdots & \vdots\\
a_{d-2,-1} & a_{d-2,0} & a_{d-2,1}\\
a_{d-1,-1} & a_{d-1,0} & a_{d-1,1}
\end{array}\right]\right)=\left[\begin{array}{ccccc}
a_{0,0} & a_{0,1} & \cdots & 0 & a_{0,-1}\\
a_{1,-1} & a_{1,0} & a_{1,1} &  & 0\\
\vdots & \ddots & \ddots & \ddots & \vdots\\
0 &  & a_{d-2,-1} & a_{d-2,0} & a_{d-2,1}\\
a_{d-1,1} & 0 & \cdots & a_{d-1,-1} & a_{d-1,0}
\end{array}\right].
\]
Below, we will state a number of lemmas, some of which are well known, which we will use in the proofs of our main results. Proofs are provided in the appendix. 
Our first lemma  summarizes a number of properties of the discrete Fourier transform and the operators above. 
\begin{lem}
\label{lem:DFTproperties} For all ${\bf x}\in\mathbb{C}^{d}$
and $\ell\in\left[d\right]_{0}$,
\begin{enumerate}
\item $F_{d}\widehat{{\bf x}}=d\widetilde{{\bf x}},$ \label{FF}
\item $F_{d}\left(W_{\ell}{\bf x}\right)=S_{-\ell}\widehat{{\bf x}},$ \label{FW}
\item $F_{d}\left(S_{\ell}{\bf x}\right)=W_{\ell}\widehat{{\bf x}},$ \label{FS}
\item $W_{-\ell}F_{d}\left(S_{\ell}\overline{\widetilde{{\bf x}}}\right)=\overline{\widehat{{\bf x}}},$ \label{WFS}
\item $\overline{\widetilde{S_{\ell}{\bf x}}}=S_{-\ell}\overline{\widetilde{{\bf x}}},$ \label{BarSTilde}
\item $F_{d}\overline{{\bf x}}=\overline{F_{d}\widetilde{{\bf x}}},$ \label{FBar}
\item $\widetilde{\widehat{{\bf x}}}=\widehat{\widetilde{{\bf x}}},$ \label{TildeHat}
\item $\left|F_{d}{\bf x}\right|^{2}=F_{d}\left({\bf x}\ast_{d}\overline{\widetilde{{\bf x}}}\right).$ \label{lem:absoluteFourierSquared}
\end{enumerate}
\end{lem}
%
%
%
The following lemma is the discrete analogue of the convolution theorem.
\begin{lem}
\label{lem:convolutionTheorem} (Convolution Theorem) For all ${\bf x},{\bf y}\in\mathbb{C}^{d},$

\[
F_{d}^{-1}\left(\widehat{{\bf x}}\circ\widehat{{\bf y}}\right)={\bf x}\ast_{d}{\bf y},
\]
and
\[
\left(F_{d}{\bf x}\right)\ast_{d}\left(F_{d}{\bf y}\right)=d F_{d}\left({\bf x}\circ{\bf y}\right).
\]
\end{lem}
In much of our analysis, we will have to consider the Hadamard product  of a vector with a shifted copy of itself. The next three lemmas will be useful when we need to manipulate terms of that form.
\begin{lem}
\label{lem:fourierIndexSwap} Let ${\bf x}\in\mathbb{C}^{d},$
and let $\alpha,\omega\in\left[d\right]_{0}$. Then,
\[
\left(F_{d}\left({\bf x}\circ S_{\omega}{\bf \overline{x}}\right)\right)_{\alpha}=\frac{1}{d}\mathbbm{e}^{\frac{2\pi \mathbbm{i}\omega\alpha}{d}}\left(F_{d}\left(\widehat{{\bf x}}\circ S_{-\alpha}{\bf \overline{\widehat{{\bf x}}}}\right)\right)_{\omega}.
\]
\end{lem}

\begin{lem}\label{lem:tildekill} Let $\mathbf{x}\in\mathbb{C}^{d},$ and let $\alpha\in\mathbb{Z}.$ Then,
\begin{equation*}
F_d\left(\widetilde{\mathbf{x}}\circ S_{-\alpha}\overline{\widetilde{\mathbf{x}}}\right)=R(F_d(\mathbf{x}\circ S_\alpha \overline{\mathbf{x}})).
\end{equation*}
\end{lem}

\begin{lem}
\label{lem:indexSwap} Let ${\bf x},{\bf y}\in\mathbb{C}^{d},$
and let $\ell,k\in\left[d\right]_{0}$. Then,
\[
\left(\left({\bf x}\circ S_{-\ell}{\bf y}\right)\ast_{d}\left(\overline{\widetilde{{\bf x}}}\circ S_{\ell}\overline{\widetilde{{\bf y}}}\right)\right)_{k}=\left(\left({\bf x}\circ S_{-k}{\bf \overline{x}}\right)\ast_{d}\left(\widetilde{{\bf y}}\circ S_{k}\overline{\widetilde{{\bf y}}}\right)\right)_{\ell}.
\]
\end{lem}

For a positive integer $s$ which divides $d$, we introduce the subsampling
operator
\[
Z_{s}:\mathbb{C}^{d}\to\mathbb{C}^{\frac{d}{s}},
\]
defined by
\[
\left(Z_{s}{\bf x}\right)_{n}\coloneqq x_{n s}\ \text{for } n\in\left[\frac{d}{s}\right]_{0}.
\]
The following lemma shows that taking the Fourier transform of a subsampled vector produses an aliasing effect.
\begin{lem}
\label{lem:aliasing} (Aliasing) Let $s$ be a positive integer which divides $d.$ Then for  ${\bf x}\in\mathbb{C}^{d}$
and $\omega\in\left[\frac{d}{s}\right]_{0}$, 
\[
\left(F_{\frac{d}{s}}\left(Z_{s}{\bf x}\right)\right)_{\omega}=\frac{1}{s}\sum_{r=0}^{s-1}\widehat{x}_{\omega-r\frac{d}{s}}.
\]
\end{lem}

\section{Aliased Wigner Distribution Deconvolution for Fast Phase Retrieval}
\label{sec:WDD}

As in Section \ref{sec:intro}, we 
let ${\bf x}\in\mathbb{C}^{d}$ denote an unknown quantity of interest and let ${\bf m}\in\mathbb{C}^{d}$ denote a known
measurement mask, and consider measurements $Y_{k,\ell}$ of the form \eqref{equ:Ymap}. By \eqref{equ:GaborCase}, we see we may write $Y_{k,\ell}$ as a noisy windowed Fourier magnitude measurement of the form
\begin{equation}
Y_{k,\ell}=\left|\sum_{n=0}^{d-1}x_{n}m_{n-\ell}\mathbbm{e}^{-\frac{2\pi \mathbbm{i}nk}{d}}\right|^{2}+N_{k,\ell},\quad\text{for }0\leq k,\ell\leq d-1. \label{eq:y_ell_k}
\end{equation}
Let $\mathbf{y}_\ell$ and $\n_\ell$ denote the $\ell$-th columns of the measurement matrix $Y=(Y_{k,\ell})_{0\leq k,\ell\leq d-1}$ and the noise matrix $N=(N_{k,\ell})_{0\leq k,\ell\leq d-1}$ respectively, 
and, as in Section \ref{sec:intro}, let $Y_{K,L}$ be the $K\times L$ partial measurement matrix obtained by restricting $Y$ to rows in $\frac{d}{K}\left[K\right]_{0}$ and columns in $\frac{d}{L}\left[L\right]_{0}$ 
so the the entries of $Y_{K,L}$ are given by \eqref{eqn:YKLdef},
and let $N_{K,L}$ be the analogous matrix obtained by restricting $N$ to rows and columns in $\frac{d}{K}\left[K\right]_{0}$ and $\frac{d}{L}\left[L\right]_{0}.$

Our goal is to recover  $\mathbf{x}$ (up to a global phase) from these measurements with an error that may be bounded in terms of the magnitude of the noise $N.$ Our method will be based on the following result that is an aliased and discrete variant of the Wigner Distribution Deconvolution (WDD)
approach presented in the continuous setting by Chapman  in \cite{Chapman1996}. Together with Lemmas \ref{thm:generalSumCollapse-2}, \ref{thm:generalSumCollapse}, and \ref{thm:generalSumCollapse-1}, it will allow us to recover portions of the rank one matrices $\mathbf{x}\mathbf{x}^*$ and $\widehat{\mathbf{x}}\widehat{\mathbf{x}}^*.$

\begin{thm}
\label{thm:subsamplingInFrequencyAndSpace} 
Let $Y_{K,L}$ be the $K\times L$ partial measurement matrix defined in \eqref{eqn:YKLdef}, and let $N_{K,L}$ be the corresponding partial noise matrix. 
Let $\widetilde{Y}$ and $\widetilde{N}$ be the $L\times K$ matrices defined by
\begin{equation*}
\widetilde{Y}\coloneqq F_LY^T_{K,L}F_K^T\quad\text{and}\quad\widetilde{N}\coloneqq F_LN^T_{K,L}F_K^T.
\end{equation*} 
Then for any $\omega\in\left[K\right]_{0}$
and $\alpha\in\left[L\right]_{0}$,
\begin{align}\widetilde{Y}_{\alpha,\omega}&=\frac{KL}{d^{3}}\sum_{r=0}^{\frac{d}{K}-1}\sum_{\ell=0}^{\frac{d}{L}-1}\left(F_{d}\left(\widehat{{\bf x}}\circ S_{\ell L-\alpha}\overline{{\bf \widehat{x}}}\right)\right)_{\omega-rK}\left(F_{d}\left({\bf \widehat{m}}\circ S_{\alpha-\ell L}\overline{{\bf \widehat{m}}}\right)\right)_{\omega-rK}+\widetilde{N}_{\alpha,\omega}\label{eq:FYF_double_aliasing}\\
&=\frac{KL}{d^2}\sum_{r=0}^{\frac{d}{K}-1}\sum_{\ell=0}^{\frac{d}{L}-1}\mathbbm{e}^{-2\pi \mathbbm{i}(\ell L-\alpha)(\omega-rK)/d}\left(F_{d}\left(\widehat{{\bf x}}\circ S_{\ell L-\alpha}\overline{{\bf \widehat{x}}}\right)\right)_{\omega-rK}\left(F_{d}\left({\bf m}\circ S_{\omega-rK}\overline{{\bf m}}\right)\right)_{\ell L-\alpha}+\widetilde{N}_{\alpha,\omega}\label{eq:FYF_double_aliasing2}\\
&=\frac{KL}{d^2}\sum_{r=0}^{\frac{d}{K}-1}\sum_{\ell=0}^{\frac{d}{L}-1}\mathbbm{e}^{2\pi \mathbbm{i}(\ell L-\alpha)(\omega-rK)/d}\left(F_{d}\left({\bf x}\circ S_{\omega-rK}\overline{{\bf x}}\right)\right)_{\alpha-\ell L}\left(F_{d}\left({\bf \widehat{m}}\circ S_{\alpha-\ell L}\overline{{\bf \widehat{m}}}\right)\right)_{\omega-rK}+\widetilde{N}_{\alpha,\omega}\label{eq:FYF_double_aliasing3}\\
&=\frac{KL}{d}\sum_{r=0}^{\frac{d}{K}-1}\sum_{\ell=0}^{\frac{d}{L}-1}\left(F_{d}\left({\bf x}\circ S_{\omega-rK}\overline{{\bf x}}\right)\right)_{\alpha-\ell L}\left(F_{d}\left({\bf m}\circ S_{\omega-rK}\overline{{\bf m}}\right)\right)_{\ell L-\alpha}+\widetilde{N}_{\alpha,\omega},\label{eq:FYF_double_aliasing4}.
\end{align}
\end{thm}

To aid in the readers understanding, before proving Theorem \ref{thm:subsamplingInFrequencyAndSpace}, we will first give a short proof of the following lemma which is the special case of \eqref{eq:FYF_double_aliasing4} where $K=L=d$.  It is the direct analogue of Chapman's WDD
approach as formulated in the continuous setting in \cite{Chapman1996}.

\begin{lem}\label{lem:specialcase} Let $Y$ be the $d\times d$ measurement matrix with entries defined as in \eqref{eq:y_ell_k} and let $N$ be the corresponding noise matrix. Then, the $\omega$-th column of $\widetilde{Y}=F_dY^TF_d^T$ is given by

 \begin{equation}\label{eqn:WDDdd}
\widetilde{Y}_\omega=d\cdot F_{d}\left({\bf x}\circ S_{\omega}{\bf \overline{x}}\right)\circ R\left(F_{d}\left({\bf m}\circ S_{\omega}\overline{{\bf m}}\right)\right)+\widetilde{N}_\omega, 
\end{equation}
where $\widetilde{N}=F_dN^TF_d^T.$
\end{lem}

\begin{proof}[The Proof of Lemma \ref{lem:specialcase}]
As noted in \eqref{eq:y_ell_k}, we may write $\mathbf{y}_\ell$ as the STFT of $\mathbf{x}$ with window $S_{-\ell}\mathbf{m}.$ Therefore, by  Lemma \ref{lem:DFTproperties}, parts \ref{BarSTilde} and \ref{lem:absoluteFourierSquared}, we see that for any $\ell\in\left[d\right]_{0}$,
\begin{align}
{\bf y}_{\ell} & =\left|F_{d}\left({\bf x}\circ S_{-\ell}{\bf m}\right)\right|^{2}+{\bf \eta}_{\ell}\nonumber \\
 & =F_{d}\left(\left({\bf x}\circ S_{-\ell}{\bf m}\right)\ast_{d}\left(\overline{\widetilde{{\bf x}}}\circ S_{\ell}\overline{\widetilde{{\bf m}}}\right)\right)+{\bf \eta}_{\ell}.\label{eq:y_ell}
\end{align}
Thus, taking a Fourier transform of ${\bf y}_{\ell}$ and applying Lemma \ref{lem:DFTproperties}, part \ref{FF}, 
yields
\[
\left(F_{d}{\bf y}_{\ell}\right)_{\omega}=d\left(\left({\bf x}\circ S_{-\ell}{\bf m}\right)\ast_{d}\left(\overline{\widetilde{{\bf x}}}\circ S_{\ell}\overline{\widetilde{{\bf m}}}\right)\right)_{-\omega}+\left(F_{d}{\bf \eta}_{\ell}\right)_{\omega},
\]
and so, by Lemma \ref{lem:indexSwap}, 
\begin{align}
\left(F_{d}{\bf y}_{\ell}\right)_{\omega} & =d\left(\left({\bf x}\circ S_{\omega}{\bf \overline{x}}\right)\ast_{d}\left(\widetilde{{\bf m}}\circ S_{-\omega}\overline{\widetilde{{\bf m}}}\right)\right)_{\ell}+\left(F_{d}{\bf \eta}_{\ell}\right)_{\omega}.\label{eq:y_hat_ell_k}
\end{align}
Since $\left(F_{d}{\bf y}_{\ell}\right)_{\omega}=(F_dY)_{\omega,\ell},$ taking the transpose of the above equation implies
\[
\left(Y^{T}F_{d}^{T}\right)_{\ell,\omega}=d\left(\left({\bf x}\circ S_{\omega}{\bf \overline{x}}\right)\ast_{d}\left(\widetilde{{\bf m}}\circ S_{-\omega}\overline{\widetilde{{\bf m}}}\right)\right)_{\ell}+\left(N^{T}F_{d}^{T}\right)_{\ell,\omega},
\]
Therefore, the $\omega$-th columns of $Y^{T}F_{d}^{T}$ and  $N^{T}F_{d}^{T}$ 
satisfy
\[
\left(Y^{T}F_{d}^{T}\right)_{\omega}=d\left({\bf x}\circ S_{\omega}{\bf \overline{x}}\right)\ast_{d}\left(\widetilde{{\bf m}}\circ S_{-\omega}\overline{\widetilde{{\bf m}}}\right) +\left(N^{T}F_{d}^{T}\right)_{\omega},
\]
so, taking the Fourier transform of both sides and applying Lemmas \ref{lem:convolutionTheorem} and \ref{lem:tildekill} 
yields
\begin{align*}
\left(F_{d}Y^{T}F_{d}^{T}\right)_{\omega} & =d F_{d}\left({\bf x}\circ S_{\omega}{\bf \overline{x}}\right)\circ F_{d}\left(\widetilde{{\bf m}}\circ S_{-\omega}\overline{\widetilde{{\bf m}}}\right)+\left(F_{d}N^{T}F_{d}^{T}\right)_{\omega}\\
& =d F_{d}\left({\bf x}\circ S_{\omega}{\bf \overline{x}}\right)\circ R\left(F_{d}\left({\bf m}\circ S_{\omega}\overline{{\bf m}}\right)\right)+\left(F_{d}N^{T}F_{d}^{T}\right)_{\omega}.
\end{align*}
Recalling that $\widetilde{Y}=F_dY^TF_d^T$ and $\widetilde{N}=F_dN^TF_d^T$ completes the proof.
\end{proof}

The following lemma 
applies analysis similar to the previous lemma to subsampled column vectors using Lemma \ref{lem:aliasing}.
\begin{lem}\label{lem:subsampledcols}
For $\ell\in\left[d\right]_{0}$ and $\omega\in\left[K\right]_{0}$,
\[
\left(F_{K}Z_{\frac{d}{K}}\left({\bf y}_{\ell}\right)\right)_{\omega}=K\sum_{r=0}^{\frac{d}{K}-1}\left(\left({\bf x}\circ S_{\omega-rK}\overline{{\bf x}}\right)\ast_{d}\left(\widetilde{{\bf m}}\circ S_{rK-\omega}\overline{\widetilde{{\bf m}}}\right)\right)_{\ell}+\left(F_{K}Z_{\frac{d}{K}}\left({\bf \eta}_{\ell}\right)\right)_{\omega}.
\]
\end{lem}
\begin{proof}
As in the proof of Lemma \ref{lem:specialcase}, the $\ell^{\text{th}}$ columns of the $Y$ and $N$ satisfy
\begin{align*}
{\bf y}_{\ell} & =\left|F_{d}\left({\bf x}\circ S_{-\ell}{\bf m}\right)\right|^{2}\nonumber \\ 
 & =F_{d}\left(\left({\bf x}\circ S_{-\ell}{\bf m}\right)\ast_{d}\left(\overline{\widetilde{{\bf x}}}\circ S_{\ell}\overline{\widetilde{{\bf m}}}\right)\right)+{\bf \eta}_{\ell}.\label{eq:y_ell}
\end{align*}
Therefore, subtracting $\mathbf{\eta}_{\ell}$ from both sides, taking the Fourier transform, applying Lemma \ref{lem:aliasing}, and then using \eqref{eq:y_hat_ell_k} 
we see
\begin{align*}
\left(F_{K}\left(Z_{\frac{d}{K}}\left({\bf y}_{\ell}-{\bf \eta}_{\ell}\right)\right)\right)_{\omega} & =\frac{K}{d}\sum_{r=0}^{\frac{d}{K}-1}\left(F_d({\bf y}_{\ell}-{\bf \eta}_{\ell})\right)_{\omega-rK}\\
 & =d\cdot\frac{K}{d}\sum_{r=0}^{\frac{d}{K}-1}\left(\left({\bf x}\circ S_{\omega-rK}{\bf \overline{x}}\right)\ast_{d}\left(\widetilde{{\bf m}}\circ S_{rK-\omega}\overline{\widetilde{{\bf m}}}\right)\right)_{\ell}\\
 & =K\sum_{r=0}^{\frac{d}{K}-1}\left(\left({\bf x}\circ S_{\omega-rK}{\bf \overline{x}}\right)\ast_{d}\left(\widetilde{{\bf m}}\circ S_{rK-\omega}\overline{\widetilde{{\bf m}}}\right)\right)_{\ell}.
\end{align*}
The lemma follows from the linearity of the Fourier transform and of the subsampling operator $Z_{\frac{d}{K}}.$
\end{proof}

Now we shall prove Theorem \ref{thm:subsamplingInFrequencyAndSpace}.

\begin{proof}[The Proof of Theorem \ref{thm:subsamplingInFrequencyAndSpace}]
 Noting that  $Y_{K,L}$ is obtained by subsampling the rows and columns of $Y,$  we see that 
the $\ell$-th column of  $Y_{K,L}-N_{K,L}$ is given by
\begin{equation*}
\left(Y_{K,L}-N_{K,L}\right)_\ell =\left(Z_{\frac{d}{K}}\left({\bf y}_{\ell\frac{d}{L}}-{\bf \eta}_{\ell\frac{d}{L}}\right)\right).
\end{equation*}
Therefore, applying Lemma \ref{lem:subsampledcols} we see
\begin{align*}
\left(\left(Y_{K,L}-N_{K,L}\right)^TF_K^T\right)_{\ell,\omega}
&=\left(F_K\left(Y_{K,L}-N_{K,L}\right)_{\ell}\right)_{\omega}\\
&=\left(F_K\left( Z_{\frac{d}{K}}\left({\bf y}_{\ell\frac{d}{L}}-{\bf \eta}_{\ell\frac{d}{L}}\right)\right)\right)_{\omega}\\
& =K\sum_{r=0}^{\frac{d}{K}-1}\left(\left({\bf x}\circ S_{\omega-rK}\overline{{\bf x}}\right)\ast_{d}\left(\widetilde{{\bf m}}\circ S_{rK-\omega}\overline{\widetilde{{\bf m}}}\right)\right)_{\ell\frac{d}{L}}\\
 & =K\left(Z_{\frac{d}{L}}\left(\sum_{r=0}^{\frac{d}{K}-1}\left({\bf x}\circ S_{\omega-rK}\overline{{\bf x}}\right)\ast_{d}\left(\widetilde{{\bf m}}\circ S_{rK-\omega}\overline{\widetilde{{\bf m}}}\right)\right)\right)_{\ell}.
\end{align*}
 Thus,  the $\omega$-th column of $\left(Y_{K,L}-N_{K,L}\right)^TF_K^T$ is given by 
\begin{equation*}
\left(\left(Y_{K,L}-N_{K,L}\right)^TF_K^T\right)_{\omega}
  =K\left(Z_{\frac{d}{L}}\left(\sum_{r=0}^{\frac{d}{K}-1}\left({\bf x}\circ S_{\omega-rK}\overline{{\bf x}}\right)\ast_{d}\left(\widetilde{{\bf m}}\circ S_{rK-\omega}\overline{\widetilde{{\bf m}}}\right)\right)\right).
\end{equation*}

Taking the Fourier transform of both sides and applying  Lemmas \ref{lem:aliasing},
  \ref{lem:convolutionTheorem}, and \ref{lem:tildekill} we see that
\begin{align*}\left(F_L\left(Y_{K,L}-N_{K,L}\right)F_K^T\right)_{\alpha,\omega}&=\left(F_L\left(\left(Y_{K,L}-N_{K,L}\right)^TF_K^T\right)_{\omega}
\right)_\alpha\nonumber\\
 & =\frac{KL}{d}\sum_{r=0}^{\frac{d}{K}-1}\sum_{\ell=0}^{\frac{d}{L}-1}\left(F_{d}\left({\bf x}\circ S_{\omega-rK}\overline{{\bf x}}\right)\right)_{\alpha-\ell L}\left(F_{d}\left(\widetilde{{\bf m}}\circ S_{rK-\omega}\overline{\widetilde{{\bf m}}}\right)\right)_{\alpha-\ell L}\\
& =\frac{KL}{d}\sum_{r=0}^{\frac{d}{K}-1}\sum_{\ell=0}^{\frac{d}{L}-1}\left(F_{d}\left({\bf x}\circ S_{\omega-rK}\overline{{\bf x}}\right)\right)_{\alpha-\ell L}\left(F_{d}\left({\bf m}\circ S_{\omega-rK}\overline{{\bf m}}\right)\right)_{\ell L-\alpha}
\end{align*}
for all $\alpha\in\left[L\right]_{0}.$  

Using the linearity of the Fourier transform and the definitions of $\widetilde{Y}$ and $\widetilde{N}$ completes the proof of \eqref{eq:FYF_double_aliasing4}.  
\eqref{eq:FYF_double_aliasing}, \eqref{eq:FYF_double_aliasing2}, and \eqref{eq:FYF_double_aliasing3} follow by using Lemma \ref{lem:fourierIndexSwap} to see that 
\begin{equation*}
\left(F_{d}\left({\bf x}\circ S_{\omega-rK}\overline{{\bf x}}\right)\right)_{\alpha-\ell L}
=\frac{1}{d}\mathbbm{e}^{-2\pi \mathbbm{i} (\ell L-\alpha)(\omega-rK)/ d}\left(F_{d}\left(\widehat{{\bf x}}\circ S_{\ell L-\alpha}\overline{{\bf \widehat{x}}}\right)\right)_{\omega-rK},
\end{equation*}
and
\begin{equation*}
\left(F_{d}\left({\bf m}\circ S_{\omega-rK}\overline{{\bf m}}\right)\right)_{\ell L-\alpha}
=\frac{1}{d}\mathbbm{e}^{2\pi \mathbbm{i} (\ell L-\alpha)(\omega-rK)/ d}\left(F_{d}\left({\bf \widehat{m}}\circ S_{\alpha-\ell L}\overline{{\bf \widehat{m}}}\right)\right)_{\omega-rK}.
\end{equation*}
\end{proof}

\subsection{Solving for Diagonal Bands of the Rank-One Matrices}

We wish to use Theorem \ref{thm:subsamplingInFrequencyAndSpace} to solve for diagonal bands of the rank-one matrix $\mathbf{x}\mathbf{x}^*.$ In the case where $K=L=d,$ one can use \eqref{eqn:WDDdd} to see that for $\omega\in[d]_0$
\begin{equation*}
\mathbf{x}\circ S_\omega \overline{\mathbf{x}} = \frac{1}{d} F_d^{-1}\left( \frac{\left(F_dY^TF_d^T\right)_\omega}{F_d\left(\widetilde{\mathbf{m}}\circ S_{-\omega}\overline{\widetilde{\mathbf{m}}}\right)}\right)-\frac{1}{d} F_d^{-1}\left( \frac{\left(F_dN^TF_d^T\right)_\omega}{F_d\left(\widetilde{\mathbf{m}}\circ S_{-\omega}\overline{\widetilde{\mathbf{m}}}\right)}\right).
\end{equation*}
 
However, in general, the right-hand side of \eqref{eq:FYF_double_aliasing}-\eqref{eq:FYF_double_aliasing4} are linear combinations of multiple terms and therefore, it is not as straightforward to solve for these diagonal bands. In this subsection, we present several lemmas which make different assumptions on the spatial and frequency supports of $\mathbf{x},$  $\mathbf{m},$ and $\widehat{\mathbf{m}}$ and identify special cases where these sums reduce to a single nonzero term. In these cases, we will then be able to solve for diagonal bands of either $\mathbf{x}\mathbf{x}^*$ or $\widehat{\mathbf{x}}\widehat{\mathbf{x}}^*$ by formulas similar to the one above. We will use Lemmas \ref{thm:generalSumCollapse-2} and \ref{thm:generalSumCollapse}  in the proofs of Theorems  \ref{thm:Algorithm_2_guarantees} and \ref{thm:Algorithm_1_guarantees}. We state Lemma \ref{thm:generalSumCollapse-1} in order to demonstrate that Wigner deconvolution approach can also be applied to the setting considered in \cite{iwen2018phase}. We will provide the proof of Lemma \ref{thm:generalSumCollapse}. The proofs of Lemmas \ref{thm:generalSumCollapse-2} and \ref{thm:generalSumCollapse-1} are nearly identical.

The first lemma in this section assumes that $\mathbf{x}$ is bandlimited and the spatial support of $\mathbf{m}$ is contained in an interval of length $\delta.$  It allows us to recover diagonal bands of the rank-one matrix $\mathbf{\widehat{x}}\mathbf{\widehat{x}}^*.$
\begin{lem}
\label{thm:generalSumCollapse-2} Let ${\bf x}, {\bf m}\in\mathbb{C}^{d}$
 with $\mbox{supp}\left(\widehat{{\bf x}}\right)\subseteq\left[\gamma\right]_{0}$
and  $\mbox{supp}\left({\bf m}\right)\subseteq\left[\delta\right]_{0}$.
Let $K$ and $L$ divide $d,$ and let $Y_{K,L}$  be the $K\times L$ partial measurement matrix defined as in \eqref{eqn:YKLdef} and let $N_{K,L}$ be the corresponding subsampled noise matrix.
As in the statement of Theorem \ref{thm:subsamplingInFrequencyAndSpace},  let
\begin{equation*}
\widetilde{Y}= F_LY_{K,L}F_K^T\quad\text{and}\quad \widetilde{N}= F_LN_{K,L}F_K^T.
\end{equation*} 
Then for any $\alpha\in\left[L\right]_{0}$ and $\omega\in\left[K\right]_{0}$,\begin{align*}
\widetilde{Y}_{\alpha,\omega} & =\frac{KL}{d^{2}}\sum_{r=0}^{\frac{d}{K}-1}\sum_{\ell=0}^{\frac{d}{L}-1}\mathbbm{e}^{\frac{2\pi \mathbbm{i}}{d}\left(\omega-rK\right)\left(\alpha-\ell L\right)}\left(F_{d}\left({\bf \widehat{{\bf x}}}\circ S_{\ell L-\alpha}\overline{{\bf \widehat{{\bf x}}}}\right)\right)_{\omega-rK}\left(F_{d}\left({\bf m}\circ S_{\omega-rK}\overline{{\bf m}}\right)\right)_{\ell L-\alpha} +\widetilde{N}_{\alpha,\omega}\\
 &=\frac{KL}{d^{3}}\sum_{r=0}^{\frac{d}{K}-1}\sum_{\ell=0}^{\frac{d}{L}-1}\left(F_{d}\left({\bf \widehat{{\bf x}}}\circ S_{\ell L-\alpha}\overline{{\bf \widehat{{\bf x}}}}\right)\right)_{\omega-rK}\left(F_{d}\left(\widehat{{\bf m}}\circ S_{\alpha-\ell L}\overline{\widehat{{\bf m}}}\right)\right)_{\omega-rK}+\widetilde{N}_{\alpha,\omega}.
\end{align*}
Moreover, if $K=\delta-1+\kappa$ for some $2\le\kappa\leq\delta$ and
$L=\gamma-1+\xi$ for some $1\leq\xi\leq\gamma$, and if $0\leq\omega\leq \kappa-1$ or $K-\kappa-1\leq \omega\leq K-1$ 
and $0\leq \alpha\leq \xi-1$ or $L-\xi+1\leq \alpha\leq L-1,$
the sum above collapses to only one term, so that
\begin{equation*}
\widetilde{Y}_{\alpha,\omega}=
\begin{cases}
\frac{KL}{d^{3}}\left(F_{d}\left(\widehat{{\bf x}}\circ S_{-\alpha}\overline{\widehat{{\bf x}}}\right)\right)_{\omega}\left(F_{d}\left(\widehat{{\bf m}}\circ S_{\alpha}\overline{\widehat{{\bf m}}}\right)\right)_{\omega} +\widetilde{N}_{\alpha,\omega} &\text{ if } 0\leq\alpha\leq\xi-1\text{ and }0\leq\omega\leq\kappa-1\\
\frac{KL}{d^{3}}\left(F_{d}\left(\widehat{{\bf x}}\circ S_{-\alpha}\overline{\widehat{{\bf x}}}\right)\right)_{\omega-K}\left(F_{d}\left(\widehat{{\bf m}}\circ S_{\alpha}\overline{\widehat{{\bf m}}}\right)\right)_{\omega-K} +\widetilde{N}_{\alpha,\omega} &\text{ if } 0\leq\alpha\leq\xi-1\text{ and }\delta\leq\omega\leq K-1\\
\frac{KL}{d^{3}}\left(F_{d}\left(\widehat{{\bf x}}\circ S_{L-\alpha}\overline{\widehat{{\bf x}}}\right)\right)_{\omega}\left(F_{d}\left(\widehat{{\bf m}}\circ S_{\alpha-L}\overline{\widehat{{\bf m}}}\right)\right)_{\omega} +\widetilde{N}_{\alpha,\omega} &\text{ if } \gamma\leq\alpha\leq L-1\text{ and }0\leq\omega\leq\kappa-1\\
\frac{KL}{d^{3}}\left(F_{d}\left(\widehat{{\bf x}}\circ S_{L-\alpha}\overline{\widehat{{\bf x}}}\right)\right)_{\omega-K}\left(F_{d}\left(\widehat{{\bf m}}\circ S_{\alpha-L}\overline{\widehat{{\bf m}}}\right)\right)_{\omega-K} +\widetilde{N}_{\alpha,\omega} &\text{ if } \gamma\leq\alpha\leq L-1\text{ and }\delta\leq\omega\leq K-1
\end{cases}.
\end{equation*}
\end{lem}
The next lemma is similar to the previous one, but replaces the assumptions that $\mathbf{m}$ has compact spatial support and that $\mathbf{x}$ is bandlimited and with the assumption  that $\mathbf{m}$ is bandlimited. It allows us to recover diagonals of $\mathbf{x}\mathbf{x}^*.$
\begin{lem}
\label{thm:generalSumCollapse} Let ${\bf x},{\bf m}\in\mathbb{C}^{d}$%
and assume $\mbox{supp}\left({\bf \widehat{{\bf m}}}\right)\subseteq\left[\rho\right]_{0}$.
Let $L$ divide $d,$ let $Y_{d,L}$  be the $d\times L$ partial measurement matrix defined as in \eqref{eqn:YKLdef}, and let $N_{d,L}$ be the corresponding partial noise matrix. As in the statement of Theorem \ref{thm:subsamplingInFrequencyAndSpace},  let
\begin{equation*}
\widetilde{Y}= F_LY_{d,L}F_d^T\quad\text{and}\quad \widetilde{N}= F_LN_{d,L}F_d^T.
\end{equation*} 
 Then for any $\alpha\in\left[L\right]_{0}$
and $\omega\in\left[d\right]_{0}$,
\begin{equation}\widetilde{Y}_{\alpha,\omega}=\frac{L}{d^{2}}\sum_{\ell=0}^{\frac{d}{L}-1}\left(F_{d}\left({\bf {\bf \widehat{x}}}\circ S_{\ell L-\alpha}\overline{{\bf {\bf \widehat{x}}}}\right)\right)_{\omega}\left(F_{d}\left({\bf \widehat{{\bf m}}}\circ S_{\alpha-\ell L}\overline{{\bf \widehat{{\bf m}}}}\right)\right)_{\omega}+\widetilde{N}_{\alpha,\omega}.\label{eq:FYF_single_aliasing}
\end{equation}
Moreover, if $L=\rho+\kappa-1$ for some $2\le\kappa\leq\rho$,
then for all $\omega\in\left[d\right]_{0}$ and all $\alpha$ such that either $0\leq \alpha\leq \kappa-1$ or $\rho\leq\alpha\leq L-1,$ the sum above reduces to a single term and
\begin{align*}
\widetilde{Y}_{\alpha,\omega}&=\begin{cases}
\frac{L}{d^{2}}\left(F_{d}\left({\bf {\bf \widehat{{\bf x}}}}\circ S_{-\alpha}\overline{{\bf {\bf \widehat{{\bf x}}}}}\right)\right)_{\omega}\left(F_{d}\left({\bf \widehat{{\bf m}}}\circ S_{\alpha}\overline{{\bf \widehat{{\bf m}}}}\right)\right)_{\omega}+\widetilde{N}_{\alpha,\omega}, & \text{\ensuremath{\text{if }0\leq \alpha\leq \kappa-1}}\\
\frac{L}{d^{2}}\left(F_{d}\left({\bf {\bf \widehat{{\bf x}}}}\circ S_{L-\alpha}\overline{{\bf {\bf \widehat{{\bf x}}}}}\right)\right)_{\omega}\left(F_{d}\left({\bf \widehat{{\bf m}}}\circ S_{\alpha-L}\overline{{\bf \widehat{{\bf m}}}}\right)\right)_{\omega}+\widetilde{N}_{\alpha,\omega}, & \text{if }\rho \leq \alpha\leq L-1
\end{cases}.
\end{align*}

\end{lem}
\begin{proof}
 (\ref{eq:FYF_single_aliasing}) follows from  
 Theorem \ref{thm:subsamplingInFrequencyAndSpace} by setting $K=d$ in \eqref{eq:FYF_double_aliasing}. 
To prove the second claim, we note that by the assumption that $\mbox{supp}\left({\bf \widehat{{\bf m}}}\right)\subseteq\left[\rho\right]_{0},$
 ${\bf \widehat{{\bf m}}}\circ S_{\alpha-\ell L}\overline{{\bf \widehat{{\bf m}}}}={\bf 0}$ unless 
\begin{equation}\label{eq:withindelta}
|\alpha-\ell L| < \rho. 
\end{equation}
If $L=\rho-1+\kappa,$ and $0\leq \alpha\leq \kappa-1$, this can only occur if $\ell=0.$ Indeed, if $\ell\geq 1$ then
\begin{equation*}
\alpha-\ell L\leq \alpha-L \leq \kappa-1-(\rho-1+\kappa)=-\rho,
\end{equation*}
and if $\ell\leq -1,$ then
\begin{equation*}
\alpha-\ell L\geq \alpha+L\geq L=\rho-1+\kappa\geq \rho.
\end{equation*}
Therefore, all other terms in the above sum are zero, and the right-hand side of \eqref{eq:FYF_single_aliasing} reduces to the desired result. Likewise, if $\rho\leq \alpha\leq L-1,$ then \eqref{eq:withindelta} can only hold when $\ell=1.$ 
\end{proof}
As in Lemma \ref{thm:generalSumCollapse-2}, the following lemma assumes that the spatial support of $\mathbf{m}$ is contained in an interval of length $\delta$ and allows us to recover diagonals of $\mathbf{\widehat{x}}\mathbf{\widehat{x}}^*.$  However, it differs in that it assumes that $L=d,$ but does not assume that  that $\mathbf{x}$ is $\gamma$-bandlimited.

\begin{lem}
\label{thm:generalSumCollapse-1} Let ${\bf x}, {\bf m}\in\mathbb{C}^{d}$ with $\mbox{supp}\left({\bf m}\right)\subseteq\left[\delta\right]_{0}$.
Let $K$ divide $d,$  let $Y_{K,d}$  be the $K\times d$ partial measurement matrix defined as in \eqref{eqn:YKLdef}, and let $N_{K,d}$ be the corresponding subsampled noise matrix.
As in the statement of Theorem \ref{thm:subsamplingInFrequencyAndSpace},  let
\begin{equation*}
\widetilde{Y}= F_dY_{K,d}F_K^T\quad\text{and}\quad \widetilde{N}= F_dN_{K,d}F_K^T.
\end{equation*} 
Then for any $\alpha\in\left[d\right]_{0}$ and $\omega\in\left[K\right]_{0}$,
\begin{align*}
\widetilde{Y}_{\alpha,\omega} & =\frac{K}{d^{2}}\sum_{r=0}^{\frac{d}{K}-1}\left(F_{d}\left(\widehat{{\bf x}}\circ S_{-\alpha}\overline{{\bf \widehat{{\bf x}}}}\right)\right)_{\omega-rK}\left(F_{d}\left(\widehat{{\bf m}}\circ S_{\alpha}\overline{\widehat{{\bf m}}}\right)\right)_{\omega-rK}+\widetilde{N}_{\alpha,\omega}\\
 & =K\sum_{r=0}^{\frac{d}{K}-1}\left(F_{d}\left({\bf x}\circ S_{\omega-rK}\overline{{\bf x}}\right)\right)_{\alpha}\left(F_{d}\left({\bf m}\circ S_{\omega-rK}\overline{{\bf m}}\right)\right)_{-\alpha}+\widetilde{N}_{\alpha,\omega}.
\end{align*}
Moreover, if $K=\delta-1+\kappa$ for some $2\le\kappa\leq\delta$, and
if  either $0\leq \omega\leq \kappa-1$ or $\delta\leq\omega\leq K-1,$ 
then for all $\alpha\in\left[d\right]_{0}$, the sum above reduces
to only one term and

\[
\widetilde{Y}_{\alpha,\omega}=\begin{cases}
K\left(F_{d}\left({\bf x}\circ S_{\omega}\overline{{\bf x}}\right)\right)_{\alpha}\left(F_{d}\left({\bf m}\circ S_{\omega}\overline{{\bf m}}\right)\right)_{-\alpha} +\widetilde{N}_{\alpha,\omega}, & \text{if }0\leq \omega\leq \kappa-1\\
K\left(F_{d}\left({\bf x}\circ S_{\omega-K}\overline{{\bf x}}\right)\right)_{\alpha}\left(F_{d}\left({\bf m}\circ S_{\omega-K}\overline{{\bf m}}\right)\right)_{-\alpha} +\widetilde{N}_{\alpha,\omega}, & \text{if }\delta\leq \omega\leq K-1
\end{cases}.
\]

\end{lem}

\begin{remark}
For convenience, in Lemmas \ref{thm:generalSumCollapse-2}, \ref{thm:generalSumCollapse}, and \ref{thm:generalSumCollapse-1} we have assumed that the support of $\widehat{\mathbf{m}}, \mathbf{m}$ or  $\widehat{\mathbf{x}},$ were contained in the first $\rho, \delta,$ or $\gamma$ entries. However, inspecting the proofs we see these results remain valid if these intervals are replaced with any other intervals of the same length.
\end{remark}

\section{Recovery Guarantees}
\label{sec:recGuar}

%

In this section, we will present two algorithms which allow us to reconstruct $\mathbf{x}$ from our matrix of  noisy measurements $Y_{K,L}$ and prove Theorems \ref{thm:Algorithm_2_guarantees} and \ref{thm:Algorithm_1_guarantees}, presented in the introduction, which guarantee that these algorithms converge. Before providing the proofs of these theorems, we will first state two propositions which show that it is possible to design masks in such a way that the mask dependent constants $\mu_1$ and $\mu_2$ are nonzero. For proofs of these propositions, please see the appendix.

%

\begin{proposition}
\label{prop:mu_condition} Let ${\bf m}\in\mathbb{C}^{d}$ be
bandlimited with $\mbox{supp}\left(\widehat{\mathbf{m}}\right)\subseteq\left[\rho\right]_{0}$, so that its Fourier transform
may be written as
\[
\widehat{{\bf m}}=\left(a_{0}\mathbbm{e}^{\mathbbm{i}\theta_{0}},\ldots,a_{\rho-1}\mathbbm{e}^{\mathbbm{i}\theta_{\rho-1}},0,\ldots,0\right)^{T}
\]
for some real numbers $a_{0},\dots,a_{\rho-1}$. As in \eqref{eq:mu}, let 
\begin{equation*}
\mu_1= \min_{_{\substack{\left|p\right|\leq\kappa-1\\
q\in\left[d\right]_{0}
}
}}\left|F_{d}\left(\widehat{\mathbf{m}}\circ S_{p}\overline{\widehat{\mathbf{m}}}\right){}_{q}\right|,
\end{equation*}
for some $2\leq\kappa\leq\rho$. If 
\begin{equation}
|a_{0}|>\left(\rho-1\right)|a_{1}|,\label{eqn: a0big}
\end{equation}
and 
\begin{equation}
|a_{1}|\geq|a_{2}|\geq\cdots\geq|a_{\rho-1}|>0,\label{eqn: decreasingamplitudes}
\end{equation}
then $\mu_1>0.$ 
\end{proposition}

\begin{proposition}
\label{prop:mu_condition2} Let ${\bf m}\in\mathbb{C}^{d}$ be
a compactly supported mask with $\mbox{supp}\left(\mathbf{m}\right)\subseteq\left[\delta\right]_{0}$, given by
\[
{\bf m}=\left(a_{0}\mathbbm{e}^{\mathbbm{i}\theta_{0}},\ldots,a_{\delta-1}\mathbbm{e}^{\mathbbm{i}\theta_{\delta-1}},0,\ldots,0\right)^{T}
\]
for some real numbers $a_{0},\dots,a_{\delta-1}$. As in \eqref{eq:mu2}, let 
\begin{equation*}
\mu_2= \min_{_{\substack{\left|p\right|\leq\gamma-1\\
\left|q\right|\leq\delta-1
}
}}\left|F_{d}\left(\widehat{{\bf m}}\circ S_{p}\overline{\widehat{{\bf m}}}\right)_{q}\right|\label{eq:mu-2}
\end{equation*}
for some $1\leq\gamma\leq2\delta-1$. If 
\begin{equation}
|a_{0}|>\left(\delta-1\right)|a_{1}|,\label{eqn: a0big2}
\end{equation}
and 
\begin{equation}
|a_{1}|\geq|a_{2}|\geq\cdots\geq|a_{\delta-1}|>0,\label{eqn: decreasingamplitudes2}
\end{equation}
then $\mu_2>0.$ 
\end{proposition}

\textbf{}
\begin{algorithm}\label{firstalg}
\begin{raggedright}
\textbf{Inputs}
\par\end{raggedright}
\begin{enumerate}
\item \begin{raggedright}
$d\times L$ noisy measurement matrix $Y_{d,L}\in\mathbb{R}^{d\times L}$ with entries 

\begin{equation*}
\left(Y_{d,L}\right)_{k,\ell}=\left|\sum_{n=0}^{d-1}x_{n}m_{n-\ell\frac{L}{d}}\mathbbm{e}^{-\frac{2\pi \mathbbm{i}nk}{d}}\right|^{2}+\left(N_{d,L}\right)_{k,\ell},\ k\in[d]_0,\ell\in\frac{d}{L}\left[L\right]_{0}.
\end{equation*}
\par\end{raggedright}
\item \begin{raggedright}
Bandlimited mask  ${\bf m}\in\mathbb{C}^{d}$ with
$\mbox{supp}\left(\widehat{\mathbf{m}}\right)\subseteq\left[\rho\right]_{0}$ for some $\rho<\frac{d}{2}.$
\par\end{raggedright}
\end{enumerate}
\begin{raggedright}
\textbf{Steps}
\par\end{raggedright}
\begin{enumerate}
\item Let $\kappa=L-\rho+1,$ and for $1-\kappa\leq \alpha\leq \kappa-1$ estimate  $F_{d}\left({\bf \widehat{{\bf x}}}\circ S_{\alpha}\overline{\widehat{{\bf x}}}\right)$ by
\begin{equation*}
F_{d}\left({\bf {\bf \widehat{{\bf x}}}}\circ S_{\alpha}\overline{{\bf {\bf \widehat{{\bf x}}}}}\right)  \approx\begin{cases}\frac{d^{2}\left(F_dY_{d,L}F_L^T\right)_{-\alpha}}{L F_{d}\left({\bf \widehat{{\bf m}}}\circ S_{-\alpha}\overline{{\bf \widehat{{\bf m}}}}\right)}\ \ &\text{if }1-\kappa\leq \alpha\leq0\\
\frac{d^{2}\left(F_dY_{d,L}F_L^T\right)_{L-\alpha}}{L F_{d}\left({\bf \widehat{{\bf m}}}\circ S_{-\alpha}\overline{{\bf \widehat{{\bf m}}}}\right)}\ \ &\text{if } 1\leq\alpha\leq\kappa-1
\end{cases}.
\end{equation*}

\item Invert the Fourier transforms above to recover estimates of the $\left(2\kappa-1\right)$
vectors ${\bf {\bf \widehat{{\bf x}}}}\circ S_{\alpha}\overline{{\bf {\bf \widehat{{\bf x}}}}}$. 
\item Organize these vectors into a banded matrix, $C_{2\kappa-1}\left(Y_{2\kappa-1}\right)$ as described in \eqref{eqn:Y2kappa} (see also \eqref{eq:circulant_operator}).
\item Hermitianize the matrix above: $C_{2\kappa-1}\left(Y_{2\kappa-1}\right)\mapsfrom\frac{1}{2}\left(C_{2\kappa-1}\left(Y_{2\kappa-1}\right)+C_{2\kappa-1}\left(Y_{2\kappa-1}\right)^{*}\right)$.
\item Estimate $\left|\widehat{{\bf x}}\right|$ from the main diagonal of $C_{2\kappa-1}\left(Y_{2\kappa-1}\right)$.
\item Normalize $C_{2\kappa-1}\left(Y_{2\kappa-1}\right)$ componentwise
to form $\widetilde{Y}_{2\kappa-1}$. 
\item Compute $\mathbf{v_1}$ the leading normalized eigenvector of $\widetilde{Y}_{2\kappa-1}.$ 
\end{enumerate}
\begin{raggedright}
\textbf{Output}
\par\end{raggedright}
\begin{raggedright}
  ${\bf x}_e\coloneqq F_d^{-1}\widehat{{\bf x}}_{e},$ an estimate of ${\bf x},$ where  $\widehat{{\bf x}}_e$
is given componentwise by
\[
\left(\widehat{{\bf x}}_{e}\right)_{j}\coloneqq \sqrt{\left(C_{2\kappa-1}\left(Y_{2\kappa-1}\right)\right)_{j,j}}\left(\mathbf{v_1}\right)_{j}.
\]
\par\end{raggedright}
\raggedright{}\textbf{\caption{\textbf{Wigner Deconvolution and Angular Synchronization for Bandlimited
Masks}}
}
\end{algorithm}

We will now prove Theorem \ref{thm:Algorithm_1_guarantees}, which we restate below for the  convenience of the reader.
\Algone*

\begin{proof}[The Proof of Theorem \ref{thm:Algorithm_1_guarantees}]
Let ${\bf x},{\bf m}\in\mathbb{C}^{d},$  
$\mbox{supp}\left({\bf \widehat{{\bf m}}}\right)=\left[\rho\right]_{0},$ and let $\kappa=L-\rho+1.$ Then by Lemma \ref{thm:generalSumCollapse}, %
 if $0\leq \beta\leq \kappa-1,$ then
\begin{align*}
F_{d}\left({\bf {\bf \widehat{{\bf x}}}}\circ S_{-\beta}\overline{{\bf {\bf \widehat{{\bf x}}}}}\right) & =\frac{d^{2}}{L}\frac{\left(F_{d}Y_{d,L}F_{L}^{T}\right)_{\beta}}{F_{d}\left({\bf \widehat{{\bf m}}}\circ S_{\beta}\overline{{\bf \widehat{{\bf m}}}}\right)}-\frac{d^{2}}{L}\frac{\left(F_{d}N_{d,L}F_{L}^{T}\right)_{\beta}}{F_{d}\left({\bf \widehat{{\bf m}}}\circ S_{\beta}\overline{{\bf \widehat{{\bf m}}}}\right)},\end{align*}
and therefore
\begin{align*}
{\bf {\bf \widehat{{\bf x}}}}\circ S_{-\beta}\overline{{\bf {\bf \widehat{{\bf x}}}}} & +\frac{d^{2}}{L}F_{d}^{-1}\left(\frac{\left(F_{d}N_{d,L}F_{L}^{T}\right)_{\beta}}{F_{d}\left({\bf \widehat{{\bf m}}}\circ S_{\beta}\overline{{\bf \widehat{{\bf m}}}}\right)}\right)=\frac{d^{2}}{L}F_{d}^{-1}\left(\frac{\left(F_{d}Y_{d,L}F_{L}^{T}\right)_{\beta}}{F_{d}\left({\bf \widehat{{\bf m}}}\circ S_{\beta}\overline{{\bf \widehat{{\bf m}}}}\right)}\right).\ 
\end{align*}
Substituting $\alpha=-\beta,$ we see
\begin{align}\label{eqn:diag1}
{\bf {\bf \widehat{{\bf x}}}}\circ S_{\alpha}\overline{{\bf {\bf \widehat{{\bf x}}}}} & +\frac{d^{2}}{L}F_{d}^{-1}\left(\frac{\left(F_{d}N_{d,L}F_{L}^{T}\right)_{-\alpha}}{F_{d}\left({\bf \widehat{{\bf m}}}\circ S_{-\alpha}\overline{{\bf \widehat{{\bf m}}}}\right)}\right)=\frac{d^{2}}{L}F_{d}^{-1}\left(\frac{\left(F_{d}Y_{d,L}F_{L}^{T}\right)_{-\alpha}}{F_{d}\left({\bf \widehat{{\bf m}}}\circ S_{-\alpha}\overline{{\bf \widehat{{\bf m}}}}\right)}\right)
\end{align} for all $1-\kappa\leq \alpha\leq 0.$
Likewise, for $\rho\leq \beta\leq L-1,$
\begin{align*}
{\bf {\bf \widehat{{\bf x}}}}\circ S_{L-\beta}\overline{{\bf {\bf \widehat{{\bf x}}}}} & +\frac{d^{2}}{L}F_{d}^{-1}\left(\frac{\left(F_{d}N_{d,L}F_{L}^{T}\right)_{\beta}}{F_{d}\left({\bf \widehat{{\bf m}}}\circ S_{\beta-L}\overline{{\bf \widehat{{\bf m}}}}\right)}\right)=\frac{d^{2}}{L}F_{d}^{-1}\left(\frac{\left(F_{d}Y_{d,L}F_{L}^{T}\right)_{\beta}}{F_{d}\left({\bf \widehat{{\bf m}}}\circ S_{\beta-L}\overline{{\bf \widehat{{\bf m}}}}\right)}\right), 
\end{align*}
so, since $L=\rho+\kappa-1,$ substituting $\alpha=L-\beta$ implies
\begin{align}\label{eqn:diag2}
{\bf {\bf \widehat{{\bf x}}}}\circ S_{\alpha}\overline{{\bf {\bf \widehat{{\bf x}}}}} & +\frac{d^{2}}{L}F_{d}^{-1}\left(\frac{\left(F_{d}N_{d,L}F_{L}^{T}\right)_{L-\alpha}}{F_{d}\left({\bf \widehat{{\bf m}}}\circ S_{-\alpha}\overline{{\bf \widehat{{\bf m}}}}\right)}\right)=\frac{d^{2}}{L}F_{d}^{-1}\left(\frac{\left(F_{d}Y_{d,L}F_{L}^{T}\right)_{L-\alpha}}{F_{d}\left({\bf \widehat{{\bf m}}}\circ S_{-\alpha}\overline{{\bf \widehat{{\bf m}}}}\right)}\right) 
\end{align}
for all $1\leq \alpha\leq \kappa-1.$

In order to write the equations above in a compact form, we will construct three $d\times 2\kappa-1$ matrices, $X_{2\kappa-1}, N_{2\kappa-1}$, and $Y_{2\kappa-1}.$ As in Section \ref{Sec:BasicProps}, for notatational convenience, we will index the columns of these matrices from $-\kappa+1$ to $\kappa-1$ so that column zero is the middle column.  For $-\kappa+1\leq \alpha\leq \kappa-1,$ we let the $\alpha$-th column of $X_{2\kappa-1}$  be the diagonal band of $\widehat{\mathbf{x}}\widehat{\mathbf{x}}^*$ which is $\alpha$ terms off of the main diagonal, i.e.
\begin{equation*}
(X_{2\kappa-1})_{\alpha} = \widehat{\mathbf{x}}\circ S_\alpha\overline{\widehat{\mathbf{x}}},
\end{equation*} 
and we define the columns of $N_{2\kappa-1}$ and $Y_{2\kappa-1}$ by
\begin{equation*}
(N_{2\kappa-1})_{\alpha}=\begin{cases}
&\frac{d^{2}}{L}F_{d}^{-1}\left(\frac{\left(F_{d}N_{d,L}F_{L}^{T}\right)_{-\alpha}}{F_{d}\left({\bf \widehat{{\bf m}}}\circ S_{-\alpha}\overline{{\bf \widehat{{\bf m}}}}\right)}\right) \:\:\: \text{if }-\kappa+1\leq \alpha\leq 0\\
&\frac{d^{2}}{L}F_{d}^{-1}\left(\frac{\left(F_{d}N_{d,L}F_{L}^{T}\right)_{L-\alpha}}{F_{d}\left({\bf \widehat{{\bf m}}}\circ S_{-\alpha}\overline{{\bf \widehat{{\bf m}}}}\right)}\right) \: \text{if } 1\leq \alpha\leq \kappa-1
\end{cases},
\end{equation*}
and
\begin{equation}\label{eqn:Y2kappa}
(Y_{2\kappa-1})_{\alpha}=\begin{cases}
&\frac{d^{2}}{L}F_{d}^{-1}\left(\frac{\left(F_{d}Y_{d,L}F_{L}^{T}\right)_{-\alpha}}{F_{d}\left({\bf \widehat{{\bf m}}}\circ S_{-\alpha}\overline{{\bf \widehat{{\bf m}}}}\right)}\right) \: \text{if }-\kappa+1\leq \alpha\leq 0\\
&\frac{d^{2}}{L}F_{d}^{-1}\left(\frac{\left(F_{d}Y_{d,L}F_{L}^{T}\right)_{L-\alpha}}{F_{d}\left({\bf \widehat{{\bf m}}}\circ S_{-\alpha}\overline{{\bf \widehat{{\bf m}}}}\right)}\right) \:\:\: \text{if } 1\leq \alpha\leq \kappa-1
\end{cases}.
\end{equation}
By construction, \eqref{eqn:diag1} and \eqref{eqn:diag2} imply
\begin{equation}
Y_{2\kappa-1}=X_{2\kappa-1}+N_{2\kappa-1}.\label{eq:noncirculant_meas}
\end{equation}
Using the fact $\frac{1}{\sqrt{d}}F_d$ is unitary,
we see 
\begin{align}
 \left\Vert N_{2\kappa-1}\right\Vert _{F}^2& \leq\frac{d^{4}}{L^2}\cdot\frac{1}{d}\frac{\left\Vert F_{d}N_{d,L}F_{L}^{T}\right\Vert _{F}^2}{{\displaystyle \mu_1^2}}\nonumber\\
 & \leq\frac{d^{4}}{L\mu_1^2}\left\Vert N_{d,L}\right\Vert _{F}^2, \label{eq:N_2kappa-1_bound}
\end{align}
where $\mu_1$ is as in (\ref{eq:mu}). 
%
Let $H:\mathbb{C}^{d\times d}\rightarrow\mathbb{C}^{d\times d}$ be the Hermitianizing operator
\begin{equation}\label{eqn:herm}
H(M)=\frac{M+M^*}{2},
\end{equation}
and note that $\|H(M)\|_F\leq \|M\|_F.$ Since operator $C_{2\kappa-1}$ defined in (\ref{eq:circulant_operator}) is linear and $C_{2\kappa-1}\left(X_{2\kappa-1}\right)$ is Hermitian, (\ref{eq:noncirculant_meas}) 
implies
\[
C_{2\kappa-1}\left(X_{2\kappa-1}\right)=H\left(C_{2\kappa-1}\left(Y_{2\kappa-1}\right)\right)-H\left(C_{2\kappa-1}\left(N_{2\kappa-1}\right)\right).
\]
Let $\mbox{sgn}:\mathbb{C}\to\mathbb{C}$ be the signum function,
\[
\mbox{sgn}\left(z\right)\coloneqq\begin{cases}
\frac{z}{\left|z\right|}, & \text{if }z\neq0\\
\text{if }1, & z=0
\end{cases},
\]
and let $\widetilde{X}_{2\kappa-1}$ and $\widetilde{Y}_{2\kappa-1}$,
be the (componentwise) normalized
versions of $C_{2\kappa-1}\left(X_{2\kappa-1}\right)$ and $H\left(C_{2\kappa-1}\left(Y_{2\kappa-1}\right)\right)$,
respectively, i.e.,
\begin{equation*}
\widetilde{X}_{2\kappa-1}  \coloneqq \mbox{sgn}\left(C_{2\kappa-1}\left(\left(X_{2\kappa-1}\right)\right)\right),\quad
\text{and}\quad\widetilde{Y}_{2\kappa-1}  \coloneqq \mbox{sgn}\left(HC_{2\kappa-1}\left(\left(Y_{2\kappa-1}\right)\right)\right),
\end{equation*}
and note that 
\begin{equation}\label{eqn:samesign}
\left(\widetilde{Y}_{2\kappa-1}\right)_{j,k}=
\mbox{sgn}\left(\left(H\left(C_{2\kappa-1}\left(Y_{2\kappa-1}\right)\right)\right)_{j,k}\right)= \mbox{sgn}\left(\frac{\left(H\left(\left(C_{2\kappa-1}\left(Y_{2\kappa-1}\right)\right)\right)\right)_{j,k}}{\left|\left(C_{2\kappa-1}\left(X_{2\kappa-1}\right)\right)_{j,k}\right|}\right).
\end{equation}
For all $j$ and $k,$ we have that 
\begin{equation*}
\frac{1}{\left|\left(C_{2\kappa-1}\left(X_{2\kappa-1}\right)\right)_{j,k}\right|}\leq\frac{1}{\min|\widehat{\x}|^2}.
\end{equation*}
Therefore, 
we  can apply \eqref{eqn:samesign} and the fact that for all $z_1, z_2\in\mathbb{C},$ 
\begin{equation*}
\left|\frac{z_2}{|z_1|}-\mbox{sgn}\left(\frac{z_2}{|z_1|}\right)\right|=\left|\frac{z_2}{|z_1|}-\mbox{sgn}\left(z_2\right)\right|\leq \left|\mbox{sgn}(z_1)-\frac{z_2}{|z_1|}\right|
\end{equation*}
to see
\begin{align}
&\left|\left(\widetilde{X}_{2\kappa-1}\right)_{j,k}-\left(\widetilde{Y}_{2\kappa-1}\right)_{j,k}\right|\nonumber\\
  =&\left|\left(\widetilde{X}_{2\kappa-1}\right)_{j,k}-\mbox{sgn}\left(\frac{\left(H\left(C_{2\kappa-1}\left(Y_{2\kappa-1}\right)\right)\right)_{j,k}}{\left|\left(C_{2\kappa-1}\left(X_{2\kappa-1}\right)\right)_{j,k}\right|}\right)\right|\nonumber\\
  \leq&\left|\left(\widetilde{X}_{2\kappa-1}\right)_{j,k}-\frac{\left(H\left(C_{2\kappa-1}\left(Y_{2\kappa-1}\right)\right)\right)_{j,k}}{\left|\left(C_{2\kappa-1}\left(X_{2\kappa-1}\right)\right)_{j,k}\right|}\right|+\left|\frac{\left(H\left(C_{2\kappa-1}\left(Y_{2\kappa-1}\right)\right)\right)_{j,k}}{\left|\left(C_{2\kappa-1}\left(X_{2\kappa-1}\right)\right)_{j,k}\right|}-\mbox{sgn}\left(\frac{\left(H\left(C_{2\kappa-1}\left(Y_{2\kappa-1}\right)\right)\right)_{j,k}}{\left|\left(C_{2\kappa-1}\left(X_{2\kappa-1}\right)\right)_{j,k}\right|}\right)\right|\nonumber\\
  \leq&2\left|\left(\widetilde{X}_{2\kappa-1}\right)_{j,k}-\frac{\left(H\left(C_{2\kappa-1}\left(Y_{2\kappa-1}\right)\right)\right)_{j,k}}{\left|\left(C_{2\kappa-1}\left(X_{2\kappa-1}\right)\right)_{j,k}\right|}\right|\nonumber\\
  =&2\frac{\left|\left(H\left(C_{2\kappa-1}\left(N_{2\kappa-1}\right)\right)\right)_{j,k}\right|}{\left|\left(C_{2\kappa-1}\left(X_{2\kappa-1}\right)\right)_{j,k}\right|}\nonumber\\
  \leq&\frac{2}{\min|\widehat{\x}|^2}\left|\left(H\left(C_{2\kappa-1}\left(N_{2\kappa-1}\right)\right)\right)_{j,k}\right|\label{eqn:useSeps},
\end{align}
for all $j,k\in[d]_0.$
Thus,  by
 (\ref{eq:N_2kappa-1_bound}) and the fact that  $\left\Vert \widetilde{X}_{2\kappa-1}\right\Vert _{F}=\sqrt{\left(2\kappa-1\right)d},$
we see that 
\begin{align*}
\left\Vert \widetilde{Y}_{2\kappa-1}-\widetilde{X}_{2\kappa-1}\right\Vert _{F} & \leq\frac{2}{\min|\widehat{\x}|^2}\|H\left(C_{2\kappa-1}\left(N_{2\kappa-1}\right)\right)\|_F\\
 & \leq\frac{2}{\min|\widehat{\x}|^2}\left\Vert N_{2\kappa-1}\right\Vert _{F}\\
 & \leq\frac{2}{\mu_1\min|\widehat{\x}|^2}\frac{d^{2}}{\sqrt{L}}\left\Vert N_{d,L}\right\Vert _{F}\\
&\leq C\frac{d^{3/2}\|N_{d,L}\|_F}{\mu_1\min|\widehat{\x}|^2\sqrt{\kappa L}}\|X_{2\kappa-1}\|_F.
\end{align*}
 Therefore, by Corollary 2 of \cite{iwen2018phase},
we have
\begin{align}
\min_{\phi\in\left[0,2\pi\right]}\left\Vert \mbox{sgn}\left({\widehat{\x}}\right)-\mathbbm{e}^{\mathbbm{i}\phi}\mbox{sgn}\left(\mathbf{v_1}\right)\right\Vert _{2}
&\leq C\frac{d^{3/2}\|N_{d,L}\|_F}{\mu_1\min|\widehat{\x}|^2\sqrt{\kappa L}}\frac{d^{\frac{5}{2}}}{\kappa^{2}}\nonumber\\
&=C\frac{d^{4}\|N_{d,L}\|_F}{L^{1/2}\mu_1\kappa^{5/2}\min|\widehat{\x}|^2},\label{eq:bound_true_minus_recovered_phases}
\end{align}
where $\frac{{\bf \widehat{x}}}{\left|\widehat{{\bf x}}\right|}$
is the vector of true phases of $\widehat{{\bf x}}$, and $\mathbf{v_1}$
is the lead eigenvector of $\widetilde{Y}_{2\kappa-1}$.

As in Algorithm 1, define $\widehat{{\bf x}}_{e}$ by
\[
\left(\widehat{{\bf x}}_{e}\right)_{j}=\sqrt{\left(C_{2\kappa-1}\left(Y_{2\kappa-1}\right)\right)_{j,j}}\cdot\mbox{sgn}\left(\mathbf{v_1}\right)_{j}
\]
for $j\in[d]_0.$ Lemma 3 of \cite{iwen2016fast} implies that
\begin{equation*}
\big\Vert \left|\widehat{{\bf x}}\right|-\left|\widehat{{\bf x}}_{e}\right|\big\Vert _{\infty}^2 \leq C \|N_{2\kappa-1}\|_\infty,
 \end{equation*}
and so
\begin{equation*}
\big\Vert \left|\widehat{{\bf x}}\right|-\left|\widehat{{\bf x}}_{e}\right|\big\Vert _{2} \leq C \sqrt{d\|N_{2\kappa-1}\|_\infty} \leq C \sqrt{d\|N_{2\kappa-1}\|_F}.
 \end{equation*}
Therefore, by  \eqref{eq:N_2kappa-1_bound}, we see 
\begin{align*}
\min_{\phi\in\left[0,2\pi\right]}\left\Vert \widehat{{\bf x}}-\mathbbm{e}^{\mathbbm{i}\phi}\widehat{{\bf x}}_{e}\right\Vert _{2} & =\min_{\phi\in\left[0,2\pi\right]}\left\Vert \left|\widehat{{\bf x}}\right|\circ\mbox{sgn}\left(\widehat{\x}\right)-\left|\widehat{{\bf x}}_{e}\right|\circ \mathbbm{e}^{\mathbbm{i}\phi}\mbox{sgn}\left(\widehat{\x}_e\right)\right\Vert _{2}\\
 & \leq\min_{\phi\in\left[0,2\pi\right]}\left(\left\Vert \left|\widehat{{\bf x}}\right|\circ\mbox{sgn}\left({\widehat{\x}}\right)-\left|\widehat{{\bf x}}\right|\circ \mathbbm{e}^{\mathbbm{i}\phi}\mbox{sgn}\left(\widehat{\x}_e\right)\right\Vert _{2}+\left\Vert \left|\widehat{{\bf x}}\right|\circ \mathbbm{e}^{\mathbbm{i}\phi}\mbox{sgn}\left(\widehat{\x}_e\right)-\left|\widehat{{\bf x}}_{e}\right|\circ \mathbbm{e}^{\mathbbm{i}\phi}\mbox{sgn}\left(\widehat{\x}_e\right)
\right\Vert _{2}\right).\\
& =\min_{\phi\in\left[0,2\pi\right]}\left(\left\Vert \left|\widehat{{\bf x}}\right|\circ\mbox{sgn}\left(\widehat{\x}\right)-\left|\widehat{{\bf x}}\right|\circ \mathbbm{e}^{\mathbbm{i}\phi}\mbox{sgn}\left(\widehat{\x}_e\right)\right\Vert _{2}\right)+\big\Vert \left|\widehat{{\bf x}}\right|-\left|\widehat{{\bf x}}_{e}\right|\big\Vert _{2}\\
&\leq\left\Vert \widehat{{\bf x}}\right\Vert _{\infty}\left(\min_{\phi\in\left[0,2\pi\right]}\left\Vert \mbox{sgn}\left(\widehat{\x}\right)-\mathbbm{e}^{\mathbbm{i}\phi}\mbox{sgn}\left(\widehat{\x}_e\right)\right\Vert _{2}\right)+C\sqrt{d \left\Vert N_{2\kappa-1}\right\Vert _{F}}\\
&\leq\left\Vert \widehat{{\bf x}}\right\Vert _{\infty}\left(\min_{\phi\in\left[0,2\pi\right]}\left\Vert \mbox{sgn}\left(\widehat{\x}\right)-\mathbbm{e}^{\mathbbm{i}\phi}\mbox{sgn}\left(\widehat{\x}_e\right)\right\Vert _{2}\right)+C\sqrt{\frac{d^{3}}{\sqrt{L}\mu_1}\left\Vert N_{d,L}\right\Vert _{F}}.
\end{align*}
Together with (\ref{eq:bound_true_minus_recovered_phases})
this yields
\begin{equation*}
\min_{\phi\in\left[0,2\pi\right]}\left\Vert \widehat{{\bf x}}-\mathbbm{e}^{\mathbbm{i}\phi}\widehat{{\bf x}}_{e}\right\Vert _{2}\leq C\frac{d^{4}\left\Vert \widehat{{\bf x}}\right\Vert _{\infty}\left\Vert N_{d,L}\right\Vert _{F}}{L^{\frac{1}{2}}\mu_1\kappa^{\frac{5}{2}}\min\left|\widehat{{\bf x}}\right|^{2}}+C'\frac{d^{3/2}}{L^{1/4}}\sqrt{\frac{\left\Vert N_{d,L}\right\Vert _{F}}{\mu_1}.}
\end{equation*}
\eqref{equ:NewErrorCDP} now follows from the fact that $\|\widehat{\mathbf{x}}\|_2=\sqrt{d}\|\x\|_2$  for all $\mathbf{x}\in\mathbb{C}^d.$ 
\end{proof}
Theorem \ref{thm:Algorithm_2_guarantees}, restated below, provides recovery guarantees for Algorithm 2 under the assumptions  that $\mathbf{m}$ is compactly supported in space and that $\mathbf{x}$ is bandlimited. The proof is somewhat similar to the proof of Theorem \ref{thm:Algorithm_1_guarantees} but uses Lemma \ref{thm:generalSumCollapse-2} in place of Lemma \ref{thm:generalSumCollapse} and uses the Lemma 8 of \cite{iwen2018phase} during the angular synchronization step. 

\textbf{}
\begin{algorithm}[htbp] \label{alg:secondalg}
\begin{raggedright}
\textbf{Inputs}
\par\end{raggedright}
\begin{enumerate}
\item $K\times L$ noisy measurement matrix, $Y_{K,L}\in\mathbb{R}^{K\times L},$ with entries 
\begin{equation*}
\left(Y_{K,L}\right)_{k,\ell}=\left|\sum_{n=0}^{d-1}x_{n}m_{n-\ell\frac{L}{d}}\mathbbm{e}^{-\frac{2\pi \mathbbm{i}nkK}{d^2}}\right|^{2}+\left(N_{K,L}\right)_{k,\ell},\ k\in\frac{d}{K}\left[K\right]_{0}, \ell\in\frac{d}{L}\left[L\right]_{0}.
\end{equation*}
\item Compactly supported mask  ${\bf m}\in\mathbb{C}^{d}$.
\item Integers $\delta$ and $\gamma$, such that $\mbox{supp}\left({\bf {\bf m}}\right)\subseteq\left[\delta\right]_{0},$  $\mbox{supp}\left({\bf \widehat{{\bf x}}}\right)\subseteq\left[\gamma\right]_{0},$ and $\gamma\leq 2\delta-1<d.$
\end{enumerate}
\begin{raggedright}
\textbf{Steps}
\par\end{raggedright}
\begin{enumerate}
\item Ensure $L=2\gamma-1$ and $K=2\delta-1$. 
\item Estimate $\left(F_{d}\left({\bf \widehat{{\bf x}}}\circ S_{\alpha}\overline{\widehat{{\bf x}}}\right)\right)_{\beta}$
 for $\left|\sigma\right|\leq\gamma-1$ and $\left|\beta\right|\leq\delta-1$
by
\begin{equation*}
\left(F_{d}\left(\widehat{{\bf x}}\circ S_{\alpha}\overline{\widehat{{\bf x}}}\right)\right)_{\omega}\approx
\begin{cases}
\frac{d^{3}}{KL}\frac{\left(F_{L}Y_{K,L}^{T}F_{K}^{T}\right)_{-\alpha,\omega}}{\left(F_{d}\left(\widehat{{\bf m}}\circ S_{-\alpha}\overline{\widehat{{\bf m}}}\right)\right)_{\omega}} & \text{if } 1-\gamma\leq\alpha\leq0 \text{ and }1-\delta\leq\omega\leq -1 \\
\frac{d^{3}}{KL}\frac{\left(F_{L}Y_{K,L}^{T}F_{K}^{T}\right)_{-\alpha,\omega+K}}{\left(F_{d}\left(\widehat{{\bf m}}\circ S_{-\alpha}\overline{\widehat{{\bf m}}}\right)\right)_{\omega}} & \text{if } 1-\gamma\leq\alpha\leq 0\text{ and }0\leq\omega\leq\delta-1 \\
\frac{d^{3}}{KL}\frac{\left(F_{L}Y_{K,L}^{T}F_{K}^{T}\right)_{L-\alpha,\omega}}{\left(F_{d}\left(\widehat{{\bf m}}\circ S_{-\alpha}\overline{\widehat{{\bf m}}}\right)\right)_{\omega}} & \text{if }  1\leq\alpha\leq\gamma-1 \text{ and }1-\delta\leq\omega\leq -1\\
\frac{d^{3}}{KL}\frac{\left(F_{L}Y_{K,L}^{T}F_{K}^{T}\right)_{L-\alpha,\omega+K}}{\left(F_{d}\left(\widehat{{\bf m}}\circ S_{-\alpha}\overline{\widehat{{\bf m}}}\right)\right)_{\omega}} & \text{if }1\leq\alpha\leq\gamma-1 \text{ and }0\leq\omega\leq\delta-1 
\end{cases}.
\end{equation*}

\item Organize the $\left(2\delta-1\right)\cdot\left(2\gamma-1\right)$ values
of $\left(F_{d}\left({\bf \widehat{{\bf x}}}\circ S_{\sigma}\overline{\widehat{{\bf x}}}\right)\right)_{\beta}$
for $\left|\sigma\right|\leq\delta-1$ and $\left|\beta\right|\leq\gamma-1$
in a matrix $V\in\mathbb{C}^{\left(2\delta-1\right)\times\left(2\gamma-1\right)}$ as specified in \eqref{eq:Vdef}.
\item Estimate $A\approx W^{\dagger}V\in\mathbb{C}^{\gamma\times\left(2\gamma-1\right)}$,
where
\begin{align*}
W_{j,k} & \coloneqq\mathbbm{e}^{-\frac{2\pi \mathbbm{i}\left(j-\delta+1\right)k}{d}},\ \text{for }j\in\left[2\delta-1\right]_{0}, k\in\left[\gamma\right]_{0},\nonumber \\
W^{\dagger} & \coloneqq\left(W^{*}W\right)^{-1}W^{*}\in\mathbb{C}^{\gamma\times\left(2\delta-1\right)},\nonumber \\
A & \coloneqq\left[\begin{array}{ccccccc}
0 & \cdots & 0 & \left|\widehat{x}_{0}\right|^{2} & \widehat{x}_{0}\overline{\widehat{x}_{1}} & \cdots & \widehat{x}_{0}\overline{\widehat{x}_{\gamma-1}}\\
0 & \cdots & \widehat{x}_{1}\overline{\widehat{x}_{0}} & \left|\widehat{x}_{1}\right|^{2} & \widehat{x}_{1}\overline{\widehat{x}_{2}} & \cdots & 0\\
\vdots & \vdots & \vdots & \vdots & \vdots & \vdots & \vdots\\
0 & \cdots & \widehat{x}_{\gamma-2}\overline{\widehat{x}_{\gamma-3}} & \left|\widehat{x}_{\gamma-2}\right|^{2} & \widehat{x}_{\gamma-2}\overline{\widehat{x}_{\gamma-1}} & \cdots & 0\\
\widehat{x}_{\gamma-1}\overline{\widehat{x}_{0}} & \cdots & \widehat{x}_{\gamma-1}\overline{\widehat{x}_{\gamma-2}} & \left|\widehat{x}_{\gamma-1}\right|^{2} & 0 & \cdots & 0
\end{array}\right].
\end{align*}
\item Reshape $W^{\dagger}V$, into an estimate $G\in\mathbb{C}^{\gamma\times\gamma}$
of the rank-one matrix 
\[
\left.\widehat{{\bf x}}\right|{}_{\left[\gamma\right]_{0}}\left.\widehat{{\bf x}}^{*}\right|{}_{\left[\gamma\right]_{0}}=\left[\begin{array}{ccccc}
\left|\widehat{x}_{0}\right|^{2} & \widehat{x}_{0}\overline{\widehat{x}_{1}} & \widehat{x}_{0}\overline{\widehat{x}_{2}} & \cdots & \widehat{x}_{0}\overline{\widehat{x}_{\gamma-1}}\\
\widehat{x}_{1}\overline{\widehat{x}_{0}} & \left|\widehat{x}_{1}\right|^{2} & \widehat{x}_{1}\overline{\widehat{x}_{2}} & \cdots & \widehat{x}_{1}\overline{\widehat{x}_{\gamma-1}}\\
\vdots & \vdots & \ddots & \vdots & \vdots\\
\widehat{x}_{\gamma-1}\overline{\widehat{x}_{0}} & \widehat{x}_{\gamma-1}\overline{\widehat{x}_{1}} & \widehat{x}_{\gamma-1}\overline{\widehat{x}_{2}} & \cdots & \left|\widehat{x}_{\gamma-1}\right|^{2}
\end{array}\right].
\]
\item Hermitianize the matrix $G$ above: $G\mapsfrom\frac{1}{2}\left(G+G^{*}\right)$.
\item Compute $\lambda_{1},$ the largest  eigenvalue of $G$,
and ${\bf v}_{1}$, its associated normalized eigenvector.
\end{enumerate}
\begin{raggedright}
\textbf{Output}
\par\end{raggedright}
\begin{raggedright}
${\bf x}_e\coloneqq F_d^{-1}\widehat{{\bf x}}_{e},$ an estimate of ${\bf x},$  where $\widehat{{\bf x}}_e$
is given componentwise by
\[
\left(\widehat{{\bf x}}_{e}\right)_{j}=\begin{cases}
\sqrt{\left|\lambda_{1}\right|}\left({\bf v}_{1}\right)_{j}, & j\in\left[\gamma\right]_{0},\\
0, & \text{otherwise}.
\end{cases}
\]
\par\end{raggedright}
\raggedright{}\textbf{\caption{\textbf{Wigner Deconvolution and Angular Synchronization for Bandlimited
Signals}}
}
\end{algorithm}

\Algtwo*

\begin{proof}[The Proof of Theorem \ref{thm:Algorithm_2_guarantees}]Analogously to the proof of Theorem \ref{thm:Algorithm_1_guarantees}, we apply Lemma \ref{thm:generalSumCollapse-2} with $\xi=\gamma$ and $\kappa=\delta$ to the cases where  $\gamma\leq \beta\leq L-1,$ $0\leq \beta\leq \gamma-1,$  $\delta\leq\nu\leq K-1,$ and $0\leq \nu\leq \delta-1,$ and then substitute,    $\alpha=L-\beta,$ $\alpha=-\beta,$ $\omega=\nu-K,$ and $\omega=\nu,$ to see that 

\begin{enumerate}
\item if $ 1-\gamma\leq\alpha\leq0\text{ and }1-\delta\leq\omega\leq-1$,
then
\[
\left(F_{d}\left(\widehat{{\bf x}}\circ S_{\alpha}\overline{\widehat{{\bf x}}}\right)\right)_{\omega}+\frac{d^{3}}{KL}\frac{\left(F_{L}N_{K,L}^{T}F_{K}^{T}\right)_{-\alpha,\omega+K}}{\left(F_{d}\left(\widehat{{\bf m}}\circ S_{-\alpha}\overline{\widehat{{\bf m}}}\right)\right)_{\omega}}=\frac{d^{3}}{KL}\frac{\left(F_{L}Y_{K,L}^{T}F_{K}^{T}\right)_{-\alpha,\omega+K}}{\left(F_{d}\left(\widehat{{\bf m}}\circ S_{-\alpha}\overline{\widehat{{\bf m}}}\right)\right)_{\omega}},
\]
\item if $1-\gamma\leq\alpha\leq0\text{ and }0\leq\omega\leq\delta-1$,
then
\[
\left(F_{d}\left(\widehat{{\bf x}}\circ S_{\alpha}\overline{\widehat{{\bf x}}}\right)\right)_{\omega}+\frac{d^{3}}{KL}\frac{\left(F_{L}N_{K,L}^{T}F_{K}^{T}\right)_{-\alpha,\omega}}{\left(F_{d}\left(\widehat{{\bf m}}\circ S_{-\alpha}\overline{\widehat{{\bf m}}}\right)\right)_{\omega}}=\frac{d^{3}}{KL}\frac{\left(F_{L}Y_{K,L}^{T}F_{K}^{T}\right)_{-\alpha,\omega}}{\left(F_{d}\left(\widehat{{\bf m}}\circ S_{-\alpha}\overline{\widehat{{\bf m}}}\right)\right)_{\omega}},
\]
\item if $ 1\leq\alpha\leq\gamma-1 \text{ and }1-\delta\leq\omega\leq-1$,
then
\[
\left(F_{d}\left(\widehat{{\bf x}}\circ S_{\alpha}\overline{\widehat{{\bf x}}}\right)\right)_{\omega}+\frac{d^{3}}{KL}\frac{\left(F_{L}N_{K,L}^{T}F_{K}^{T}\right)_{L-\alpha,\omega+K}}{\left(F_{d}\left(\widehat{{\bf m}}\circ S_{-\alpha}\overline{\widehat{{\bf m}}}\right)\right)_{\omega}}=\frac{d^{3}}{KL}\frac{\left(F_{L}Y_{K,L}^{T}F_{K}^{T}\right)_{L-\alpha,\omega+K}}{\left(F_{d}\left(\widehat{{\bf m}}\circ S_{-\alpha}\overline{\widehat{{\bf m}}}\right)\right)_{\omega}},
\]
\item if $1\leq\alpha\leq\gamma-1 \text{ and }0\leq\omega\leq\delta-1$,
then
\[
\left(F_{d}\left(\widehat{{\bf x}}\circ S_{\alpha}\overline{\widehat{{\bf x}}}\right)\right)_{\omega}+\frac{d^{3}}{KL}\frac{\left(F_{L}N_{K,L}^{T}F_{K}^{T}\right)_{L-\alpha,\omega}}{\left(F_{d}\left(\widehat{{\bf m}}\circ S_{L-\alpha}\overline{\widehat{{\bf m}}}\right)\right)_{\omega}}=\frac{d^{3}}{KL}\frac{\left(F_{L}Y_{K,L}^{T}F_{K}^{T}\right)_{L-\alpha,\omega}}{\left(F_{d}\left(\widehat{{\bf m}}\circ S_{L-\alpha}\overline{\widehat{{\bf m}}}\right)\right)_{\omega}}.
\]
\end{enumerate}
We will write the above equations in matrix form, with rows indexed from $1-\delta$ to $\delta-1$ and columns indexed from $1-\gamma$ to $\gamma-1$, as 
\begin{equation}\label{eqn: TUV}
T+U=V,
\end{equation}
where 
$T$ is $(2\delta-1)\times(2\gamma-1)$ matrix
with entries defined by
\begin{equation*}
T_{\alpha,\omega}=\left(F_{d}\left(\widehat{{\bf x}}\circ S_{\alpha}\overline{\widehat{{\bf x}}}\right)\right)_{\omega}, 
\end{equation*}
and the entries of $U$ and $V$ are given by
\begin{equation}\label{eq:Vdef}
U_{\alpha,\omega}= \frac{d^{3}}{KL}\frac{\left(F_{L}N_{K,L}^{T}F_{K}^{T}\right)_{\beta(\alpha),\nu(\omega)}}{\left(F_{d}\left(\widehat{{\bf m}}\circ S_{-\alpha}\overline{\widehat{{\bf m}}}\right)\right)_{\omega}} \quad\text{and}\quad V_{\alpha,\omega}=\frac{d^{3}}{KL}\frac{\left(F_{L}N_{K,L}^{T}F_{K}^{T}\right)_{\beta(\alpha),\nu(\omega)}}{\left(F_{d}\left(\widehat{{\bf m}}\circ S_{-\alpha}\overline{\widehat{{\bf m}}}\right)\right)_{\omega}},
\end{equation}
with
\begin{equation*}
\nu(\omega)=\begin{cases}
\omega+K &\text{if }1-\delta\leq \omega\leq -1\\
\omega &\text{if } 0\leq \omega\leq \delta-1,
\end{cases}
\quad\text{and}\quad 
\beta(\alpha)=\begin{cases}
-\alpha &\text{if }1-\delta\leq \omega\leq -1\\
L-\alpha &\text{if } 0\leq \omega\leq \delta-1,
\end{cases}
\end{equation*}
for $1-\gamma\leq\alpha\leq\gamma-1, 1-\delta\leq\omega\leq\delta-1.$
 Let $\mu_2$
be as in (\ref{eq:mu-2}). Then, by the same reasoning as in \eqref{eq:N_2kappa-1_bound}, we see
\begin{align}
\left\Vert U\right\Vert _{F}^{2}
 & \leq\frac{d^{6}}{K^{2}L^{2}}\frac{\left\Vert F_{L}N_{K,L}F_{K}^{T}\right\Vert _{F}^{2}}{{\displaystyle \mu_2^{2}}}\nonumber\\
 & =\frac{d^{6}}{K^{2}L^{2}\mu_2^{2}} L K\left\Vert N_{K,L}\right\Vert _{F}^{2}\nonumber\\
 & =\frac{d^{6}}{KL\mu_2^{2}}\left\Vert N_{K,L}\right\Vert _{F}^{2}.\label{eq:U_bound}
\end{align}
Furthermore, for all $\alpha$ and $\omega$ such that  $\left|\alpha\right|\leq\gamma-1$ and $\left|\omega\right|\leq\delta-1,$
\begin{align*}
T_{\omega,\alpha}=\left(F_{d}\left(\widehat{{\bf x}}\circ S_{\alpha}\overline{\widehat{{\bf x}}}\right)\right)_{\omega} & =\sum_{n=0}^{\gamma-1}\mathbbm{e}^{-\frac{2\pi \mathbbm{i}\phi n}{d}}\widehat{x}_{n}\overline{\widehat{x}_{n+\alpha}}.
\end{align*}
 Therefore, we see $T=WA,$ where 
$W\in\mathbb{C}^{\left(2\delta-1\right)\times\gamma}$  is the Vandermone matrix
 are given below by
\begin{align}
W & =\left[\begin{array}{cccc}
1 & \mathbbm{e}^{-\frac{2\pi \mathbbm{i}\cdot\left(-\delta+1\right)\cdot1}{d}} & \cdots & \mathbbm{e}^{-\frac{2\pi \mathbbm{i}\cdot\left(-\delta+1\right)\cdot\left(\gamma-1\right)}{d}}\\
\vdots & \vdots & \ddots & \vdots\\
1 & \mathbbm{e}^{-\frac{2\pi \mathbbm{i}\cdot\left(-1\right)\cdot1}{d}} & \cdots & \mathbbm{e}^{-\frac{2\pi \mathbbm{i}\cdot\left(-1\right)\cdot\left(\gamma-1\right)}{d}}\\
1 & 1 & \cdots & 1\\
1 & \mathbbm{e}^{-\frac{2\pi \mathbbm{i}\cdot1\cdot1}{d}} & \cdots & \mathbbm{e}^{-\frac{2\pi \mathbbm{i}\cdot1\cdot\left(\gamma-1\right)}{d}}\\
\vdots & \vdots & \ddots & \vdots\\
1 & \mathbbm{e}^{-\frac{2\pi \mathbbm{i}\cdot\left(\delta-1\right)\cdot1}{d}} & \cdots & \mathbbm{e}^{-\frac{2\pi \mathbbm{i}\cdot\left(\delta-1\right)\cdot\left(\gamma-1\right)}{d}}
\end{array}\right], \label{eqn:Vandermonde} \end{align}
and $A\in\mathbb{C}^{\gamma\times\left(2\gamma-1\right)}$ is the partial autocorrelation type matrix 
\begin{align*}
A & =\left[\begin{array}{ccccccc}
0 & \cdots & 0 & \left|\widehat{x}_{0}\right|^{2} & \widehat{x}_{0}\overline{\widehat{x}_{1}} & \cdots & \widehat{x}_{0}\overline{\widehat{x}_{\gamma-1}}\\
0 & \cdots & \widehat{x}_{1}\overline{\widehat{x}_{0}} & \left|\widehat{x}_{1}\right|^{2} & \widehat{x}_{1}\overline{\widehat{x}_{2}} & \cdots & 0\\
\vdots & \vdots & \vdots & \vdots & \vdots & \vdots & \vdots\\
0 & \cdots & \widehat{x}_{\gamma-2}\overline{\widehat{x}_{\gamma-3}} & \left|\widehat{x}_{\gamma-2}\right|^{2} & \widehat{x}_{\gamma-2}\overline{\widehat{x}_{\gamma-1}} & \cdots & 0\\
\widehat{x}_{\gamma-1}\overline{\widehat{x}_{0}} & \cdots & \widehat{x}_{\gamma-1}\overline{\widehat{x}_{\gamma-2}} & \left|\widehat{x}_{\gamma-1}\right|^{2} & 0 & \cdots & 0
\end{array}\right].
\end{align*}
Therefore, by \eqref{eqn: TUV} we have
\[
WA+U=V.
\]
$W$ has full rank since it is a Vandermonde matrix with distinct nodes, and therefore since $2\delta-1\geq \gamma,$ it has a left inverse given by 
$
W^{\dagger}=\left(W^{*}W\right)^{-1}W^{*}.
$
Thus,
\begin{equation*}
W^{\dagger}V=A+W^{\dagger}U,
\end{equation*}
and so, by \eqref{eq:U_bound} \begin{align}
\left\Vert W^{\dagger}V-A\right\Vert _{F} & =\left\Vert W^{\dagger}U\right\Vert _{F}\nonumber\\
 & \leq\left\Vert W^{\dagger}\right\Vert _{2}\left\Vert U\right\Vert _{F}\nonumber\\
 & \leq\frac{1}{\sigma_{min}\left(W\right)}\frac{d^{3}}{\sqrt{KL}\mu_2}\left\Vert N_{K,L}\right\Vert _{F}\label{eqn: WVAWU},
\end{align}
where $\sigma_{min}(W)$ is the minimal singular value of $W.$

Let $P:\mathbb{C}^{\gamma\times(2\gamma-1)}\rightarrow\mathbb{C}^{\gamma\times\gamma}$ be a reshaping operator, such that if $M=(M_{1-\gamma},\ldots,M_0,\ldots,M_{\gamma-1})$ is a $\gamma\times(2\gamma-1)$ matrix, 
\begin{equation}
(P(M))_{i,j}=M_{i,j-i}
\label{eqn:reshape_A}
\end{equation}
for $0\leq i,j\leq \gamma-1$ so that 
\[
P(A)=\left.\widehat{{\bf x}}\right|{}_{\left[\gamma\right]_{0}}\left.\widehat{{\bf x}}^{*}\right|{}_{\left[\gamma\right]_{0}}=\left[\begin{array}{ccccc}
\left|\widehat{x}_{0}\right|^{2} & \widehat{x}_{0}\overline{\widehat{x}_{1}} & \widehat{x}_{0}\overline{\widehat{x}_{2}} & \cdots & \widehat{x}_{0}\overline{\widehat{x}_{\gamma-1}}\\
\widehat{x}_{1}\overline{\widehat{x}_{0}} & \left|\widehat{x}_{1}\right|^{2} & \widehat{x}_{1}\overline{\widehat{x}_{2}} & \cdots & \widehat{x}_{1}\overline{\widehat{x}_{\gamma-1}}\\
\vdots & \vdots & \ddots & \vdots & \vdots\\
\widehat{x}_{\gamma-1}\overline{\widehat{x}_{0}} & \widehat{x}_{\gamma-1}\overline{\widehat{x}_{1}} & \widehat{x}_{\gamma-1}\overline{\widehat{x}_{2}} & \cdots & \left|\widehat{x}_{\gamma-1}\right|^{2}
\end{array}\right].
\]
and let $G=H(P(W^{\dagger}V)),$ where $H$ is the Hermitianizing operator defined in \eqref{eqn:herm}.
\eqref{eq:beta} and the fact that $\|\widehat{\mathbf{x}}\|_2=\sqrt{d}\|\mathbf{x}\|_2$ for all $\mathbf{x}\in\mathbb{C}^d,$
imply that $\left\Vert N_{K,L}\right\Vert _{F}\leq\frac{\beta}{d}\left\Vert \widehat{{\bf x}}\right\Vert _{2}^{2},$
 so the fact that  $P(A)=\left.\widehat{{\bf x}}\right|{}_{\left[\gamma\right]_{0}}\left.\widehat{{\bf x}}^{*}\right|{}_{\left[\gamma\right]_{0}}$ is Hermitian, together with \eqref{eqn: WVAWU} implies

\begin{align*}
\left\Vert G-\left.\widehat{{\bf x}}\right|{}_{\left[\gamma\right]_{0}}\left.\widehat{{\bf x}}^{*}\right|{}_{\left[\gamma\right]_{0}}\right\Vert _{F} &= \left\Vert H(P(W^{\dagger}V))-H(P(A))\right\Vert _{F}\\
& \leq\left\Vert W^{\dagger}V-A\right\Vert _{F}\\
 & \leq\frac{1}{\sigma_{min}\left(W\right)}\frac{d^{3}}{\sqrt{KL}\mu_2}\left\Vert N_{K,L}\right\Vert _{F}\\
&\leq\frac{\beta}{\sigma_{min}\left(W\right)}\frac{d^{2}}{\sqrt{KL}\mu_2}\left\Vert \widehat{{\bf x}}\right\Vert _{2}^{2}.
\end{align*}
By Lemma 8 of \cite{iwen2018phase}, if $\lambda_{1}$ is the
lead eigenvalue of $G$ and ${\bf v}_{1}$
is an associated normalized eigenvector, 
then
\begin{align*}
\min_{\theta\in\left[0,2\pi\right]}\left\Vert \widehat{{\bf x}}-\mathbbm{e}^{\mathbbm{i}\theta}\widehat{\mathbf{x}}_{e}\right\Vert _{2} & =\min_{\theta\in\left[0,2\pi\right]}\left\Vert \mathbbm{e}^{\mathbbm{i}\theta}\left.\widehat{{\bf x}}\right|_{\left[\gamma\right]_{0}}-\sqrt{\left|\lambda_{1}\right|}{\bf v}_{1}\right\Vert _{2}\\
 & \leq\frac{\left(1+2\sqrt{2}\right)\beta}{\sigma_{min}\left(W\right)}\frac{d^{2}}{\sqrt{KL}\mu_2}\left\Vert \widehat{{\bf x}}\right\Vert _{2}.
\end{align*}
\eqref{eq:guarantee_2intro} follows by taking the inverse Fourier transform of both sides.
\end{proof}

\section{Numerical Experiments}
\label{sec:NumEval}
We now present numerical experiments which demonstrate the robustness and efficiency of the proposed algorithms and  provide comparisons to existing phase retrieval methods. These results were
generated using the open source {\em BlockPR} MATLAB software package (freely available at
\cite{bitbucket_BlockPR}) on a desktop computer (iMac, 2017) with an Intel$^\circledR$ Core\texttrademark i7-7700
(7\textsuperscript{th} generation, quad core) processor, 16GB RAM, and running macOS High Sierra and 
MATLAB R2018b. In all of our plots, each data point  
was obtained by averaging the results of $100$ trials.

Unless otherwise stated, we used i.i.d. mean zero complex Gaussian random test signals with 
measurement errors modeled using a (real) i.i.d. Gaussian noise model. We will report both the  signal to noise ratio (SNR) and reconstruction error
in decibels (dB) with 
\[  \mbox{SNR (dB)} = 10 \log_{10} \left( 
        \frac{\sum_{k=1}^K\sum_{\ell=1}^L\vert \langle \mathbf{x},  S_{\ell}W_k\mathbf{m} 
            \rangle \vert ^4}
            {D\sigma^2} \right), \quad  
    \mbox{Error (dB)} = 10 \log_{10} \left( \frac{\min_\theta \|\mathbbm{e}^{\mathbbm{i} \theta} \mathbf{x}_e - 
	\mathbf{x} \|_2^2}{\| \mathbf{x} \|^2_2}\right),  \]
where $\mathbf{x}, \mathbf{x}_e, \sigma^2$ and $D\coloneqq KL$ denote the true signal, recovered 
signal, (Gaussian) noise variance, and number of measurements respectively.

We will present selected results comparing the proposed formulation against other
popular phase retrieval algorithms such as {\em PhaseLift} \cite{candes2013phaselift} (implemented as a
trace-regularized least-squares problem using the first order convex optimization package TFOCS 
\cite{tfocspaper,tfocs}, {\em Hybrid Input-Output/Error Reduction (HIO+ER)} alternating projection 
algorithm \cite{bauschke2002phase,fienup1982phase}, and {\em Wirtinger Flow} \cite{Candes2015Wirtinger}. 
We note that more accurate results using {\em PhaseLift} may be obtained using other solvers and 
software packages (such as CVX \cite{cvxpaper,cvx}), albeit at a prohibitively expensive computational 
cost. For the HIO+ER algorithm, the following two projections were utilized: (i) projection onto the 
measured magnitudes, and (ii) projection onto the span of these measurement vectors. The initial guess was set to be the 
zero vector, although use of a random starting guess did not change the qualitative nature of the results. 
As is common practice, (see, for example, \cite{fienup1982phase})   we implemented the HIO+ER algorithm  in blocks of twenty-five HIO iterations followed by five ER iterations  in order  to accelerate 
 the convergence of the algorithm.
To minimize computational cost while ensuring convergence, the total number of HIO+ER iterations was limited to $600$ (see Figure 
\ref{fig:iter_HIOER}).
\begin{figure}[h]
    \centering
 \includegraphics[clip=true, trim=1.75cm 6.75cm 2cm 7cm, scale=0.5]{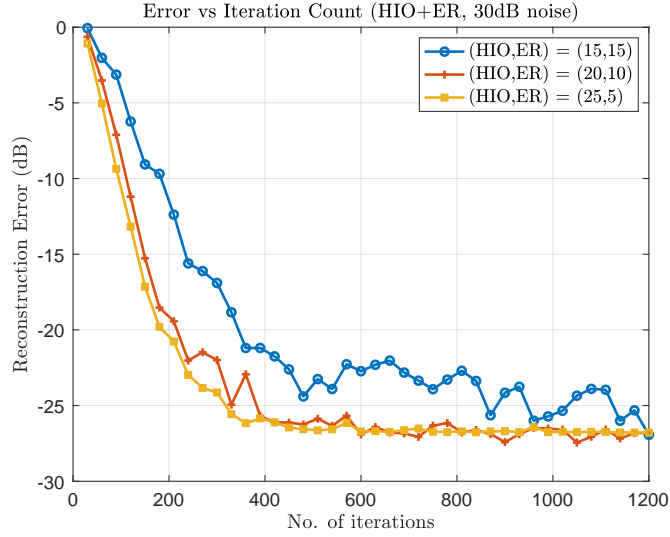}     
 \caption[Selection of HIO+ER iteration parameters.]{Selection of HIO+ER iteration parameters\footnotemark}
    \label{fig:iter_HIOER}
\end{figure}
\footnotetext{(HIO,ER) $= (x,y)$ indicates 
    that every $x$ iterations of the HIO algorithm was followed by $y$ iterations of the 
    ER algorithm.}
\subsection{Empirical Validation of Algorithm 1}
\label{sec:numerics_alg1}
In Algorithm 1, whose convergence is guaranteed by Theorem \ref{thm:Algorithm_1_guarantees}, we assume that our measurements are obtained using a bandlimited mask with $\mbox{supp}\left(\widehat{\mathbf{m}}\right)\subseteq[\rho]_0.$
To demonstrate the effictiveness of this algorithm, we performed numerical experiments on the following two types of masks:
\begin{equation}\label{eqn: randommask}
	\widehat{m}_k = \begin{cases}
				\left(1+0.5a_{\mathcal U}\right) \mathbbm{e}^{2\pi \mathbbm i a_{\mathcal U}}&\text{if }k \in [\rho]_0 \\
				0&\text{otherwise}
				\end{cases}
				\quad 
		a_{\mathcal U} \sim \mathcal U (0, 1),  \quad \text{(Random Mask)}
\end{equation}
where $\mathcal U(0,1)$ denotes an i.i.d uniform random distribution on the interval $[0,1]$, and  
\begin{equation}
	\widehat{m}_k = \begin{cases}
				\frac{\mathbbm{e}^{-k/a}}{\sqrt[4]{2\rho-1}}&\text{if }k \in [\rho]_0 \\
				0&\text{otherwise}
				\end{cases}
				\quad 
		a \coloneqq \max\left(4, (\rho-1)/2\right). 
		\quad \text{(Exponential Mask)}
	\label{eq:expmask}
\end{equation}
The exponential mask in (\ref{eq:expmask}) is closely related to the deterministic masks first introduced  
in \cite{iwen2016fast}. The mask-dependent constant $\mu_1$ (see (\ref{eq:mu}) 
in Theorem \ref{thm:Algorithm_1_guarantees}) for the random mask, with $d=60$ and $\rho=8$ (and
averaged over $50$ trials), was 
$2.858\times 10^{-1}$. The behavior for other choices of $d$ and
$\rho$ was similar. For the exponential mask, this constant was 
$2.267\times 10^{-2}$. The qualitative and quantitative performance of the algorithm was similar with 
both families of masks.
 
We performed experiments with both Algorithm 1, as presented in Section \ref{sec:recGuar}, and also with a modified version which uses a post-processing procedure to obtain improved accuracy. The modified algorithm replaces Steps (5) and (7) of Algorithm~1 (referred to  as Diag. Mag. Est. and Norm. Ang. Sync. in Figure \ref{fig:postprocess}) with 
the eigenvector based magnitude estimation procedure (Eig. Mag. Est.)  in Section 6.1 of \cite{iwen2018phase}, and the graph 
Laplacian based angular synchronization method (Graph Ang. Sync.) described in Algorithm 3 of B. Preskitt's dissertation 
\cite{preskitt2018phase}. 
  As seen in Figure \ref{fig:postprocess},  which plots the reconstruction 
error at various noise levels with  $d=255$,  $L=15,$ and a random mask constructed as in \eqref{eqn: randommask} with $\rho=8,$ these changes offered improved reconstruction accuracy.

\begin{figure}[hbtp]
\centering
\begin{subfigure}[t]{0.495\textwidth}
\centering
    \includegraphics[clip=true, trim=1.75cm 6.5cm 2.5cm 7cm, scale=0.465]{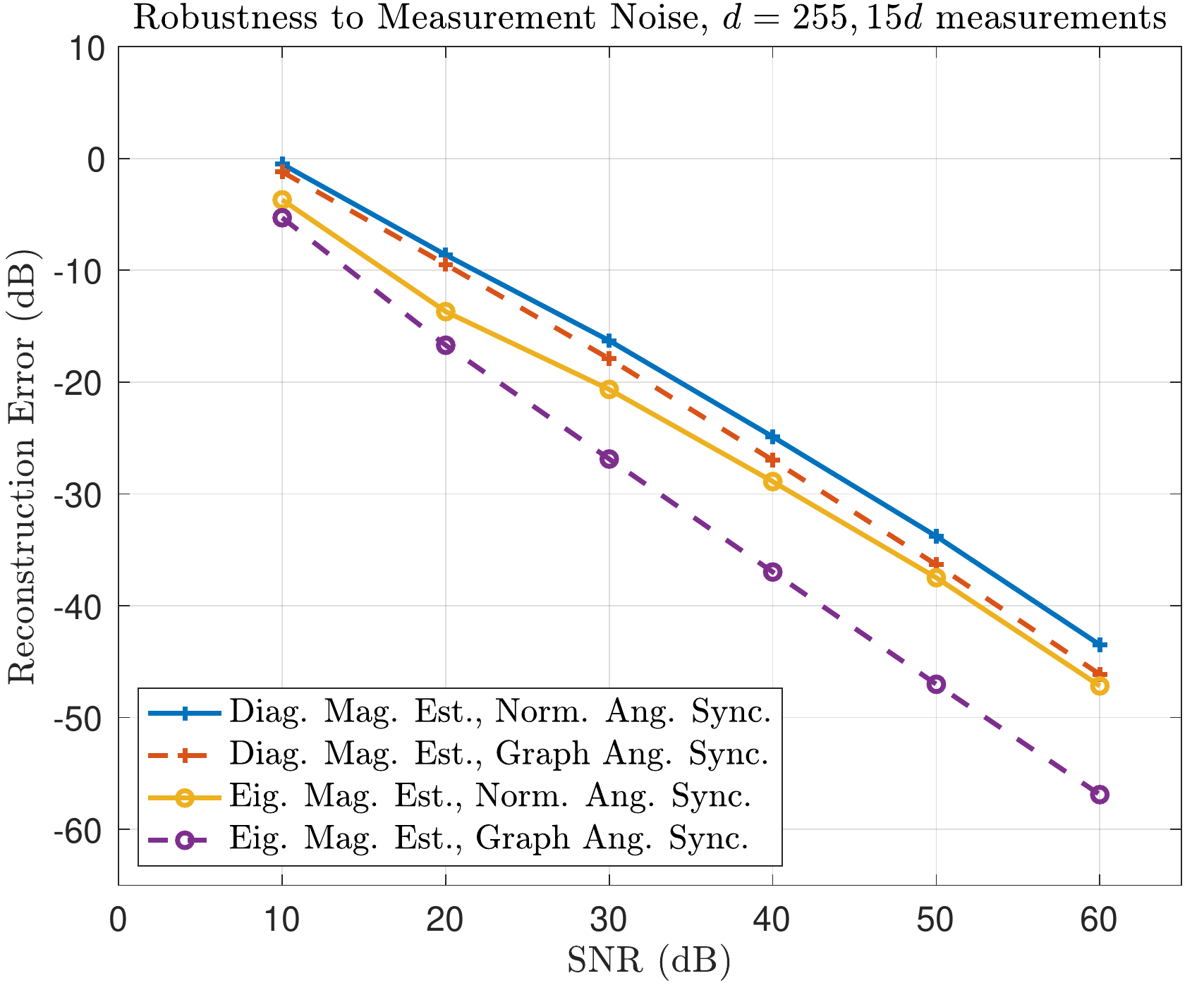}
   \caption{Reconstruction accuracy for Algorithm 1 with and without modifications to Steps (5) and/or (7).}
    \label{fig:postprocess}
\end{subfigure}
\hfill
\begin{subfigure}[t]{0.495\textwidth}
\centering
    \includegraphics[clip=true, trim=1.75cm 6.5cm 2.5cm 7cm, scale=0.465]{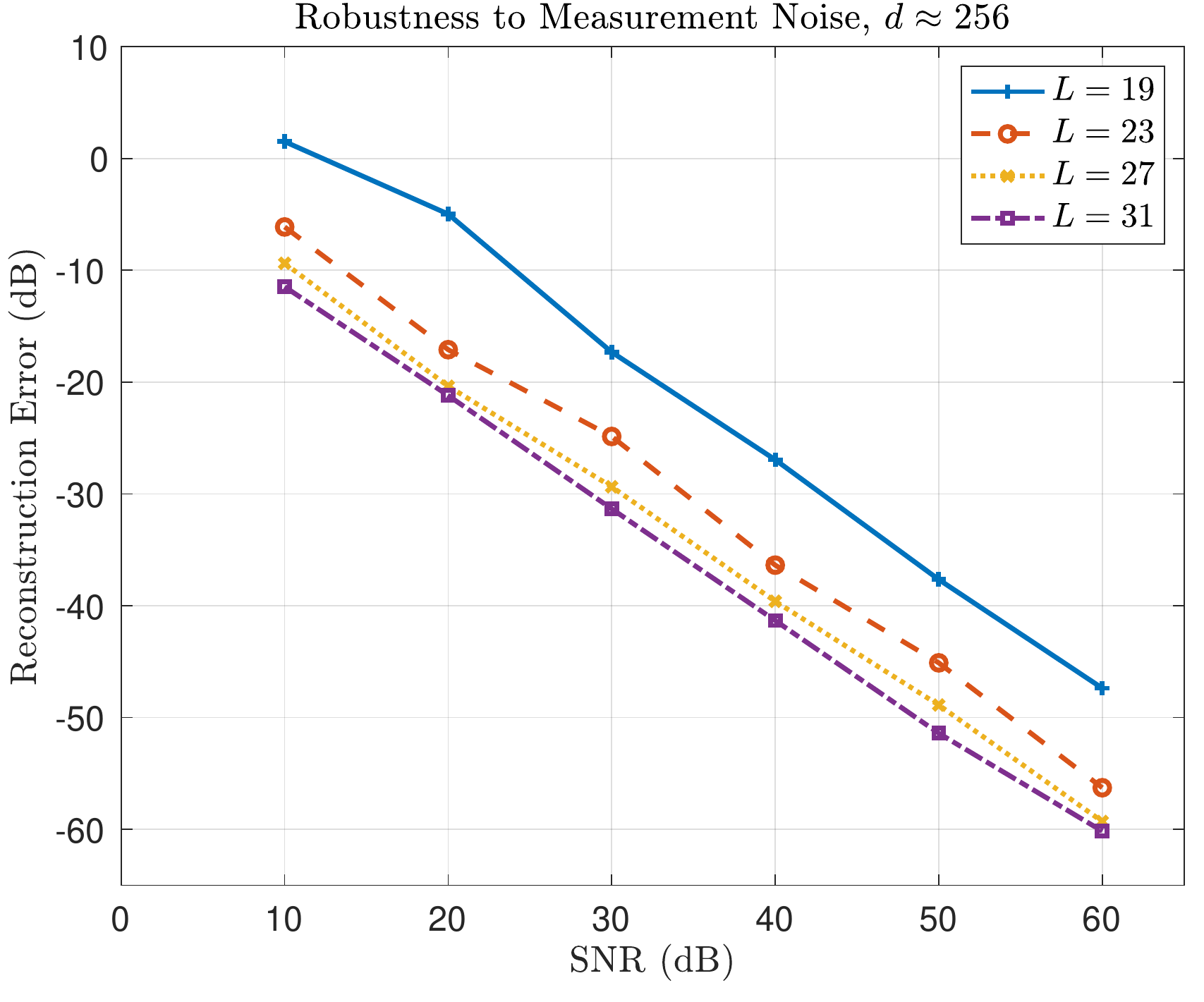}
    \caption{Reconstruction accuracy vs. number of shifts $L$ for Algorithm 1 (w/ mod. Steps (5),(7)).}
    \label{fig:error_vs_shifts}\end{subfigure}
\caption{Evaluating the performance of Algorithm 1  
  for various parameter choices }
\label{fig:alg1_parameters}
\end{figure}
\begin{figure}[hbtp]
\centering
\begin{subfigure}[t]{0.495\textwidth}
\centering
    \includegraphics[clip=true, trim=1.75cm 6.5cm 2.5cm 7cm, scale=0.465]{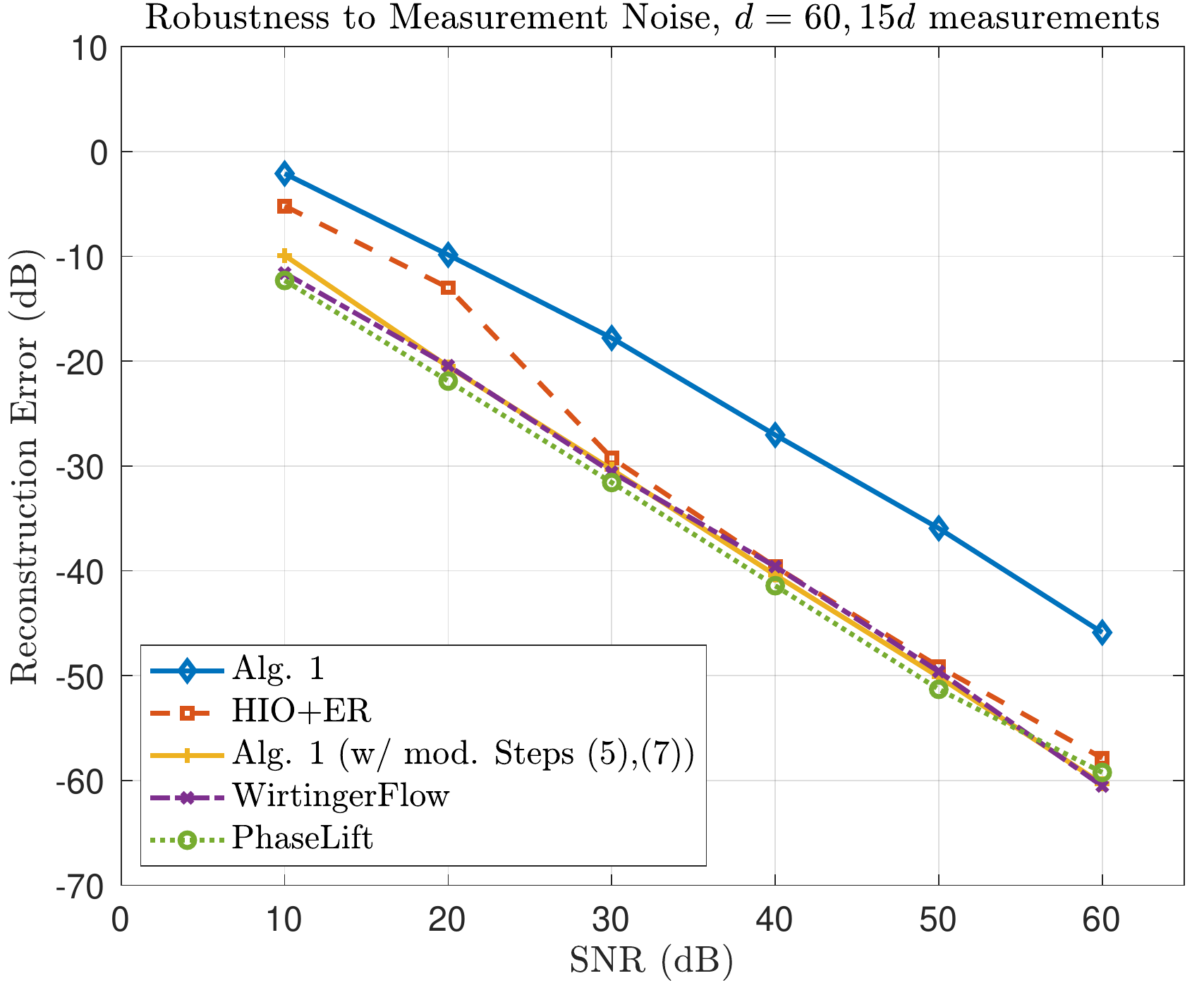}
    \caption{Reconstruction accuracy vs. added noise}
    \label{fig:noise}
\end{subfigure}
\hfill
\begin{subfigure}[t]{0.495\textwidth}
\centering
    \includegraphics[clip=true, trim=1.75cm 6.5cm 2.5cm 7cm, scale=0.465]{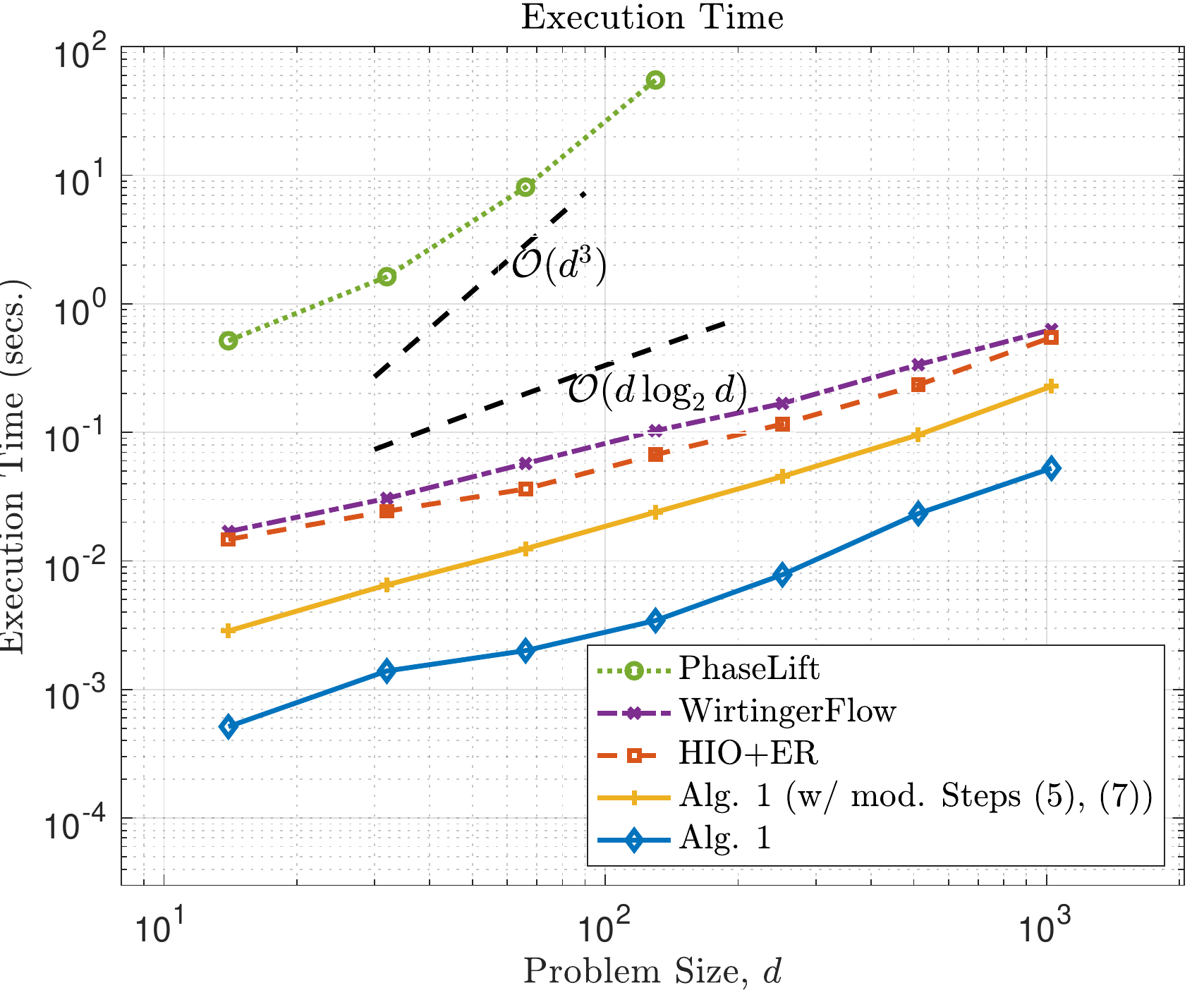}
    \caption{Execution time versus signal size $d$}
    \label{fig:exectime}
\end{subfigure}
\caption{Evaluating the robustness and efficiency of Algorithm 1 (and Theorem \ref{thm:Algorithm_1_guarantees})}
\label{fig:alg1_robustness_efficiency}
\end{figure}

Figure \ref{fig:error_vs_shifts} demonstrates the importance of  the number of shifts $L.$ 
As expected, the reconstructions using larger $L$ (which entails using more measurements, each corresponding 
to greater overlap between successive masked regions of the specimen) 
 offered improved accuracy. 
In order to ensure that $L$ divides $d,$ we varied the value of $d\approx256$ slightly for different values of $L.$ As in Figure \ref{fig:postprocess}, we used random masks, constructed as in \eqref{eqn: randommask}, with $\rho=8.$ 
We observe that for  larger values of $L$, performance improved by about $10$dB. We also note that, in practice, a suitable value of $L$ can be 
chosen depending on whether the proposed method is used as a reconstruction procedure or as an initializer for 
another algorithm. 

In Figure \ref{fig:noise}, we compare the performance of the proposed method to other popular phase retrieval methods. 
Reconstruction errors for recovering a signal of length $d=60$ using $L=15$ shifts and a random mask with $\rho=8$  
 are plotted for different levels of noise. We see that the proposed method performs well in comparison to 
the other algorithms, and even nearly matches the significantly more expensive algorithms such as  {\em PhaseLift} which are based on semidefinite programming (SDP).  
We note that the Wirtinger Flow method is sensitive to the choice of parameters and iteration counts. We 
used fewer total iterations (150 at 10dB SNR) at higher noise levels and more iterations (4500 at 60dB) at lower levels in order  to ensure that the algorithm converged to the level of noise. We are not aware of any 
methodical procedure for setting the various algorithmic parameters when utilizing the (local) measurement 
constructions considered in this paper. We also note that, of the algorithms considered, Algorithm 1 is the only one that has  a theoretical convergence guarantee which applies to this class of spectrogram-type measurements. 

Figure \ref{fig:exectime} plots the corresponding execution time 
for the various algorithms as a function of the problem size $d$. In this case, random masks were chosen 
with $\rho = \lceil 1.25 \log_2 d\rceil$ along with $L= \rho + \lceil\rho/2\rceil -1$ shifts. The figure confirms the 
essentially FFT--time computational cost of Algorithm 1. Furthermore, it also shows that while the post-processing 
procedure of modifying steps (5) and (7)  does  increase the computational cost of the algorithm, it does not increase it drastically.\footnote{The modified Step (7) uses MATLAB's 
{\tt eigs} command which can be computationally inefficient for this problem for large $d$; we defer a more detailed 
analysis and more efficient implementations to future work}  In particular, even with these modifications, the proposed method provides best--in--class computational efficiency, 
and is significantly faster than the {\em HIO+ER}, {\em Wirtinger Flow}, and {\em PhaseLift} algorithms.

\subsection{Empirical Validation of a Lemma \ref{thm:generalSumCollapse-1} Based Approach}
\label{sec:numerics_lemma11}
We next provide numerical results validating an approach based on Lemma \ref{thm:generalSumCollapse-1} that applies the Wigner deconvolution method to the setting considered in \cite{iwen2018phase}. As in Theorem \ref{thm:Algorithm_2_guarantees}, we assume that $\mbox{supp}(\mathbf{m})\subseteq[\delta]_0,$ and we also add the assumption that $L=d.$  In this setting, we may apply Lemma \ref{thm:generalSumCollapse-1} and then solve for diagonal bands of $\mathbf{x}\mathbf{x}^*$ in a manner  analogous to Algorithms 1 and 2. We then can recover $\mathbf{x}$ by applying the same angular synchronization procedure as in Algorithm 1. We note that, because we are using Lemma \ref{thm:generalSumCollapse-1} rather than Lemma \ref{thm:generalSumCollapse-2},  we do not need to assume that $\mathbf{x}$ is bandlimited as we do in Theorem \ref{thm:Algorithm_2_guarantees}.  
As in Section \ref{sec:numerics_alg1}, we conducted experiments with both deterministicly constructed and randomly constructed masks, and found that we obtained similar results for both families of masks. The figures below use the 
exponential mask construction first introduced in \cite{iwen2018phase},
\begin{equation}
	m_k = \begin{cases}
			\frac{\mathbbm{e}^{-k/a}}{\sqrt[4]{2\delta-1}},&\text{if }k \in [\delta]_0, \\
			0,&\text{otherwise},
		\end{cases}
				\quad 
		a := \max(4, (\delta-1)/2), 
	\label{eq:expmask_lemma15}
\end{equation}
and therefore allow us to directly compare the performance of the proposed method with the algorithm introduced in 
\cite{iwen2018phase}. The mask-dependent constant $\mu_2$ (see (\ref{eq:mu2} in Theorem \ref{thm:Algorithm_2_guarantees})
for this mask, with $d = 247$ and $\delta = 10$,
was $1.392 \times 10^{-2}$, with similar behavior for different choices of $d$ and $\delta$. Figure
\ref{fig:postprocess_K}  plots the reconstruction error with $d=247$, $K=19,$ $\delta=10,$ and $\mathbf{m}$ as in \eqref{eq:expmask_lemma15}. Results with and without the post-processing modifications 
 described in Section \ref{sec:numerics_alg1} are provided, along with results 
from \cite{iwen2018phase} with and without the modified (see \S 6.1 in \cite{iwen2018phase}) magnitude estimation and {\em HIO+ER} post-processing 
 ($60$ iterations).

\begin{figure}[hbtp]
\centering
\begin{subfigure}[t]{0.495\textwidth}
\centering
    \includegraphics[clip=true, trim=1.75cm 6.5cm 2.5cm 7cm, scale=0.465]{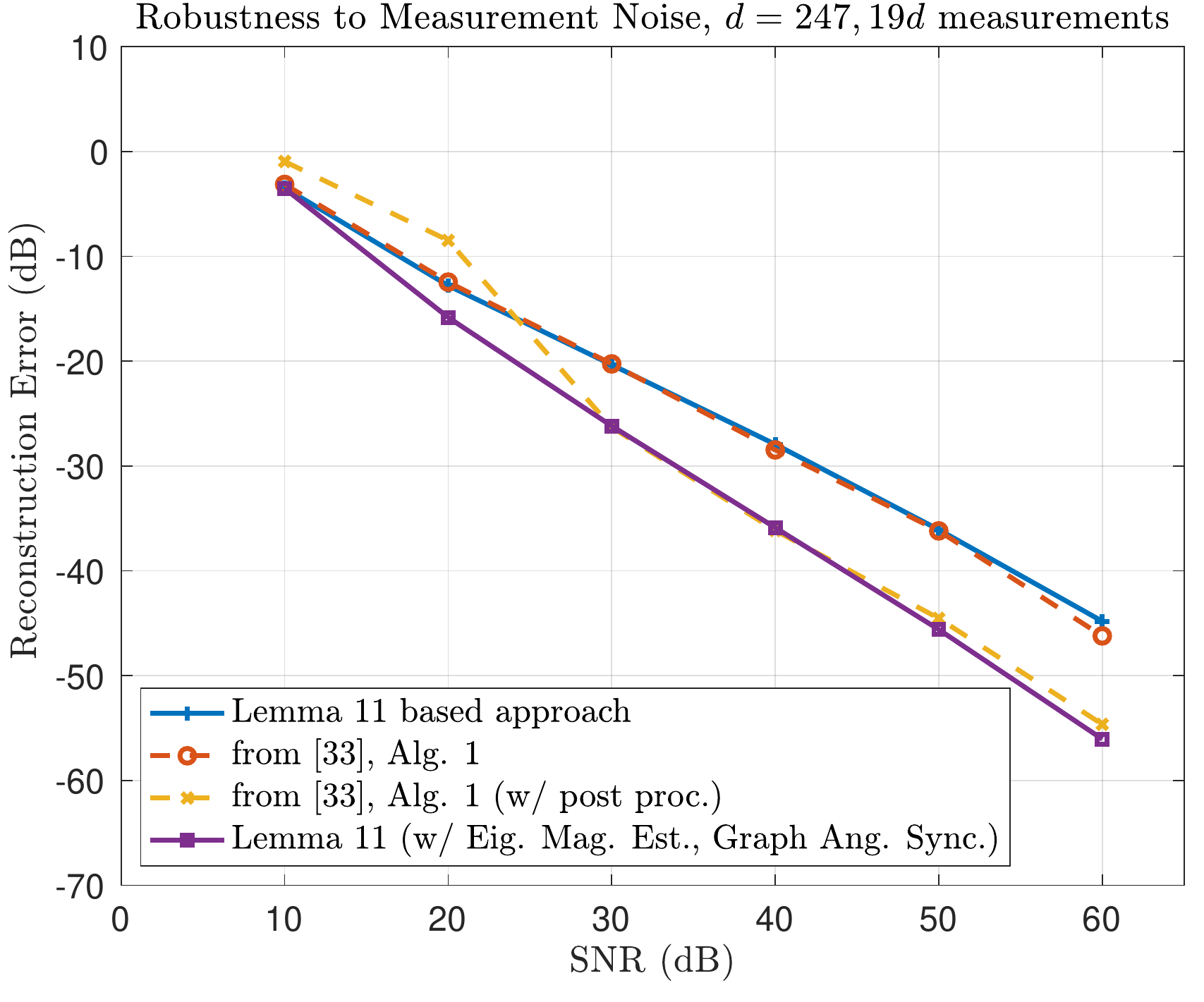}
    \caption{Reconstruction accuracy with and without improved magnitude estimation/angular synchronization, 
and comparison with results from \cite{iwen2018phase}.}
    \label{fig:postprocess_K}
\end{subfigure}
\hfill
\begin{subfigure}[t]{0.495\textwidth}
\centering
    \includegraphics[clip=true, trim=1.75cm 6.5cm 2.5cm 7cm, scale=0.465]{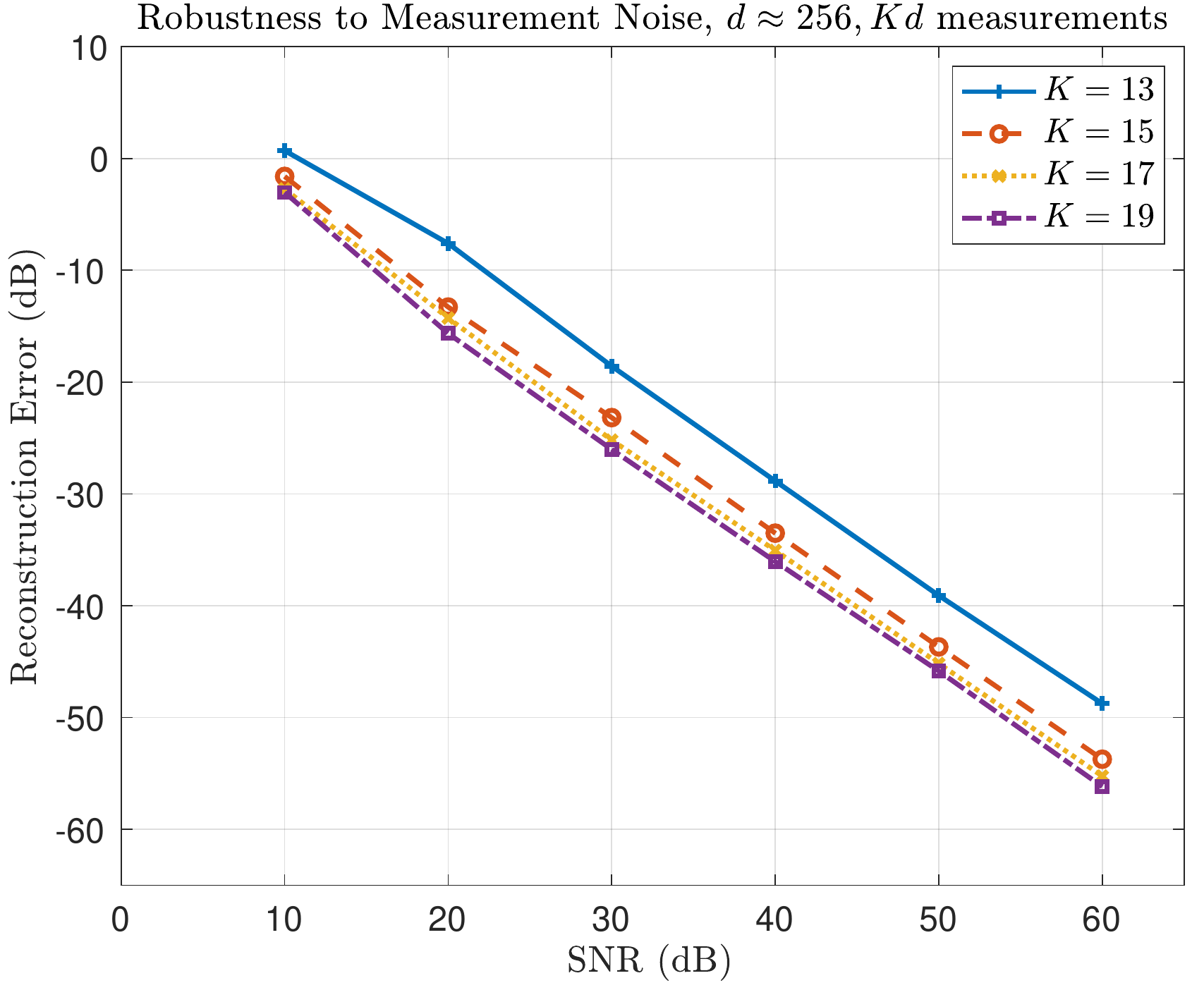}
    \caption{Reconstruction accuracy vs. $K,$ the number of Fourier modes.}
    \label{fig:error_vs_modes}
\end{subfigure}
\caption{Evaluating the performance of the Lemma \ref{thm:generalSumCollapse-1} 
  based approach}
\label{fig:lemma11_parameters}
\end{figure}

As can be seen in Figure \ref{fig:lemma11_parameters}, the post-processing procedure yields a small improvement of about
$5$-$10$dB in the reconstruction error, especially at low noise levels. We  observe that the
Wigner deconvolution based approach yields numerical performance which is comparable to \cite{iwen2018phase} in the settings where the theoretical guarantees of  \cite{iwen2018phase} are applicable, while also adding the additional flexibility of allowing shifts of length $a>1$ under certain assumptions on either $\mathbf{m}$ or $\mathbf{x}$ as discussed in Theorems \ref{thm:Algorithm_2_guarantees} and \ref{thm:Algorithm_1_guarantees}. 

Next, we investigate the reconstruction accuracy as a function of $K$, the number of Fourier modes. 
Figure \ref{fig:error_vs_modes} plots reconstruction error in recovering a test signal for 
$K=13,15,17,$ and $19$ respectively, with  the 
exponential masks defined as in \eqref{eq:expmask_lemma15} with $\delta=10.$ As in Figure \ref{fig:error_vs_shifts}, we vary the signal length $d$ slighlty, in order to ensure that $K$ divides $d.$ As expected, the plot shows that reconstruction 
accuracy improves 
 when $K$ increases, i.e., when more 
measurements are acquired.

For completeness, we include noise robustness and execution time plots comparing the performance of
the proposed method to the {\em HIO+ER}, {\em PhaseLift}, and {\em Wirtinger Flow} algorithms in 
Figures \ref{fig:noise_lemma11} and  \ref{fig:exectime_lemma11} respectively. From 
Figure \ref{fig:noise_lemma11}, we see that the proposed
method (both with and without the modified magnitude estimation/angular synchronization procedures) 
performs well in comparison to {\em HIO+ER} 
and the other algorithms across a wide range of SNRs. Furthermore,  Figure \ref{fig:exectime_lemma11} 
demonstrates the essentially FFT--time computational cost of
the method as well as the best-in-class computational efficiency when compared to other competing 
algorithms.

\begin{figure}[hbtp]
\centering
\begin{subfigure}[t]{0.495\textwidth}
\centering
    \includegraphics[clip=true, trim=1.75cm 6.5cm 2.5cm 7cm, scale=0.465]{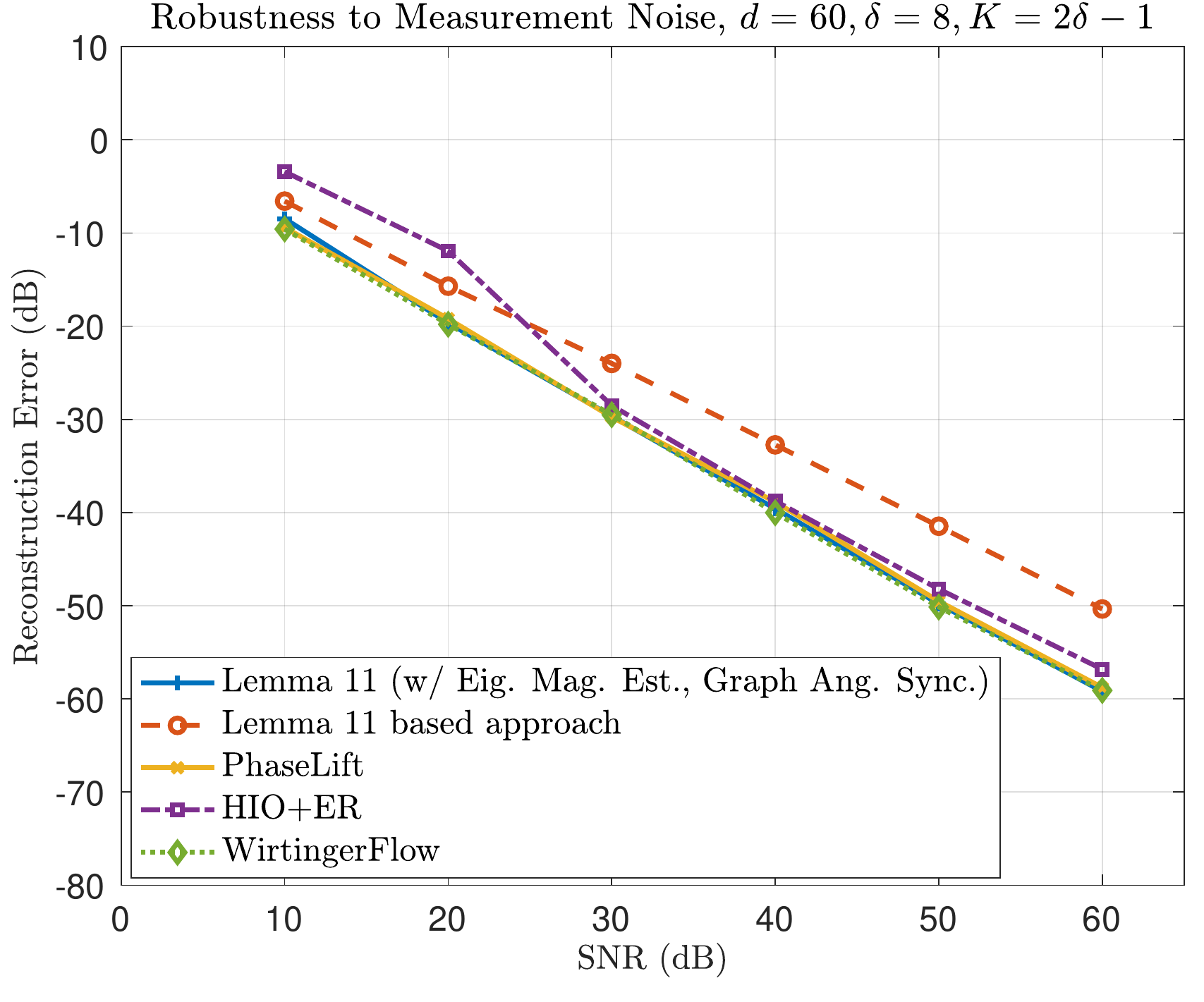}
    \caption{Reconstruction accuracy vs. added noise}
    \label{fig:noise_lemma11}
\end{subfigure}
\hfill
\begin{subfigure}[t]{0.495\textwidth}
\centering
    \includegraphics[clip=true, trim=1.75cm 6.5cm 2.5cm 7cm, scale=0.465]{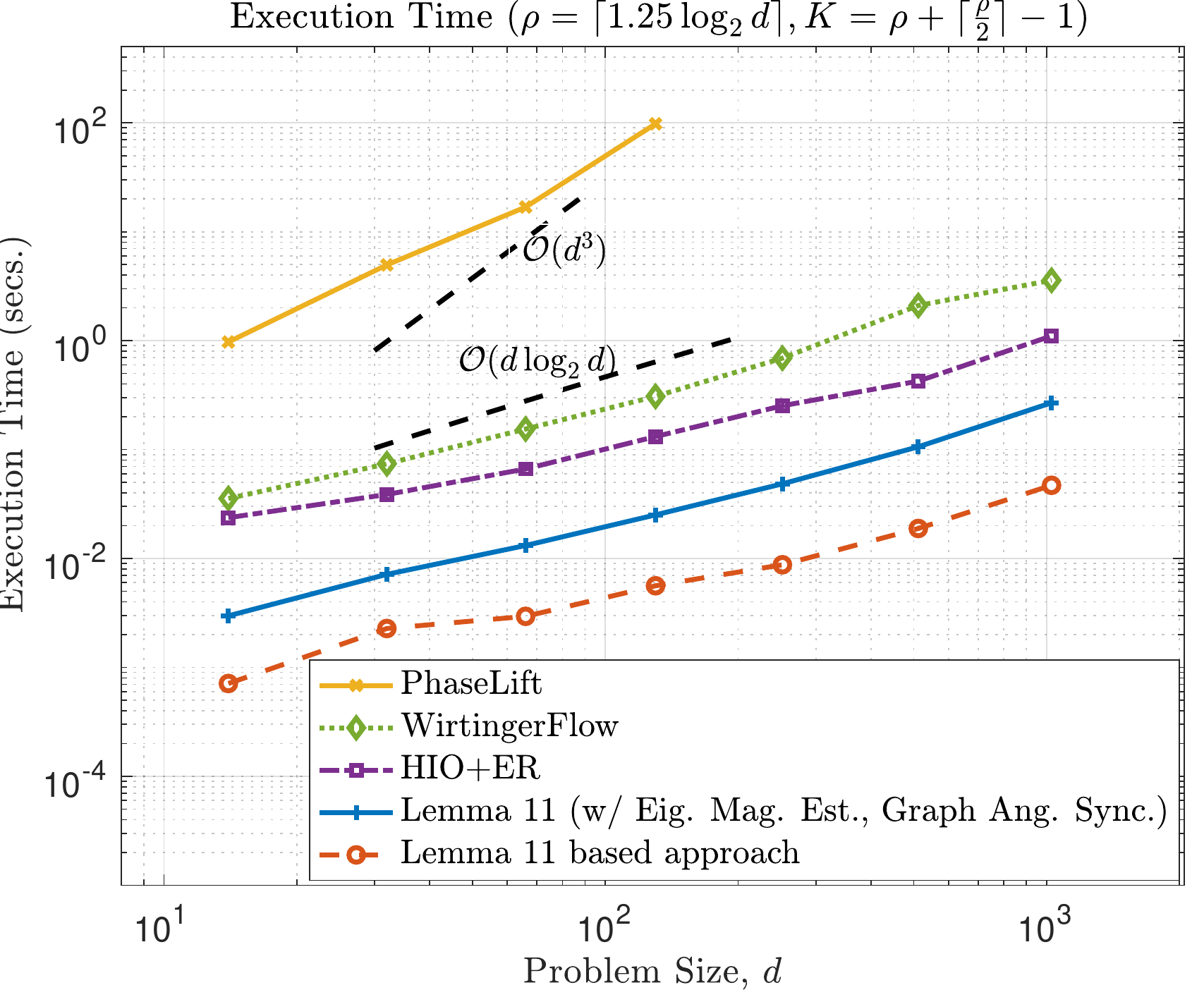}
    \caption{Execution time versus signal size $d$}
    \label{fig:exectime_lemma11}
\end{subfigure}
\caption{Evaluating the robustness and efficiency of the Lemma \ref{thm:generalSumCollapse-1} based approach}
\label{fig:lemma11_robustness_efficiency}
\end{figure}

\subsection{Empirical Validation of Algorithm 2}
\label{sec:numerics_alg2}

We now provide numerical results validating   
Algorithm 2, whose convergence is guaranteed by Theorem \ref{thm:Algorithm_2_guarantees}. We begin by noting that the Vandermonde matrix $W$ defined in
(\ref{eqn:Vandermonde}) often has a large condition number which poses a challenge in the
accurate evaluation of Step (4) in Algorithm 2. One possible solution is to utilize the 
Tikhonov regularized solution  $A = (W^*W +\sigma^2 I)^{-1}W^*V$ (see \cite{hansen2005rank} for example), 
where the regularization parameter $\sigma^2$ is chosen using a procedure such as the L-curve 
method \cite{hansen2000lcurve}. However, empirical simulations suggest that this procedure is not
sufficiently robust to achieve reconstruction accuracy up to the level of added noise. Therefore, we 
replace Steps (4)--(6) in Algorithm 2 by a modified non-stationary iterated 
Tikhonov method inspired by the work of Buccini et al. in \cite{buccini2017iterated}, which we detail in  Algorithm 3. This procedure 
works by iteratively computing a Tikhonov regularized solution to the equation in Step (4) of 
Algorithm 2; however, at each step, the solution is applied to the residual of 
$WA=V$,  with a geometrically decreasing regularization parameter.  Buccini et al. 
showed that a similar iterative procedure has benefits over traditional 
Tikhonov regularization for more standard linear systems. While our problem setting is different, our empirical 
results suggest a similar benefit. 
 We defer a more detailed 
theoretical analysis to future work.
\textbf{}
\begin{algorithm}[htbp] \label{alg:thirdalg}
\begin{raggedright}
\textbf{Inputs}
\par\end{raggedright}
\begin{enumerate}
\item Integers $\delta$ and $\gamma$, such that $\mbox{supp}\left({\bf {\bf m}}\right)\subseteq
    \left[\delta\right]_{0}$ and $\mbox{supp}\left({\bf \widehat{{\bf x}}}\right)
    \subseteq\left[\gamma\right]_{0}$.

\item Vandermonde matrix $W \in \mathbbm C^{(2\delta-1)\times \gamma}$ where 
    $W_{j,k} =\mathbbm{e}^{-\frac{2\pi \mathbbm{i}\left(j-\delta+1\right)k}{d}}\ 
    \text{for }j\in\left[2\delta-1\right]_{0}, k\in\left[\gamma\right]_{0}$.

\item Matrix $V \in \mathbbm C^{(2\delta-1)\times(2\gamma-1)}$ from Step (3) of Algorithm 2 and as
    specified in (\ref{eq:Vdef}). 

\item Non-stationary iterated Tikhonov parameters $\alpha_0$ and $q$ satisfying $\alpha_0>0$ and
    $0<q<1$.

\item Iteration count $N.$

\end{enumerate}
\begin{raggedright}
\textbf{Steps}
\par\end{raggedright}
\begin{enumerate}

\item Initialize $G\in\mathbbm C^{\gamma\times \gamma}$ and $A \in 
    \mathbbm C^{\gamma \times (2\gamma-1)}$ to zero.

\item For $k\gets 1 \ \text{to} \ N$ do 

\begin{enumerate}

\item Compute a rank-one approximation $G_1 \in \mathbbm C^{\gamma \times \gamma}$ of $G$:
    \[ G_1 = \tau_1 \mathbf{u}_1 \mathbf{v}_1^*, \]
    where $\tau_1$ is the largest singular value of $G$ and $\mathbf{u}_1$ and $\mathbf{v}_1$ are the 
    corresponding left and right singular vectors respectively. 
\item Let $A\in\mathbbm C^{\gamma \times (2\gamma-1)}$ be the matrix such that $P(A)=G_1$ and $A_{i,j}=0$ unless $j-i\leq j \leq j-i+\gamma-1,$  \text{ where $P$ is the reshaping operator defined in (\ref{eqn:reshape_A}).} 
%
%
\item Apply Tikhonov regularization with decaying regularization parameter to the residual: 
    \[ A \mapsfrom A + (W^*W + \alpha_0q^kI)^{-1}W^*\underbrace{(V-WA)}_\textrm{current residual} \]
\item Obtain an updated estimate of $G$: \[ G = P(A),\]
    
\item Hermitianize the matrix $G$: $G\mapsfrom\frac{1}{2}\left(G+G^{*}\right)$.

\end{enumerate}
\end{enumerate}
\begin{raggedright}
\textbf{Output}
\par\end{raggedright}
\begin{raggedright}
    An estimate of the matrix $G \in \mathbbm C^{\gamma\times \gamma}$ to be utilized in Step (5) of
    Algorithm 2.
\par\end{raggedright}
\raggedright{}\textbf{\caption{\textbf{Modified Non-Stationary Iterated Tikhonov Method with Geometrically
Decaying Regularization Parameters}}
}
\end{algorithm}

Figure \ref{fig:alg23_performance} presents empirical evaluation of the noise robustness and
computational efficiency of Algorithm 2 with the Modified Iterated Tikhonov Method of Algorithm 3.
Figure \ref{fig:noise_alg2} plots the reconstruction error with signals of length $d=190,$ with
frequency support of length $\gamma=10,$ using (complex random) masks with spatial support of length
$\delta=48.$ We used $K=2\delta-1$ Fourier modes and $L=2\gamma-1$ shifts, and utilized the
following iterated Tikhonov parameters: $q=0.8$, $N=20$, and $\alpha_0$ chosen using the L-curve
method.  We  note that using standard Tikhonov regularization (Alg. 2 in Figure
\ref{fig:noise_alg2}) yields rather poor results. An aggressive regularization parameter has to be
chosen to surmount the ill-conditioning effects in Step (4) of Algorithm 2. Consequently, even a few
iterations of the {\em HIO+ER} algorithm performs better than Algorithm 2. However, using the
modified iterated Tikhonov procedure (Alg. 2 (w/ Alg. 3) in Figure \ref{fig:noise_alg2}) yields
significantly improved results, with a clear improvement in noise robustness over even the {\em
HIO+ER} algorithm. Furthermore, Figure \ref{fig:etime_alg2} plots the execution time as a function
of the problem size for both Algorithms 2 and 3. The plot confirms that the modified iterative
Tikhonov procedure of Algorithm 3 does not impose a significant computational burden.\footnote{We
note that Step (1) of Algorithm 3 is computationally tractable since $\gamma$ is typically
small, and that the matrix $(W^*W + \alpha_0q^kI)^{-1}W^*$ in Step 2(c) can be pre-computed.}
Indeed, both Algorithms 2 and Algorithm 2 with the Modified Iterated Tikhonov Method of
Algorithm 3 are faster than the {\em HIO+ER} algorithm. We note that more efficient
implementations (involving fast computations of Vandermonde systems) of all the algorithms in
Figure \ref{fig:etime_alg2} may be possible; we defer this to future research.
\begin{figure}[hbtp]
\centering
\begin{subfigure}[t]{0.495\textwidth}
\centering
    \includegraphics[clip=true, trim=1.75cm 6.5cm 2.5cm 7cm, scale=0.465]{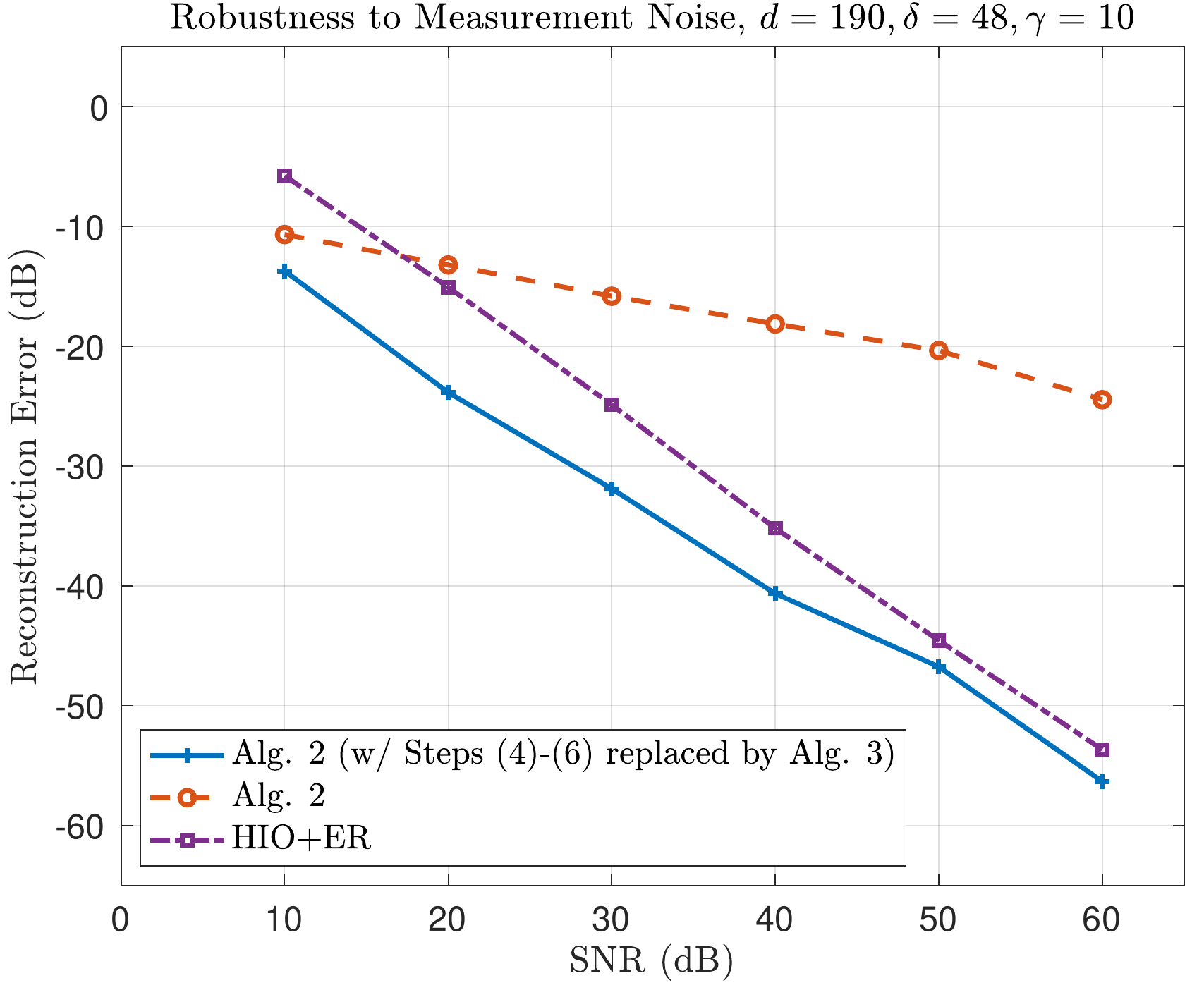}
    \caption{Reconstruction accuracy vs. added noise}
    \label{fig:noise_alg2}
\end{subfigure}
\hfill
\begin{subfigure}[t]{0.495\textwidth}
\centering
    \includegraphics[clip=true, trim=1.75cm 6.5cm 2.5cm 7cm, scale=0.465]{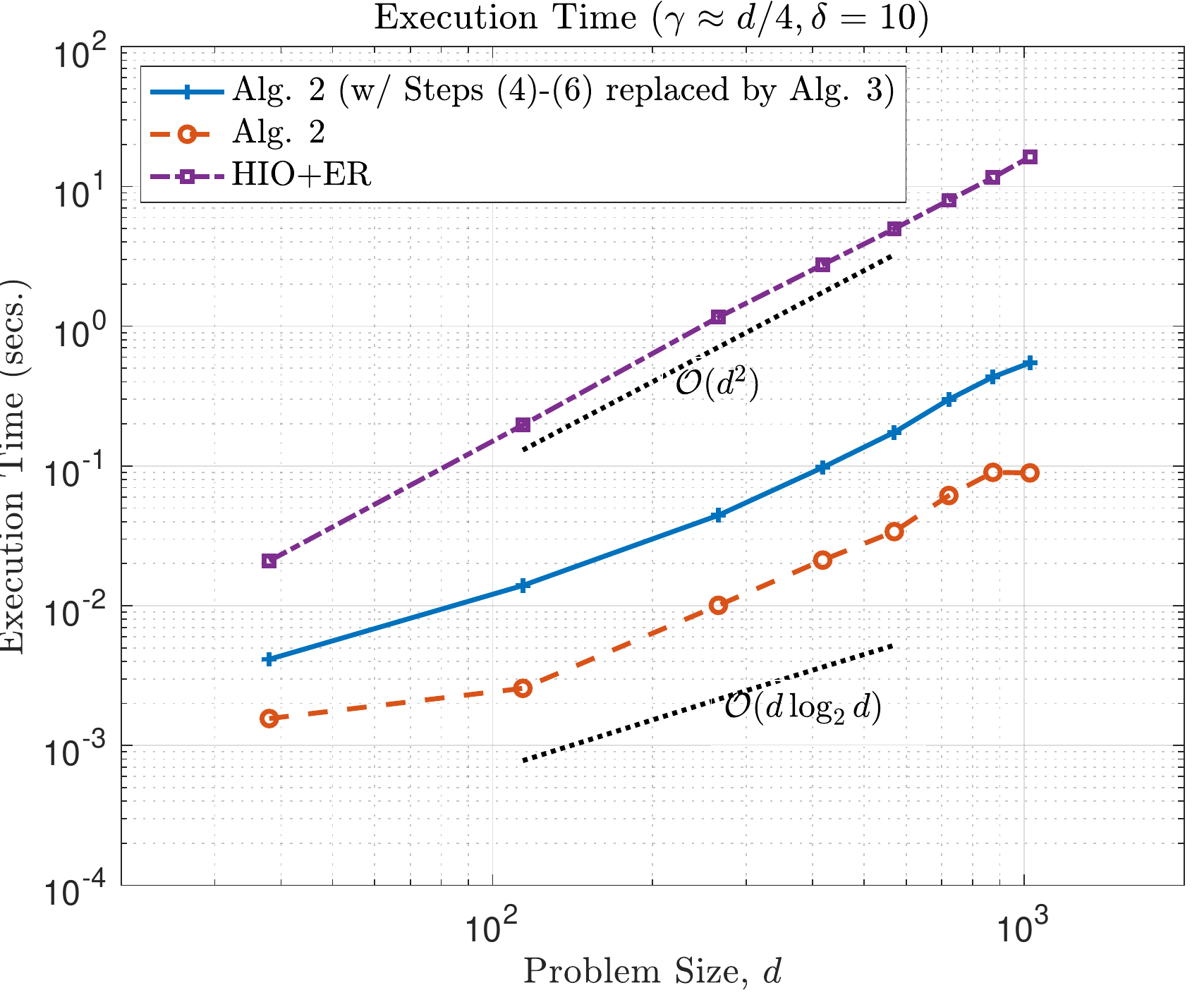}
    \caption{Execution time vs. signal size $d$}
    \label{fig:etime_alg2}
\end{subfigure}
\caption{Empirical validation of Theorem \ref{thm:Algorithm_2_guarantees} (Algorithm 2) and the
Modified Iterated Tikhonov Method of Algorithm 3}
\label{fig:alg23_performance}
\end{figure}

\section{Future Work}

In future work, one might develop  variants of the algorithms presented here for two-dimensional problems along the lines of \cite{iwen2017phase}. Additionally,
 one might also develop  variations of these algorithms for recovering compactly supported functions from sampled spectrogram measurements (see \cite{merhi2017recovery}) in the continuous setting. Furthermore, another, perhaps less direct, extension of these works would be to attempt to apply the Wigner distribution methods used here to the {\em sparse phase retrieval} problem. In, e.g., \cite{iwen2017robust} it was shown that sparse vectors $\x \in \C^d$ with $\| \x \|_0 \leq s$ can be recovered up to a global phase from only $m = \mathcal{O}(s \log (d / s))$ magnitude measurements of the form $\left\{ |\langle \x, \a_j \rangle|^2 \right\}^m_{j=1}$.  Thus, somewhat surprisingly, sparse phase retrieval problems generally do not require significantly more measurements to solve than compressive sensing problems.  One may be able to generate new sparse phase retrieval methods for STFT magnitude measurements of the type considered here by replacing the standard Fourier techniques used in the methods above with sparse Fourier transform methods \cite{Bittens2019,Merhi2019,Segal2013}. 
It has been shown that sparse phase retrieval problems can be solved in sublinear-time \cite{viswanathan2015fast}. The further development of sublinear-time methods for solving sparse phase retrieval problems involving STFT magnitude measurements could prove valuable in the future for use in extremely large imaging scenarios.

\section*{Acknowledgements}  Mark Iwen was supported in part by NSF DMS-1912706 and NSF CCF-1615489.  Sami Merhi was supported in part by NSF CCF-1615489.


\section*{Appendix}
In this section, we will prove the lemmas from Section \ref{Sec:BasicProps} as well as Propositions \ref{prop:mu_condition} and \ref{prop:mu_condition2}. 
\begin{proof}[The Proof of Lemma \ref{lem:DFTproperties}]
Let ${\bf x}\in\mathbb{C}^{d},$ and let $\ell,\omega\in\left[d\right]_{0}.$ 

Part \ref{FF}: \begin{align*}
\left(F_{d}\widehat{{\bf x}}\right)_{\omega} & =\sum_{k=0}^{d-1}\widehat{x}_{k}\mathbbm{e}^{-\frac{2\pi \mathbbm{i}k\omega}{d}}=\sum_{k=0}^{d-1}\sum_{n=0}^{d-1}x_{n}\mathbbm{e}^{-\frac{2\pi \mathbbm{i}nk}{d}}\mathbbm{e}^{-\frac{2\pi \mathbbm{i}k\omega}{d}}\\
 & =\sum_{k=0}^{d-1}\sum_{n=0}^{d-1}x_{-n}\mathbbm{e}^{\frac{2\pi \mathbbm{i}k\left(n-\omega\right)}{d}}=d x_{-\omega}=d\widetilde{x}_{\omega}.
\end{align*}

Part \ref{FW}:
\begin{align*}
\left(F_{d}\left(W_{\ell}{\bf x}\right)\right)_{\omega} & =\sum_{k=0}^{d-1}\left(x_{k}\mathbbm{e}^{\frac{2\pi \mathbbm{i}k\ell}{d}}\right)\mathbbm{e}^{-\frac{2\pi \mathbbm{i}k\omega}{d}}=\sum_{k=0}^{d-1}x_{k}\mathbbm{e}^{-\frac{2\pi \mathbbm{i}k\left(\omega-\ell\right)}{d}}\\
 & =\widehat{ x}_{\omega-\ell}=\left(S_{-\ell}\widehat{{\bf x}}\right)_{\omega}.
\end{align*}

Part \ref{FS}:
\begin{align*}
\left(F_{d}\left(S_{\ell}{\bf x}\right)\right)_{\omega} & =\sum_{k=0}^{d-1}x_{k+\ell}\mathbbm{e}^{-\frac{2\pi \mathbbm{i}k\omega}{d}}=\sum_{k=0}^{d-1}x_{k+\ell}\mathbbm{e}^{-\frac{2\pi \mathbbm{i}\left(k+\ell\right)\omega}{d}}\mathbbm{e}^{-\frac{2\pi \mathbbm{i}\left(-\ell\right)\omega}{d}} \\
 & =\mathbbm{e}^{\frac{2\pi \mathbbm{i}\ell\omega}{d}}\widehat{x}_{\omega}=\left(W_{\ell}\widehat{{\bf x}}\right)_{\omega}. \\
\end{align*}

Part \ref{WFS}:
\begin{align*}
\left(W_{-\ell}F_{d}\left(S_{\ell}\overline{\widetilde{{\bf x}}}\right)\right)_{\omega} & =\mathbbm{e}^{-\frac{2\pi \mathbbm{i}\ell\omega}{d}}\left(W_{\ell}\widehat{\overline{\widetilde{{\bf x}}}}\right)_{\omega}\quad \text{ (by part 3)}\\
&=\left(\widehat{\overline{\widetilde{{\bf x}}}}\right)_{\omega}=\sum_{k=0}^{d-1}\overline{\widetilde{x}}_{k}\mathbbm{e}^{-\frac{2\pi \mathbbm{i}k\omega}{d}}\\
 & =\overline{\sum_{k=0}^{d-1}\widetilde{x}_{k}\mathbbm{e}^{\frac{2\pi \mathbbm{i}k\omega}{d}}}=\overline{\sum_{k=0}^{d-1}x_{-k}\mathbbm{e}^{\frac{2\pi \mathbbm{i}k\omega}{d}}}=\left(\overline{\widehat{{\bf x}}}\right)_{\omega}.
\end{align*}

Part \ref{BarSTilde}:
\begin{align*}
\left(\overline{\widetilde{S_{\ell}{\bf x}}}\right)_{\omega} & =\overline{\widetilde{\left(S_{\ell}{\bf x}\right)_{\omega}}}\\
 & =\overline{x_{-\omega+\ell}}=\overline{\widetilde{x}}_{\ell-\omega}=\left(S_{-\ell}\overline{\widetilde{{\bf x}}}\right)_{\omega}.
\end{align*}

Part \ref{FBar}:
\begin{align*}
\left(F_{d}\overline{{\bf x}}\right)_{\omega} & =\sum_{k=0}^{d-1}\overline{x}_{k}\mathbbm{e}^{-\frac{2\pi \mathbbm{i}k\omega}{d}}=\overline{\sum_{k=0}^{d-1}x_{k}\mathbbm{e}^{\frac{2\pi \mathbbm{i}k\omega}{d}}}\\
 & =\overline{\sum_{k=0}^{d-1}x_{-k}\mathbbm{e}^{-\frac{2\pi \mathbbm{i}k\omega}{d}}}=\overline{\left(F_{d}\widetilde{{\bf x}}\right)_{\omega}}.
\end{align*}

Part \ref{TildeHat}:
\begin{align*}
\left(\widetilde{\widehat{{\bf x}}}\right)_{\omega} & =\widehat{{\bf x}}_{-\omega}=\sum_{k=0}^{d-1}x_{k}\mathbbm{e}^{\frac{2\pi \mathbbm{i}k\omega}{d}}\\
 & =\sum_{k=0}^{d-1}x_{-k}\mathbbm{e}^{-\frac{2\pi \mathbbm{i}k\omega}{d}}=\left(F_{d}\widetilde{{\bf x}}\right)_{\omega}.
\end{align*}

Part \ref{lem:absoluteFourierSquared}:
For all ${\bf x}\in\mathbb{C}^{d},$
\begin{align*}
\left|F_{d}{\bf x}\right|^{2} & =\left(F_{d}{\bf x}\right)\circ\overline{\left(F_{d}{\bf x}\right)}\\
 & =\left(F_{d}{\bf x}\right)\circ\left(F_{d}\overline{\widetilde{{\bf x}}}\right)\tag{by Lemma \ref{lem:DFTproperties}, part \ref{WFS}, with \ensuremath{\ell=0}}\\
 & =F_{d}\left({\bf x}\ast_{d}\overline{\widetilde{{\bf x}}}\right).\tag{by Lemma \ref{lem:convolutionTheorem}}
\end{align*}

\end{proof}

\begin{proof}[The Proof of Lemma \ref{lem:convolutionTheorem}]
For $\mathbf{x},\mathbf{y}\in\mathbb{C}^d$, $k\in [d]_0,$
\begin{align*}
(F_d({\bf x}\ast_d{\bf y}))_k &= \sum_{n=0}^{d-1} \sum_{\ell=0}^{d-1} x_\ell y_{n-\ell} \mathbbm{e}^{-\frac{2\pi \mathbbm{i}nk}{d}}\\  
&= \sum_{\ell=0}^{d-1} x_\ell\mathbbm{e}^{-\frac{2\pi \mathbbm{i}\ell k}{d}}\sum_{n=0}^{d-1}  y_{n-\ell} \mathbbm{e}^{-\frac{2\pi \mathbbm{i}(n-\ell) k}{d}}\\
&= \sum_{\ell=0}^{d-1} x_\ell\mathbbm{e}^{-\frac{2\pi \mathbbm{i}\ell k}{d}}\sum_{m=0}^{d-1}  y_{m} \mathbbm{e}^{-\frac{2\pi \mathbbm{i}m k}{d}}\\
&=\widehat{x}_k\widehat{y}_k.
\end{align*}
Therefore, $F_d(\mathbf{x}\ast_d\mathbf{y})=\mathbf{\widehat{x}}\circ\mathbf{\widehat{y}},$ so multiplying by $F_d^{-1}$ proves the first claim. To verify the second claim, note that by Lemma \ref{lem:DFTproperties} part \ref{FF},
\begin{align*}
F_d(F_d \mathbf{x} \ast_d F_d\mathbf{y}) &= F_dF_d\mathbf{x}\circ F_dF_d\mathbf{y}\\
&=d^2\widetilde{\mathbf{x}}\circ\widetilde{\mathbf{y}}\\
&=d^2\widetilde{\mathbf{x}\circ\mathbf{y}}\\
&=dF_d\left(dF_d(\mathbf{x}\circ\mathbf{y})\right).
\end{align*}

\end{proof}

\begin{proof}[The Proof of Lemma \ref{lem:fourierIndexSwap}]
Let ${\bf x}\in\mathbb{C}^{d},$ and let $\alpha,\omega\in\left[d\right]_{0}.$ Observe that
\begin{align*}
\left(F_{d}\left({\bf x}\circ S_{\omega}{\bf \overline{x}}\right)\right)_{\alpha} & =\frac{1}{d}\left(\widehat{{\bf x}}\ast_{d}F_{d}\left(S_{\omega}{\bf \overline{x}}\right)\right)_{\alpha}\tag{by Lemma \ref{lem:convolutionTheorem}}\\
 & =\frac{1}{d}\left(\widehat{{\bf x}}\ast_{d}\left(W_{\omega}\widehat{\overline{{\bf x}}}\right)\right)_{\alpha}\tag{\text{by Lemma \ref{lem:DFTproperties}, part \ref{FS}}}\\
 & =\frac{1}{d}\sum_{n=0}^{d-1}\widehat{x}_{n}\left(W_{\omega}\widehat{\overline{{\bf x}}}\right)_{\alpha-n}\tag{\text{by definition of \ensuremath{\ast_{d}}}}\\
 & =\frac{1}{d}\sum_{n=0}^{d-1}\widehat{x}_{n}\widehat{\overline{x}}_{\alpha-n}\mathbbm{e}^{\frac{2\pi \mathbbm{i}\omega\left(\alpha-n\right)}{d}}\tag{\text{by definition of \ensuremath{W_{\omega}}}}\\
 & =\frac{1}{d}\mathbbm{e}^{\frac{2\pi \mathbbm{i}\omega\alpha}{d}}\sum_{n=0}^{d-1}\widehat{x}_{n}\widetilde{\widehat{\overline{x}}}_{n-\alpha}\mathbbm{e}^{\frac{-2\pi \mathbbm{i}\omega n}{d}}\tag{by definition of \ensuremath{\widetilde{\cdot}}}\\
 & =\frac{1}{d}\mathbbm{e}^{\frac{2\pi \mathbbm{i}\omega\alpha}{d}}\sum_{n=0}^{d-1}\widehat{x}_{n}\overline{\widehat{x}}_{n-\alpha}\mathbbm{e}^{\frac{-2\pi \mathbbm{i}\omega n}{d}}\tag{by Lemma \ref{lem:DFTproperties}, parts \ref{FBar} and \ref{TildeHat}}\\
 & =\frac{1}{d}\mathbbm{e}^{\frac{2\pi \mathbbm{i}\omega\alpha}{d}}\left(F_{d}\left(\widehat{{\bf x}}\circ S_{-\alpha}\overline{\widehat{{\bf x}}}\right)\right)_{\omega}.
\end{align*}
\end{proof}

\begin{proof}[The Proof of Lemma \ref{lem:tildekill}]
For any $\mathbf{x},\mathbf{y}\in\mathbb{C}^d$ and any $\alpha\in\mathbb{Z},$ it is straightforward to check that 
\begin{equation}\label{eq:basicprops}
R\overline{\mathbf{x}}=\overline{R\mathbf{x}},\quad S_\alpha\overline{\mathbf{x}}=\overline{S_\alpha\mathbf{x}}, \quad \text{and}\quad R(\mathbf{x}\circ\mathbf{y})=(R\mathbf{x})\circ(R\mathbf{y}).
\end{equation}
Therefore,
\begin{align*}
F_d\left(\widetilde{\mathbf{x}}\circ S_{-\alpha}\overline{\widetilde{\mathbf{x}}}\right)&= F_d\left(R\mathbf{x}\circ S_{-\alpha}\overline{R\mathbf{x}
}\right)\tag{\text{by definition of }R}\\
&= F_d\left(R\mathbf{x}\circ \overline{RS_{\alpha}\mathbf{x}}\right)\tag{\text{by Lemma \ref{lem:DFTproperties}, part \ref{BarSTilde}}}\\
&= F_d\left(R\left(\mathbf{x}\circ S_{\alpha}\overline{\mathbf{x}}\right)\right)\tag{\text{by }\eqref{eq:basicprops}}\\
&= R\left(F_d\left(\mathbf{x}\circ S_{\alpha}\overline{\mathbf{x}}\right)\right).\tag{\text{by Lemma \ref{lem:DFTproperties}, part \ref{TildeHat}}}
\end{align*}

\end{proof}

\begin{proof}[The Proof of Lemma \ref{lem:indexSwap}]
Let ${\bf x},{\bf y}\in\mathbb{C}^{d},$ and let $\ell,k\in\left[d\right]_{0}.$ Then,
\begin{align*}
\left(\left({\bf x}\circ S_{-\ell}{\bf y}\right)\ast_{d}\left(\overline{\widetilde{{\bf x}}}\circ S_{\ell}\overline{\widetilde{{\bf y}}}\right)\right)_{k} & =\sum_{n=0}^{d-1}\left({\bf x}\circ S_{-\ell}{\bf y}\right)_{n}\left(\overline{\widetilde{{\bf x}}}\circ S_{\ell}\overline{\widetilde{{\bf y}}}\right)_{k-n}\tag{by definition of \ensuremath{\ast_{d}}}\\
 & =\sum_{n=0}^{d-1}x_{n}y_{n-\ell}\overline{\widetilde{x}}_{k-n}\overline{\widetilde{y}}_{\ell+k-n}\tag{by definition of \ensuremath{\circ}}\\
 & =\sum_{n=0}^{d-1}x_{n}\overline{x}_{n-k}\widetilde{y}_{\ell-n}\overline{\widetilde{y}}_{\ell-n+k}\tag{by definition of \ensuremath{\widetilde{\cdot}}}\\
 & =\left(\left({\bf x}\circ S_{-k}{\bf \overline{x}}\right)\ast_{d}\left(\widetilde{{\bf y}}\circ S_{k}\overline{\widetilde{{\bf y}}}\right)\right)_{\ell}.\tag{by definition of \ensuremath{\ast_{d}}}
\end{align*}
\end{proof}

\begin{proof}[The Proof of Lemma \ref{lem:aliasing}]
For $x\in \mathbb{C}^{d}$,
 the Fourier inversion formula states that 
\[
x_{n}=\left(F_{d}^{-1}\widehat{{\bf x}}\right)_{n}=\frac{1}{d}\sum_{k=0}^{d-1}\widehat{x}_{k}\mathbbm{e}^{\frac{2\pi \mathbbm{i}kn}{d}}.
\] 
Therefore, for all $\omega\in\left[\frac{d}{s}\right]_0,$
\begin{align*}
\left(F_{\frac{d}{s}}\left(Z_{s}{\bf x}\right)\right)_{\omega} &=\sum_{n=0}^{\frac{d}{s}-1}\left(Z_{s}{\bf x}\right)_{n}\mathbbm{e}^{-\frac{2\pi \mathbbm{i}n\omega}{d/s}}\\
&=\sum_{n=0}^{\frac{d}{s}-1}x_{ns}\mathbbm{e}^{-\frac{2\pi \mathbbm{i}n\omega}{d/s}}\\
& =\frac{1}{d}\sum_{n=0}^{\frac{d}{s}-1}\left(\sum_{k=0}^{d-1}\widehat{x}_{k}\mathbbm{e}^{\frac{2\pi \mathbbm{i}kns}{d}}\right)\mathbbm{e}^{-\frac{2\pi \mathbbm{i}\omega ns}{d}}\\
 & =\frac{1}{d}\sum_{k=0}^{d-1}\widehat{x}_{k}\sum_{n=0}^{\frac{d}{s}-1}\mathbbm{e}^{\frac{2\pi \mathbbm{i}n\left(k-\omega\right)}{d/s}}\\  
& =\frac{1}{d}\frac{d}{s}\sum_{r=0}^{s-1}\widehat{x}_{\omega+r\frac{d}{s}}=\frac{1}{s}\sum_{r=0}^{s-1}\widehat{x}_{\omega-r\frac{d}{s}}.
\end{align*}
\end{proof}

\begin{proof}[The Proof of Proposition \ref{prop:mu_condition}]
Let ${\bf m}\in\mathbb{C}^{d}$ be a bandlimited mask, whose Fourier transform may be written as
\[
\widehat{{\bf m}}=\left(a_{0}\mathbbm{e}^{\mathbbm{i}\theta_{0}},\ldots,a_{\rho-1}\mathbbm{e}^{\mathbbm{i}\theta_{\rho-1}},0,\ldots,0\right)^{T}
\]
for some real numbers $a_{0},\dots,a_{\rho-1},$ 
which satisfy \eqref{eqn: a0big} and \eqref{eqn: decreasingamplitudes}. Let $2\leq\kappa\leq\rho,$ and recall that $\mu_1$ is defined by
\[
\mu_1= \min_{_{\substack{\left|p\right|\leq\kappa-1\\
q\in\left[d\right]_{0}
}
}}\left|F_{d}\left(\widehat{\mathbf{m}}\circ S_{p}\overline{\widehat{\mathbf{m}}}\right){}_{q}\right|.
\]
For $0\leq p\leq\kappa-1$, we have
\[
\left(\widehat{\mathbf{m}}\circ S_{p}\overline{\widehat{\mathbf{m}}}\right)_{n}=\begin{cases}
a_{n}a_{n+p}\mathbbm{e}^{\mathbbm{i}\left(\theta_{n}-\theta_{n+p}\right)}, & \text{if }n\in\left[\rho-p\right]_{0},\\
0, & \text{otherwise,}
\end{cases}
\]
and for $-\kappa+1\leq p<0$, 
\[
\left(\widehat{\mathbf{m}}\circ S_{p}\overline{\widehat{\mathbf{m}}}\right)_{n}=\begin{cases}
a_{n-\left|p\right|}a_{n}\mathbbm{e}^{\mathbbm{i}\left(\theta_{n}-\theta_{n-\left|p\right|}\right)}, & \text{if }n\in\left\{ \left|p\right|,\left|p\right|+1,\dots,\rho-1\right\} ,\\
0, & \text{otherwise.}
\end{cases}
\]
Therefore, for any $q\in\left[d\right]_{0}$ and any $|p|\leq \kappa-1,$
\[
F_{d}\left(\widehat{\mathbf{m}}\circ S_{p}\overline{\widehat{\mathbf{m}}}\right)_{q}=\sum_{n=0}^{\rho-1-p}a_{n}a_{|p|+n}\mathbbm{e}^{\mathbbm{i}\phi_{n,p,q}},
\]
where $\phi_{n,p,q}$ is some real number depending on $n,p,$
and $q.$ Using the assumptions \eqref{eqn: a0big} and \eqref{eqn: decreasingamplitudes}
we see that 
\begin{align}
\left|\sum_{n=1}^{\rho-1-|p|}a_{n}a_{|p|+n}\mathbbm{e}^{\mathbbm{i}\phi_{n,p,q}}\right|\leq\left(\rho-1\right)\left|a_{1}\right|\left|a_{1+|p|}\right|<\left|a_{0}\right|\left|a_{\left|p\right|}\right|.\label{eq:mu_bound_1}
\end{align}
With this,
\begin{align*}
\left|F_{d}\left(\widehat{\mathbf{m}}\circ S_{p}\overline{\widehat{\mathbf{m}}}\right)_{q}\right| & =\left|\sum_{n=0}^{\rho-1-|p|}a_{n}a_{|p|+n}\mathbbm{e}^{\mathbbm{i}\phi_{n,p,q}}\right|\\
 & =\left|a_{0}a_{|p|}\mathbbm{e}^{\mathbbm{i}\phi_{0,p,q}}+\sum_{n=1}^{\rho-1-|p|}a_{n}a_{|p|+n}\mathbbm{e}^{\mathbbm{i}\phi_{n,p,q}}\right|\\
 & \geq\left|\left|a_{0}a_{|p|}\mathbbm{e}^{\mathbbm{i}\phi_{0,p,q}}\right|-\left|\sum_{n=1}^{\rho-1-|p|}a_{n}a_{|p|+n}\mathbbm{e}^{\mathbbm{i}\phi_{n,p,q}}\right|\right|\\
 & =\left|\left|a_{0}a_{|p|}\right|-\left|\sum_{n=1}^{\rho-1-|p|}a_{n}a_{|p|+n}\mathbbm{e}^{\mathbbm{i}\phi_{n,p,q}}\right|\right|\\
 & >0,
\end{align*}
where the last inequality follows by \ref{eq:mu_bound_1}. Therefore,
$F_{d}\left(\widehat{\mathbf{m}}\circ S_{p}\overline{\widehat{\mathbf{m}}}\right)_{q}$
is nonzero for all $p$ and $q$ and so $\mu_1>0.$

\end{proof}

\begin{proof}[The Proof of Proposition \ref{prop:mu_condition2}]
Let 
\[
{\bf m}=\left(a_{0}\mathbbm{e}^{\mathbbm{i}\theta_{0}},\ldots,a_{\delta-1}\mathbbm{e}^{\mathbbm{i}\theta_{\delta-1}},0,\ldots,0\right)^{T}
\]
be a compactly supported mask, where 
 $a_{0},\dots,a_{\delta-1},$ are real numbers which satisfy \eqref{eqn: a0big2} and \eqref{eqn: decreasingamplitudes2}. Let $1\leq\gamma\leq2\delta-1,$ and recalll that $\mu_2$ is defined by
\begin{equation*}
\mu_2= \min_{_{\substack{\left|p\right|\leq\gamma-1\\
\left|q\right|\leq\delta-1
}
}}\left|F_{d}\left(\widehat{{\bf m}}\circ S_{p}\overline{\widehat{{\bf m}}}\right)_{q}\right|.
\end{equation*}
By Lemma \ref{lem:fourierIndexSwap},
it suffices to show that
\begin{equation*}
F_{d}\left({\bf m}\circ S_{-q}\overline{{\bf m}}\right)_{p}\neq 0
\end{equation*}
for all $\left|p\right|\leq\gamma-1$ and all
$\left|q\right|\leq\delta-1.$
If $-\delta+1\leq q<0$, then
\[
\left(\mathbf{m}\circ S_{-q}\overline{\mathbf{m}}\right)_{n}=\begin{cases}
a_{n}a_{n+|q|}\mathbbm{e}^{\mathbbm{i}\left(\theta_{n}-\theta_{n+|q|}\right)}, & \text{if }0\leq n\leq \delta-|q|-1,\\
0, & \text{otherwise}
\end{cases},
\]
and if $0\leq q\leq\delta-1$, then 
\[
\left(\mathbf{m}\circ S_{-q}\overline{\mathbf{m}}\right)_{n}=\begin{cases}
a_{n-q}a_{n}\mathbbm{e}^{\mathbbm{i}\left(\theta_{n}-\theta_{n-q}\right)}, & \text{if }q\leq n\leq\delta-1 ,\\
0, & \text{otherwise}
\end{cases}.
\]
Therefore, for all $\left|p\right|\leq\gamma-1$ and all
$\left|q\right|\leq\delta-1$
\[
F_{d}\left(\mathbf{m}\circ S_{-q}\overline{\mathbf{m}}\right)_{p}=\sum_{n=0}^{\delta-1-|q|}a_{n}a_{|q|+n}\mathbbm{e}^{\mathbbm{i}\phi_{n,p,q}},
\]
where $\phi_{n,p,q}$ is some real number depending on $n,p,$
and $q.$ By the same reasoning as in the proof of Proposition \ref{prop:mu_condition}, this combined with \eqref{eqn: a0big2} and \eqref{eqn: decreasingamplitudes2} implies that $F_{d}\left(\mathbf{m}\circ S_{-q}\overline{\mathbf{m}}\right)_{p}\neq 0$ for all $\left|p\right|\leq\gamma-1$ and all
$\left|q\right|\leq\delta-1.$
%

\end{proof}

\bibliographystyle{abbrv}
\bibliography{refs2}

\begin{thebibliography}{10}

\bibitem{alexeev2014phase}
B.~Alexeev, A.~S. Bandeira, M.~Fickus, and D.~G. Mixon.
\newblock {Phase Retrieval with Polarization}.
\newblock {\em SIAM Journal on Imaging Sciences}, 7(1):35--66, 2014.

\bibitem{balan2006signal}
R.~Balan, P.~Casazza, and D.~Edidin.
\newblock On signal reconstruction without phase.
\newblock {\em Applied and Computational Harmonic Analysis}, 20(3):345--356,
  2006.

\bibitem{bandeira2014phase}
A.~S. Bandeira, Y.~Chen, and D.~G. Mixon.
\newblock Phase retrieval from power spectra of masked signals.
\newblock {\em Information and Inference: a Journal of the IMA}, 3(2):83--102,
  2014.

\bibitem{bauschke2002phase}
H.~H. Bauschke, P.~L. Combettes, and D.~R. Luke.
\newblock Phase retrieval, error reduction algorithm, and {F}ienup variants:
  {A} view from convex optimization.
\newblock {\em Journal of the Optical Society of America. A, Optics, Image
  science, and Vision}, 19(7):1334--1345, 2002.

\bibitem{tfocspaper}
S.~Becker, E.~J. Cand{\`e}s, and M.~Grant.
\newblock Templates for convex cone problems with applications to sparse signal
  recovery.
\newblock {\em Mathematical Programming Computation}, 3(3):165--218, Aug. 2011.

\bibitem{tfocs}
S.~Becker, E.~J. Candes, and M.~Grant.
\newblock {TFOCS}: Templates for first-order conic solvers, version 1.3.1.
\newblock \url{http://cvxr.com/tfocs}, Sep. 2014.

\bibitem{bendory2017non}
T.~Bendory, Y.~C. Eldar, and N.~Boumal.
\newblock Non-convex phase retrieval from stft measurements.
\newblock {\em IEEE Transactions on Information Theory}, 64(1):467--484, 2017.

\bibitem{Bittens2019}
S.~Bittens, R.~Zhang, and M.~A. Iwen.
\newblock A deterministic sparse fft for functions with structured fourier
  sparsity.
\newblock {\em Advances in Computational Mathematics}, 45(2):519--561, Apr
  2019.

\bibitem{buccini2017iterated}
A.~Buccini, M.~Donatelli, and L.~Reichel.
\newblock Iterated tikhonov regularization with a general penalty term.
\newblock {\em Numerical Linear Algebra with Applications}, 24(4):2089, 2017.

\bibitem{candes2015phase}
E.~J. Cand{\`e}s, Y.~C. Eldar, T.~Strohmer, and V.~Voroninski.
\newblock Phase retrieval via matrix completion.
\newblock {\em SIAM review}, 57(2):225--251, 2015.

\bibitem{Candes2014WF}
E.~J. Cand{\`e}s, X.~Li, and M.~Soltanolkotabi.
\newblock Phase retrieval from coded diffraction patterns.
\newblock {\em Applied and Computational Harmonic Analysis}, 39(2):277--299,
  Sept. 2015.

\bibitem{Candes2015Wirtinger}
E.~J. Cand{\`e}s, X.~Li, and M.~Soltanolkotabi.
\newblock Phase retrieval via {W}irtinger flow: {T}heory and algorithms.
\newblock {\em IEEE Transactions on Information Theory}, 61(4):1985--2007,
  April 2015.

\bibitem{candes2013phaselift}
E.~J. Cand{\`e}s, T.~Strohmer, and V.~Voroninski.
\newblock Phase{L}ift: Exact and stable signal recovery from magnitude
  measurements via convex programming.
\newblock {\em Commun. Pure Appl. Math.}, 66(8):1241--1274, 2013.

\bibitem{Chapman1996}
H.~N. Chapman.
\newblock Phase-retrieval x-ray microscopy by {W}igner-distribution
  deconvolution.
\newblock {\em Ultramicroscopy}, 66(3):153 -- 172, 1996.

\bibitem{clark2013ultrafast}
J.~Clark, L.~Beitra, G.~Xiong, A.~Higginbotham, D.~Fritz, H.~Lemke, D.~Zhu,
  M.~Chollet, G.~Williams, and M.~Messerschmidt.
\newblock Ultrafast three-dimensional imaging of lattice dynamics in individual
  gold nanocrystals.
\newblock {\em Science}, 341(6141):56--59, 2013.

\bibitem{Corbett2006}
J.~Corbett.
\newblock The {P}auli problem, state reconstruction and quantum-real numbers.
\newblock {\em Reports on Mathematical Physics}, 57(1):53--68, 2006.

\bibitem{daSilva:15}
J.~C. da~Silva and A.~Menzel.
\newblock Elementary signals in ptychography.
\newblock {\em Opt. Express}, 23(26):33812--33821, Dec 2015.

\bibitem{Fienup1987}
C.~Fienup and J.~Dainty.
\newblock Phase retrieval and image reconstruction for astronomy.
\newblock {\em Image Recovery: Theory and Application}, pages 231--275, 1987.

\bibitem{Fienup_1978_AltProj}
J.~R. Fienup.
\newblock Reconstruction of an object from the modulus of its {F}ourier
  transform.
\newblock {\em Opt. Lett.}, 3:27--29, 1978.

\bibitem{fienup1982phase}
J.~R. Fienup.
\newblock Phase retrieval algorithms: a comparison.
\newblock {\em Applied optics}, 21(15):2758--2769, 1982.

\bibitem{GSaxtonAltProj}
R.~Gerchberg and W.~Saxton.
\newblock {A Practical Algorithm for the Determination of Phase from Image and
  Diffraction Plane Pictures}.
\newblock {\em Optik}, 35:237--246, 1972.

\bibitem{cvxpaper}
M.~Grant and S.~Boyd.
\newblock Graph implementations for nonsmooth convex programs.
\newblock In V.~Blondel, S.~Boyd, and H.~Kimura, editors, {\em Recent Advances
  in Learning and Control}, Lecture Notes in Control and Information Sciences,
  pages 95--110. Springer--Verlag Limited, 2008.
\newblock \url{http://stanford.edu/~boyd/graph_dcp.html}.

\bibitem{cvx}
M.~Grant and S.~Boyd.
\newblock {CVX}: Matlab software for disciplined convex programming, version
  2.1.
\newblock \url{http://cvxr.com/cvx}, Mar. 2014.

\bibitem{griffin1984signal}
D.~Griffin and J.~Lim.
\newblock Signal estimation from modified short-time fourier transform.
\newblock {\em IEEE Transactions on Acoustics, Speech, and Signal Processing},
  32(2):236--243, 1984.

\bibitem{gross2015improved}
D.~Gross, F.~Krahmer, and R.~Kueng.
\newblock Improved recovery guarantees for phase retrieval from coded
  diffraction patterns.
\newblock {\em Applied and Computational Harmonic Analysis}, 42:37 -- 64, 2017.

\bibitem{hansen2000lcurve}
P.~C. Hansen.
\newblock The {L}-{C}urve and its use in the numerical treatment of inverse
  problems.
\newblock In {\em in Computational Inverse Problems in Electrocardiology, ed.
  P. Johnston, Advances in Computational Bioengineering}, pages 119--142. WIT
  Press, 2000.

\bibitem{hansen2005rank}
P.~C. Hansen.
\newblock {\em Rank-deficient and discrete ill-posed problems: {N}umerical
  aspects of linear inversion}, volume~4.
\newblock SIAM, 2005.

\bibitem{Harrison1993}
R.~W. Harrison.
\newblock Phase problem in crystallography.
\newblock {\em JOSA A}, 10(5):1046--1055, 1993.

\bibitem{iwen2017phase}
M.~Iwen, B.~Preskitt, R.~Saab, and A.~Viswanathan.
\newblock Phase retrieval from local measurements in two dimensions.
\newblock In {\em Wavelets and Sparsity XVII}, volume 10394, page 103940X.
  International Society for Optics and Photonics, 2017.

\bibitem{iwen2017robust}
M.~Iwen, A.~Viswanathan, and Y.~Wang.
\newblock Robust sparse phase retrieval made easy.
\newblock {\em Applied and Computational Harmonic Analysis}, 42(1):135--142,
  2017.

\bibitem{bitbucket_BlockPR}
M.~Iwen, Y.~Wang, and A.~Viswanathan.
\newblock {BlockPR}: Matlab software for phase retrieval using block circulant
  measurement constructions and angular synchronization, version 2.0.
\newblock \url{https://bitbucket.org/charms/blockpr}, Apr. 2016.

\bibitem{iwen2018lower}
M.~A. Iwen, S.~Merhi, and M.~Perlmutter.
\newblock Lower {L}ipschitz bounds for phase retrieval from locally supported
  measurements.
\newblock {\em Applied and Computational Harmonic Analysis}, 2019.

\bibitem{iwen2018phase}
M.~A. Iwen, B.~Preskitt, R.~Saab, and A.~Viswanathan.
\newblock Phase retrieval from local measurements: Improved robustness via
  eigenvector-based angular synchronization.
\newblock {\em Applied and Computational Harmonic Analysis}, 2018.

\bibitem{iwen2016fast}
M.~A. Iwen, A.~Viswanathan, and Y.~Wang.
\newblock Fast phase retrieval from local correlation measurements.
\newblock {\em SIAM J. Imaging Sci.}, 9(4):1655--1688, 2016.

\bibitem{Melnyk:19}
O.~Melnyk, F.~Filbir, and F.~Krahmer.
\newblock Phase retrieval from local correlation measurements with fixed shift
  length.
\newblock In {\em Imaging and Applied Optics 2019 (COSI, IS, MATH, pcAOP)},
  page MTu4D.3. Optical Society of America, 2019.

\bibitem{merhi2017recovery}
S.~Merhi, A.~Viswanathan, and M.~Iwen.
\newblock Recovery of compactly supported functions from spectrogram
  measurements via lifting.
\newblock In {\em Sampling Theory and Applications (SampTA), 2017 International
  Conference on}, pages 538--542. IEEE, 2017.

\bibitem{Merhi2019}
S.~Merhi, R.~Zhang, M.~A. Iwen, and A.~Christlieb.
\newblock A new class of fully discrete sparse {F}ourier transforms: Faster
  stable implementations with guarantees.
\newblock {\em Journal of Fourier Analysis and Applications}, 25(3):751--784,
  Jun 2019.

\bibitem{pfander2019robust}
G.~E. Pfander and P.~Salanevich.
\newblock Robust phase retrieval algorithm for time-frequency structured
  measurements.
\newblock {\em SIAM Journal on Imaging Sciences}, 12(2):736--761, 2019.

\bibitem{preskitt2018phase}
B.~P. {Preskitt}.
\newblock {\em {Phase Retrieval from Locally Supported Measurements}}.
\newblock PhD thesis, University of California, San Diego, 2018.

\bibitem{Rodenburg2008}
J.~Rodenburg.
\newblock Ptychography and related diffractive imaging methods.
\newblock {\em Advances in Imaging and Electron Physics}, 150:87--184, 2008.

\bibitem{Rodenburg1992}
J.~M. Rodenburg and R.~H.~T. Bates.
\newblock The theory of super-resolution electron microscopy via
  wigner-distribution deconvolution.
\newblock {\em Philosophical Transactions: Physical Sciences and Engineering},
  339(1655):521--553, 1992.

\bibitem{salanevich2015polarization}
P.~Salanevich and G.~E. Pfander.
\newblock Polarization based phase retrieval for time-frequency structured
  measurements.
\newblock In {\em Proc. 2015 Int. Conf. Sampling Theory and Applications
  (SampTA)}, pages 187--191, 2015.

\bibitem{Segal2013}
B.~Segal and M.~A. Iwen.
\newblock Improved sparse fourier approximation results: faster implementations
  and stronger guarantees.
\newblock {\em Numerical Algorithms}, 63(2):239--263, Jun 2013.

\bibitem{seibert2011single}
M.~M. Seibert, T.~Ekeberg, F.~R. Maia, M.~Svenda, J.~Andreasson,
  O.~J{\"o}nsson, D.~Odi{\'c}, B.~Iwan, A.~Rocker, and D.~Westphal.
\newblock Single mimivirus particles intercepted and imaged with an x-ray
  laser.
\newblock {\em Nature}, 470(7332):78--81, 2011.

\bibitem{singer2011angular}
A.~Singer.
\newblock Angular synchronization by eigenvectors and semidefinite programming.
\newblock {\em Applied and computational harmonic analysis}, 30(1):20--36,
  2011.

\bibitem{van2015imaging}
G.~Van Der~Schot, M.~Svenda, F.~R. Maia, M.~Hantke, D.~P. DePonte, M.~M.
  Seibert, A.~Aquila, J.~Schulz, R.~Kirian, and M.~Liang.
\newblock Imaging single cells in a beam of live cyanobacteria with an x-ray
  laser.
\newblock {\em Nature communications}, 6, 2015.

\bibitem{viswanathan2015fast}
A.~Viswanathan and M.~Iwen.
\newblock Fast angular synchronization for phase retrieval via incomplete
  information.
\newblock In {\em SPIE Optical Engineering+ Applications}, pages
  959718--959718. International Society for Optics and Photonics, 2015.

\bibitem{walther1963question}
A.~Walther.
\newblock The question of phase retrieval in optics.
\newblock {\em Optica Acta: International Journal of Optics}, 10(1):41--49,
  1963.

\end{thebibliography}

\end{document}